\documentclass[11pt]{amsart}
\usepackage{amssymb}
\usepackage{bm}
\usepackage[centertags]{amsmath}
\usepackage{amsfonts}
\usepackage{amsthm}
\usepackage{mathrsfs}
\usepackage[normalem]{ulem}
\usepackage{comment}
\usepackage{enumitem}
\usepackage{hyperref}
\usepackage[noabbrev]{cleveref}

\usepackage{xifthen}
\usepackage{ifthenx}
\usepackage{xargs}

\usepackage[table,svgnames,x11names]{xcolor}
\usepackage[colorinlistoftodos,prependcaption,textsize=tiny,textwidth=3.5cm]{todonotes}
\presetkeys{todonotes}{fancyline}{}
\newcommandx{\change}[2][1=]{\todo[linecolor=blue,backgroundcolor=blue!25,bordercolor=blue,#1]{#2}}
\newcommandx{\changein}[2][1=]{\change[inline, caption={change}, #1]{%
    \begin{minipage}{\textwidth-20pt}#2\end{minipage}}}
\newcommandx{\todoin}[2][1=]{\todo[inline, caption={todo}, #1]{%
    \begin{minipage}{\textwidth-20pt}#2\end{minipage}}}

\newcommandx{\remove}[2][1=]{\todo[linecolor=Plum,backgroundcolor=Plum!25,bordercolor=Plum,#1]{#2}}
\newcommandx{\removein}[2][1=]{\remove[inline, caption={todo}, #1]{%
    \begin{minipage}{\textwidth-20pt}#2\end{minipage}}}

\makeatletter
\def\l@subsection{\@tocline{2}{0pt}{2.5pc}{2.5pc}{}}%
\makeatother
\makeatletter
\def\l@subsubsection{\@tocline{3}{0pt}{5pc}{5pc}{}}%
\makeatother

\usepackage{hyperref}
\hypersetup{pdftex,colorlinks=true,allcolors=black}
\usepackage{hypcap}

\DeclareFontFamily{OT1}{pzc}{}
\DeclareFontShape{OT1}{pzc}{m}{it}{<-> s * [1.10] pzcmi7t}{}
\DeclareMathAlphabet{\mathpzc}{OT1}{pzc}{m}{it}

\theoremstyle{definition}
\newtheorem{thm}{Theorem}[section]
\newtheorem{dfn}[thm]{Definition}
\newtheorem{lem}[thm]{Lemma}
\newtheorem{prop}[thm]{Proposition}
\newtheorem{cor}[thm]{Corollary}
\newtheorem{rem}[thm]{Remark}
\newtheorem{exa}[thm]{Example}

\newtheorem{ntt}[thm]{Notation}

\newtheorem*{defn*}{Definition}

\newtheorem*{thm*}{Theorem}

\newtheorem*{cor*}{Corollary}

\newtheorem*{prp*}{Proposition}

\newtheorem{problem}{Problem}

\newtheorem{clm}{Claim}

\newtheorem{thmA}{Theorem}

\AddToHook{env/dfn/begin}{\crefalias{thm}{dfn}}
\AddToHook{env/lem/begin}{\crefalias{thm}{lem}}
\AddToHook{env/prop/begin}{\crefalias{thm}{prop}}
\AddToHook{env/cor/begin}{\crefalias{thm}{cor}}
\AddToHook{env/rem/begin}{\crefalias{thm}{rem}}
\AddToHook{env/exa/begin}{\crefalias{thm}{exa}}
\AddToHook{env/ass/begin}{\crefalias{thm}{ass}}
\AddToHook{env/ntt/begin}{\crefalias{thm}{ntt}}
\AddToHook{env/inst/begin}{\crefalias{thm}{inst}}
\AddToHook{env/obj/begin}{\crefalias{thm}{obj}}
\AddToHook{env/conj/begin}{\crefalias{thm}{conj}}
\AddToHook{env/question/begin}{\crefalias{thm}{question}}
\AddToHook{env/problem/begin}{\crefalias{thm}{problem}}
\AddToHook{env/rmks/begin}{\crefalias{thm}{rmks}}
\AddToHook{env/clm/begin}{\crefalias{thm}{clm}}

\newcommand{\N}{\mathbb{N}}

\newcommand{\trn}[1]{{\left\vert\kern-0.25ex\left\vert\kern-0.25ex\left\vert #1 
    \right\vert\kern-0.25ex\right\vert\kern-0.25ex\right\vert}}

\newcommand{\trnsmall}[1]{{\vert\kern-0.25ex\vert\kern-0.25ex\vert #1 
    \vert\kern-0.25ex\vert\kern-0.25ex\vert}}

\DeclareMathOperator{\supp}{supp}

\long\def\symbolfootnote[#1]#2{\begingroup%
\def\thefootnote{\fnsymbol{footnote}}\footnote[#1]{#2}\endgroup}

\makeatletter
\@namedef{subjclassname@2020}{%
  \textup{2020} Mathematics Subject Classification}
\makeatother

\allowdisplaybreaks

\begin{document}

\title[]{Bourgain-Rosenthal-Schechtman spaces in rearrangement-invariant and Hardy spaces}

\author[K. Konstantos]{Konstantinos Konstantos}

\address{Konstantinos Konstantos, Department of Mathematics and Statistics, York University, 4700 Keele Street, Toronto, Ontario, M3J 1P3, Canada}

\email{konstantoskostas95@gmail.com}

\author[P. Motakis]{Pavlos Motakis}

\address{Pavlos Motakis, Department of Mathematics and Statistics, York University, 4700 Keele Street, Toronto, Ontario, M3J 1P3, Canada}

\email{pmotakis@yorku.ca}

\author[Th. Speckhofer]{Thomas Speckhofer}

\address{Thomas Speckhofer, Department of Mathematics, Texas A\&M University, College Station, TX~77843, USA}

\email{speckhofer@tamu.edu}

\thanks{The first and second authors were supported by NSERC Grant RGPIN-2021-03639. The third author was supported by NSERC Grant RGPIN-2021-03639 and subsequently by the Austrian Science Fund (FWF), project 10.55776/J5020, under which this work was completed.}


\keywords{\hspace{-0.3pt}Bourgain-Rosenthal-Schechtman\hspace{-0.5pt} spaces, rearrangement-invariant\hspace{-0.5pt} spaces, Hardy\hspace{-0.5pt} spaces}


\begin{abstract}
In 1981, Bourgain, Rosenthal, and Schechtman constructed uncountably many pairwise non-isomorphic complemented subspaces of $L_p$ for $1 < p < \infty$, $p \neq 2$. We extend their construction to rearrangement-invariant spaces and Hardy spaces. For any separable rearrangement-invariant space $X$ with non-trivial Boyd indices, we show that these subspaces are complemented, previously obtained by Ghawadrah via a different argument, and establish the following dichotomy: Either one can extract an uncountable family of pairwise non-isomorphic subspaces, or they are all isomorphic to $X$. To this end, we introduce an ordinal index that detects whether $X$ embeds complementably into an arbitrary separable Banach space. On the other hand, we give sufficient conditions which imply that in certain spaces close to $L_1$ and $H^1$, these subspaces fail to be complemented.
\end{abstract}

\vspace*{-0.1cm}
\maketitle


\vspace{-0.2cm}
\noindent\textbf{AI usage statement.}
This project did not make substantial use of large language models. ChatGPT was used for the creation of figures, for editorial and bibliographical purposes, such as assistance in formulating the text and locating references, and for proofreading the finished manuscript (which resulted only in minor corrections). The mathematical contributions of ChatGPT were limited to two specific tasks. First, it assisted with adapting a known argument to prove \Cref{BRS orthocomplemented or not at all}. Second, it was used to identify an example of an Orlicz space satisfying a condition established by the authors (see \Cref{rem:chat-gpt-orlicz}). In both cases, the resulting arguments and mathematical content were carefully checked by the authors.

\tableofcontents

\section{Introduction}

In~1981, Bourgain, Rosenthal, and Schechtman~\cite{bourgain:rosenthal:schechtman:1981} proved the existence of uncountably many pairwise non-isomorphic complemented subspaces of $L_p$ for $1 < p < \infty$, $p\ne 2$. This was achieved by inductively defining a transfinite collection of complemented subspaces of $L_p$, denoted by $R_\alpha^p$, $0 \le  \alpha < \omega_1$, alternately taking disjoint and independent sums of subspaces of $L_p$. To prove that it is indeed possible to extract uncountably many non-isomorphic spaces from this collection, they also introduced an ordinal index which detects whether $L_p$ embeds isomorphically into a given separable Banach space.
Another central part of their proof is a result on complemented embeddings of $L_p$ into spaces with an unconditional Schauder decomposition.
 The Bourgain-Rosenthal-Schechtman spaces also have immediate applications in the theory of operator ideals---since each of the spaces $R_\alpha^p$ with infinite $\alpha$ is isomorphic to its square, they yield uncountably many closed two-sided ideals in the algebra $\mathcal{B}(L_p)$ of bounded linear operators on~$L_p$ (see, e.g., \cite[Sections~2.2 and~5.3]{pietsch:1978}). 

The isomorphic structure of the Bourgain-Rosenthal-Schechtman spaces was further analyzed in the deep work of Alspach~\cite{alspach:1999} who, in particular, determined their isomorphism classes. 
Later, in~2016, Ghawadrah~\cite{ghawadrah:2016} extended the main result of Bourgain, Rosenthal, and Schechtman to reflexive Orlicz function spaces $L_{\Phi}[0,1]$. Such a space has uncountably many pairwise non-isomorphic complemented subspaces as soon as it is not isomorphic to $L_2$. To this end, Ghawadrah used interpolation results, descriptive set theory, and some structural results about modular sequence spaces.

More recently, the first and second authors of the present paper have proved factorization results for operators on the first infinite-dimensional Bourgain-Rosenthal-Schechtman space in~$L_p$~\cite{konstantos:motakis:2025}, and the first and third authors have extended this result to more general underlying spaces~\cite{MR5094444}. Furthermore, the first two authors in \cite{konstantos:motakis:brs} defined explicit unconditional finite-dimensional decompositions on all Bourgain-Rosenthal-Schechtman spaces $R_\alpha^p$, supporting a broad class of concrete distributional embeddings between them. This provides a general framework for operator reduction yielding specific operator-factorization results throughout the hierarchy. This framework was also used in~\cite{konstantos2026shaipropertybourgainrosenthalschechtmanspaces} to prove the SHAI property for the Bourgain-Rosenthal-Schechtman spaces.

In this paper, we further extend the work of Bourgain, Rosenthal, and Schechtman, as well as the results of Ghawadrah. In our results, the underlying space is either a \emph{Haar system space} $X$ or its associated \emph{Hardy space} $H_X$. Haar system spaces, first introduced in~\cite{lechner:motakis:mueller:schlumprecht:2022}, are essentially separable rearrangement-invariant function spaces on the unit interval, defined by a minimal collection of conditions. The Hardy space $H_X$ associated with a Haar system space $X$ consists of those functions whose martingale square function belongs to $X$. These classes of spaces were analyzed in \cite{lechner:speckhofer:2025} (see also \cite{speckhofer:2025:thesis} for further background). Note that for $X = L_1$, the corresponding Hardy space is the classical dyadic Hardy space $H^1$. For any Haar system space $X$ and $Y = X^0$ or $H_X$ (where $X^0$ denotes the mean-zero functions in~$X$), we define the Bourgain-Rosenthal-Schechtman spaces $R_\alpha^Y$, $0\le \alpha < \omega_1$, by transfinite recursion, using a language that is compatible with the martingale square function and thus also well suited to the Hardy case: For each successor ordinal $\alpha$, we take $R_\alpha^Y$ to be the sum of two compressed copies of its predecessor space, supported on $[0,\frac{1}{2})$ and $[\frac{1}{2},1)$, respectively. If $\alpha$ is a limit ordinal, we define $R_\alpha^Y$ to be a certain independent sum of all previously constructed spaces $R_\beta^Y$, each canonically realized as a space of $\sigma(\{ r_n : n\in N_\beta \})$-measurable functions, where  $(r_n)_{n\in N_\beta}$ are pairwise disjoint subsequences of the Rademacher sequence (see \Cref{types of sums} for the precise definition).
Similar to \cite{bourgain:rosenthal:schechtman:1981}, we also have a more general construction of subspaces of $Y$: For each tree $\mathcal{T}$ of finite subsets of $\mathbb{N}$ satisfying a property we call \emph{disjointly branching}, we define the tree space $R_{\mathcal{T}}^Y$ as the closed linear span of the subspaces of $\sigma(\{ r_n : n\in F \})$-measurable functions in $Y$, for $F\in \mathcal{T}$.
Each Bourgain-Rosenthal-Schechtman space $R_\alpha^Y$ can be alternatively realized in this way.
Under the assumption that $X$ has non-trivial Boyd indices, we employ Stein's martingale inequality \cite[Proposition 9.5]{johnson:maurey:schechtman:tzafriri:1979} to prove that all these tree spaces $R_{\mathcal{T}}^Y$, and hence also the Bourgain-Rosenthal-Schechtman spaces, are complemented in~$Y$ via an orthogonal projection (see~\Cref{prop:complementation}).

Our main result on Haar system spaces is the following dichotomy:

\begin{thmA}\label{thmA:alternatives}
  Let $X$ be a Haar system space with the Fatou property and non-trivial Boyd indices, and let $Y = X^0$.
  Then one of the following exclusive alternatives holds:
  \begin{enumerate}[label=(\alph*),topsep=5pt]
    \item\label{thmA:alternatives:a} There exists a strictly increasing transfinite sequence of countable ordinals $(\alpha_\xi)_{\xi< \omega_1}$ such that the spaces $R_{\alpha_\xi}^{Y}$, $\xi < \omega_1$, are pairwise non-isomorphic.
    \item\label{thmA:alternatives:b} For every ordinal $\omega \le \alpha < \omega_1$, the space $R_\alpha^{Y}$ is isomorphic to $Y$.
  \end{enumerate}
\end{thmA}

This result will be proved in \Cref{thm:alternatives-ri}.
In particular, alternative~\ref{thmA:alternatives:a} implies that the algebra $\mathcal{B}(Y)$ has uncountably many closed ideals.

If $X = L_p$ for some $1 < p < \infty$, or if $X$ is a reflexive Orlicz function space on $[0,1]$, then the conclusion of \Cref{thmA:alternatives} follows from the work of Bourgain, Rosenthal, and Schechtman~\cite{bourgain:rosenthal:schechtman:1981}, and Ghawadrah~\cite{ghawadrah:2016}, respectively. Note that we cannot, in general, rule out alternative~\ref{thmA:alternatives:b}. In contrast, the authors of~\cite{bourgain:rosenthal:schechtman:1981} and~\cite{ghawadrah:2016} were able to prove that in their respective classes of spaces, alternative~\ref{thmA:alternatives:a} holds, unless the underlying space is isomorphic to~$L_2$. In the first case, this was achieved using
properties specific to $L_p$, in particular, the interplay between its subspaces $\ell_2\oplus \ell_p$ and the Rosenthal space $X_p$. The second proof~\cite{ghawadrah:2016} uses a characterization of Hilbert spaces (among rearrangement-invariant spaces) as subspaces of modular sequence spaces.
It would be of interest to identify properties of a general Haar system space $X$ that determine which alternative in the dichotomy holds. In fact, we do not know if there are any examples other than $L_2$ for which~\ref{thmA:alternatives:b} holds.

Note that $X^0 = H_X$ if and only if $X$ has non-trivial Boyd indices (see \cite[Theorem~1.d.6 and Theorem~2.c.6]{lindenstrauss:tzafriri:1979:partII}).
In the case when $X$ has a trivial Boyd index, we have the same dichotomy as in \Cref{thmA:alternatives} for the space $Y = H_X$, assuming that the orthogonal projections onto certain realizations of the Bourgain-Rosenthal-Schechtman spaces are already bounded  (see \Cref{thm:alternatives-hardy}). However, we do not know if the complementation of the Bourgain-Rosenthal-Schechtman spaces in $X^0$ or $H_X$ is in fact equivalent to $X$ having non-trivial Boyd indices.
On the other hand, we show that if $Y = L_1^0$ or $H^1$, or if $Y$ is a specific Orlicz function space close to $L_1$ (or the associated Hardy space close to $H^1$), then the space $R_\alpha^Y$ is in fact uncomplemented in $Y$ whenever $\alpha \ge \omega\cdot 2$ (see \Cref{thm:tree-spaces-uncomplemented}). This follows from the fact that the orthogonal projection onto $R_\alpha^Y$ is unbounded in~$Y$. For $Y = L_1^0$ or $H^1$, this is already alluded to in \cite[p.~223]{bourgain:rosenthal:schechtman:1981}, although no proof is provided there. We also provide general conditions on the fundamental function of a rearrangement-invariant space, which can be used to show that the aforementioned orthogonal projection is unbounded (see \Cref{prop:unboundedness}).
For $\omega \le \alpha < \omega\cdot 2$, the space $R_\alpha^Y$ may be complemented or uncomplemented in~$Y$, depending on the underlying space and choices made during the construction of $R_\alpha^Y$ (see \Cref{thm:tree-spaces-uncomplemented}). 

A new ingredient in our proof of \Cref{thmA:alternatives} is an ordinal $Y$-index that detects whether $Y$ embeds \emph{complementably} into a given separable Banach space~$W$; its value on $W$ is uncountable if and only if such an embedding exists.
By construction, for every countable ordinal $\alpha$, the $Y$-index of $R_\alpha^Y$ is at least $\alpha$.
The index is defined as the supremum of the heights of certain trees $\mathscr{T}(C)$, $C \ge 1$. Each tree $\mathscr{T}(C)$ consists of finite sequences of the form $\big((T_0,P_0),\ldots,(T_n,P_n)\big)$, with each $T_i$ being a $C$-embedding of the span of the first $i+1$ levels of the Haar system in $Y$ into $W$, and $P_i$ being a projection of norm at most $C$ onto its image. Each such sequence must satisfy certain compatibility conditions.
This index is a refinement of the original ordinal index of Bourgain, Rosenthal, and Schechtman~\cite{bourgain:rosenthal:schechtman:1981} and is necessary to circumvent assumptions such as $q$-concavity of~$Y$, which are used in~\cite{bourgain:rosenthal:schechtman:1981} and~\cite{ghawadrah:2016} to pass from subspaces of~$Y$ isomorphic to~$Y$ to \emph{complemented} subspaces isomorphic to~$Y$ (see \cite[Theorem~9.1]{johnson:maurey:schechtman:tzafriri:1979}).

Moreover, we extend the work of Bourgain, Rosenthal, and Schechtman on complemented embeddings of $L_p$ into spaces with unconditional Schauder decompositions \cite[Theorem~1.1]{bourgain:rosenthal:schechtman:1981}.
We prove that if the Haar system is unconditional in $Y = X^0$ or $H_X$, and $Y$ embeds complementably into a space with an unconditional Schauder decomposition, then there exists another complemented copy of $Y$ which is either contained in one of the decomposition's components or is spanned by a block copy of the Haar system relative to the decomposition.
This result was proved for rearrangement-invariant function spaces~$X$ with non-trivial Boyd indices by Ghawadrah \cite[Theorem~3.1]{ghawadrah:2016}, who used interpolation arguments
to translate the original proof in~\cite{bourgain:rosenthal:schechtman:1981} to this new setting. In contrast, while our proof still follows the general structure of \cite[Theorem~1.1]{bourgain:rosenthal:schechtman:1981}, we incorporate recent tools from factorization theory~\cite{lechner:speckhofer:2025} that operate directly in the space $X$ itself, thus avoiding the passage to $X(\ell_2)$ used in previous proofs, as well as interpolation. In particular, our approach also applies to the Hardy case, thus yielding a new result in this setting (see \Cref{thm:1.1}).

Our proof of \Cref{thmA:alternatives} can be outlined as follows. We distinguish two cases: If, for every $\alpha$, the $Y$-index of $R_\alpha^Y$ is countable, then we easily extract a cofinal subfamily of pairwise non-isomorphic spaces from $(R_\beta^{Y})_{\beta < \omega_1}$. On the other hand, if for some $\alpha$, the $Y$-index of $R_\alpha^Y$ is uncountable, then we apply the above unconditional Schauder decomposition theorem to conclude that this is in fact true for all spaces $R_\beta^Y$, $\omega \le \beta < \omega_1$. The uncountability of the $Y$-index guarantees the complementation of $Y$ in all these spaces. Using Pe{\l}czy\'nski's decomposition method for spaces isomorphic to their squares, we deduce that all spaces $R_\beta^Y$, $\omega \le \beta < \omega_1$, are isomorphic to~$Y$.

\section{Distributional embeddings and Haar systems}

We recall the necessary measure-theoretic notation surrounding distributional embeddings and the notation for the Rademacher sequence and the Haar system.

\begin{ntt}
\label{measure theory notation}
For a measurable subset $A$ of $\mathbb{R}$, we denote by $|A|$ its Lebesgue measure. We denote by $L_0$ the space of equivalence classes of scalar-valued measurable functions on the unit interval. All vector spaces are assumed to be real. Equipped with convergence in measure, $L_0$ is a Polish topological vector space, induced, e.g., by the metric $d(f,g) = \|(f-g)(1+|f-g|)^{-1}\|_1$. We denote $L_1 = \{f\in L_0:\;\int|f|dx<\infty\}$ and $L_1^0 = \{f\in L_1: \int f dx = 0\}$. For $f\in L_0$, the distribution of $f$ is the probability measure on the Borel subsets of $\mathbb{R}$ such that $\mathrm{dist}(f)(A) = |[f\in A]|$. For a countably indexed collection $(f_i)_{i\in I}$ in $L_0$, its joint distribution is the probability measure on the Borel $\sigma$-algebra $\mathcal{B}(\mathbb{R}^I)$ given by $\mathrm{dist}((f_i)_{i\in I})(A) = |[(f_i)_{i\in I}\in A]|$. For a vector subspace $W$ of $L_0$, a linear map $T:W\to L_0$ is called a distributional embedding if, for every $f\in W$, $\mathrm{dist}(Tf) = \mathrm{dist}(f)$. Every distributional embedding $T\colon W \to L_0$ extends uniquely to a distributional embedding defined on the closure of~$W$ with respect to convergence in measure (since convergence in measure implies convergence in distribution). For such a map, letting $Z = T(W)$, we then write $Z\equiv^\mathrm{dist} W$ and say that $T$ is a distributional isomorphism onto $Z$. Two commonly indexed countable collections $(f_i)_{i\in I}$ and $(g_i)_{i\in I}$ in $L_0$ are called distributionally equivalent if they have the same joint distribution. By a standard probability theory fact based on the multivariable characteristic function, the previous statement is equivalent to the map $f_i\mapsto g_i$ extending to a distributional embedding from $\langle\{f_i:i\in I\}\rangle$ onto $\langle\{g_i:i\in I\}\rangle$.
\end{ntt}

For a $\sigma$-subalgebra $\mathcal{A}$ of the Lebesgue $\sigma$-algebra $\mathcal{L}([0,1])$, we denote by $L_0(\mathcal{A})$ the collection of all $f\in L_0$ that are essentially $\mathcal{A}$-measurable, i.e., they contain an $\mathcal{A}$-measurable representative. In reality, $L_0(\mathcal{A})$ depends only on the measure algebra of $\mathcal{A}$, that is, the image of $\mathcal{A}$ in the quotient of $\mathcal{L}([0,1])$ modulo null sets. The subspace $L_0(\mathcal{A})$ of $L_0$ is closed in the topology of convergence in measure. For a vector subspace $W$ of $L_0$ we denote by $\sigma(W)$ the $\sigma$-subalgebra of $\mathcal{L}([0,1])$ generated by all representatives of members of $W$. 

The following is well-known; e.g., it is mentioned in \cite[page 212]{bourgain:rosenthal:schechtman:1981}. We sketch a proof because we could not locate a reference to one.

\begin{prop}
\label{distroembedding extension}
Let $T:W\to L_0$ be a distributional embedding on a vector subspace $W$ of $L_0$. Then $T$ extends uniquely to an onto distributional embedding $\tilde T:L_0(\sigma(W))\to L_0(\sigma(T[W]))$. Furthermore, $\tilde T$ is an algebra homomorphism.
\end{prop}

\begin{proof}
We will use the fact that the $\sigma$-algebra generated by a non-empty set of scalar-valued functions $\mathcal{F}$ on a common domain coincides with the collection
\begin{equation}
\label{distroembedding extension eq}
\Big\{\big[(f_n)_{n=1}^\infty\in A\big]:(f_n)_{n=1}^\infty\in\mathcal{F}^\mathbb{N}\text{ and }A\in\mathcal{B}(\mathbb{R}^\mathbb{N})\Big\}.
\end{equation}
In our case, $\mathcal{F}$ is the collection of all representatives from members of $W$. For any $(f_n)_{n=1}^\infty$, $(g_n)_{n=1}^\infty$ in $W$ and $A, B\in\mathcal{B}(\mathbb{R}^\mathbb{N})$, because $T$ preserves joint distributions,
\begin{align*}
\big|[(f_n)_{n=1}^\infty\in A]\setminus [(g_n)_{n=1}^\infty\in B]\big| &= \big|[((f_n)_{n=1}^\infty, (g_n)_{n=1}^\infty) \in A\times B^c]\big|\\
&= \big|[((Tf_n)_{n=1}^\infty, (Tg_n)_{n=1}^\infty) \in A\times B^c]\big|\\
&= \big|[(Tf_n)_{n=1}^\infty\in A]\setminus [(Tg_n)_{n=1}^\infty\in B]\big|.
\end{align*}
This implies that the map $\sigma(W)\ni \big[(f_n)_{n=1}^\infty\in A\big] = C\mapsto \varphi(C) = \big[(Tf_n)_{n=1}^\infty\in A\big]\in\sigma(T[W])$ is a well-defined measure algebra isomorphism, and it thus induces an onto distributional embedding $\tilde{T}\colon L_0(\sigma(W))\to L_0(\sigma(T[W]))$, via $\tilde T\chi_{C} = \chi_{\varphi(C)}$ ($\tilde T$ is first defined on simple functions, where it is an isometry with respect to the standard Polish metric of $L_0$ mentioned in \Cref{measure theory notation}, and then extended to $L_0(\sigma(W))$). It is easy to see that such a $\tilde T$ is a homomorphism and extends~$T$. Uniqueness is verified as follows: if $S\colon L_0(\sigma(W))\to L_0$ is a distributional embedding extending $T$ and $[(f_n)_{n=1}^\infty\in A\big]$ is as in \eqref{distroembedding extension eq}, then, because $S$ preserves joint distributions,
\begin{align*}
\big|\big[S\chi_{[(f_n)_{n=1}^\infty\in A]} \neq \tilde T\chi_{[(f_n)_{n=1}^\infty\in A]}\big]\big| &= \big|\big[S\chi_{[(f_n)_{n=1}^\infty\in A]}\in\{1\}\big]\triangle\big[(Tf_n)_{n=1}^\infty\in A\big]\big|\\
&= \big|\big[S\chi_{[(f_n)_{n=1}^\infty\in A]}\in\{1\}\big]\triangle\big[(Sf_n)_{n=1}^\infty\in A\big]\big|\\
& = \big|\big[\chi_{[(f_n)_{n=1}^\infty\in A]}\in\{1\}\big]\triangle\big[(f_n)_{n=1}^\infty\in A\big]\big| = 0.\qedhere
\end{align*}
\end{proof}

\begin{ntt}
\label{dyadic haar rademacher notation}
We establish the notation for dyadic intervals, the Haar system, and the Rademacher sequence.
\begin{enumerate}[label=(\alph*),]

\item For $n\in\mathbb{N}_0 = \mathbb{N}\cup \{ 0 \}$, we let
\begin{gather*}
\mathcal{D}^n = \Big\{\Big[\frac{i-1}{2^n},\frac{i}{2^n}\Big):1\leq i\leq 2^n\Big\},\quad \mathcal{D}_n = \cup_{k=0}^n\mathcal{D}^k,\quad \text{and}\quad \mathcal{D} = \cup_{k=0}^\infty\mathcal{D}^k.
\end{gather*}
We enumerate $\mathcal{D}$ according to the lexicographical order, defined by declaring $I<J$ whenever $|I|>|J|$, or $|I|=|J|$ and $\sup(I)\le \inf(J)$.
For any collection $\mathcal{B}$ of dyadic intervals, we write $\mathcal{B}^{*} = \cup_{I\in \mathcal{B}}I$.
        
\item For $n\in\mathbb{N}_0$ and $I = [(i-1)/2^n,i/2^n)\in\mathcal{D}^n$ we denote the left and right half of $I$ as follows:
\[I^+ = \Big[\frac{2i-2}{2^{n+1}},\frac{2i-1}{2^{n+1}}\Big)\text{ and } I^- = \Big[\frac{2i-1}{2^{n+1}},\frac{2i}{2^{n+1}}\Big).\]
For $n\in\mathbb{N}$ and $I\in\mathcal{D}^n$ we denote by $\pi(I)$ the unique member of $\mathcal{D}^{n-1}$ such that $I = \pi(I)^+$ or $\pi(I)^-$. In the former case, we define $\varepsilon(I) = 1$ and in the latter $\varepsilon(I) = -1$.

\item The standard Haar system $(h_I)_{I\in\mathcal{D}}$ is defined as follows: for $I\in\mathcal{D}$, let
\[h_I = \chi_{I^+} - \chi_{I^{-}}.\]

\item The standard Rademacher sequence $(r_n)_{n=0}^\infty$ is defined by letting, for $n\in\mathbb{N}_0$,
\[r_n = \sum_{I\in\mathcal{D}^n}h_I.\]

\end{enumerate}
\end{ntt}

\begin{ntt}
\label{notation of subsequences of the levels}
Fix a finite subset $F = \{k_0<\cdots<k_n\}$ of $\mathbb{N}_0$.
\begin{enumerate}[resume,label=(\alph*),]

\item\label{notation of subsequences of the levels a}
Let
\[\sigma_F = \sigma(r_{k_0},\ldots,r_{k_n})\]
and
\[Y_F = \big\{f\in L_1^0 : f \text{ is } \sigma_F\text{-measurable}\big\}.\]
We write
\[\mathbb{E}_F \colon L_1^0 \to Y_F\]
for the conditional expectation operator with respect to $\sigma_F$.

In the special case where $n\in\mathbb{N}_0$ and $F=\{0,1,\ldots,n\}$, we write $\sigma_n$, $Y_n$, and $\mathbb{E}_n$ for the corresponding objects.

\item\label{notation of subsequences of the levels b} We denote by $J_F:Y_n\to Y_F$ the unique distributional embedding given by extending the map $r_i \mapsto r_{k_i}$, $0\leq i\leq n$ (which exists by \Cref{distroembedding extension}). We then define the distributional copy of the Haar system $(b_I^F)_{I\in\mathcal{D}_n}$ by letting, for $I\in\mathcal{D}_n$,
\[b^F_I = J_Fh_I.\]
\end{enumerate}
The above are also defined for an infinite $F$ in the obvious way.
\end{ntt}

\begin{rem}
\label{expectations on subsets are a monoid}
For every $F,G\subset \mathbb{N}_0$, on $L_1^0$ we have $\mathbb{E}_F\mathbb{E}_G = \mathbb{E}_{F\cap G}$ (where $\mathbb{E}_{\emptyset}$ is the zero operator on~$L_1^0$). This is easily checked on finite products of terms of the Rademacher sequence $(r_n)_{n=0}^\infty$, i.e., the Walsh system, which has a dense linear span in $L_1^0$ with $\|\cdot\|_1$.
\end{rem}

The following states that, following \Cref{notation of subsequences of the levels}, $(b_I^F)_{I\in\mathcal{D}_n}$ is in $\langle\{h_I:I\in\cup_{i=0}^n\mathcal{D}^{k_i}\}\rangle$.

\begin{prop}
\label{constructing associated haar system in the more familiar way}
Let $F = \{k_0<\cdots<k_n\}$ be a finite subset of $\mathbb{N}_0$. Define $\mathcal{B}^F_{[0,1)} = \mathcal{D}^{k_0}$ and, for $1\leq i\leq n$ and $I\in\mathcal{D}^i$, let $\mathcal{B}^F_I = \{L\in\mathcal{D}^{k_i}: L\subset [b^F_{\pi(I)} = \varepsilon(I)]\}.$ Then, for $0\leq i\leq n$ and $I\in\mathcal{D}^i$,
\begin{equation}\label{eq:22}
b^F_I = \sum_{L\in \mathcal{B}^F_I}h_L,
\end{equation}
and, in particular, $b_I^F$ is a linear combination of $(h_L)_{L\in\mathcal{D}^{k_i}}$.
\end{prop}

\begin{proof}
  For $1\leq i\leq n$ and $I\in\mathcal{D}^i$, $h_I = r_i\chi_{[h_{\pi(I)}=\varepsilon(I)]}$. Using that the extension of~$J_F$ from \Cref{distroembedding extension} is a homomorphism,
  \begin{equation*}
    b_I^F = J_Fh_I =  r_{k_i}\chi_{[b^F_{\pi(I)}=\varepsilon(I)]} =\Big(\sum_{L\in\mathcal{D}^{k_i}}h_L\Big) \chi_{[b^F_{\pi(I)}=\varepsilon(I)]} = \sum_{L\in\mathcal{B}^F_I}h_L,
  \end{equation*}
where the last equality follows from $b_{\pi(I)}^F$ being $\sigma(\mathcal{D}^{k_i})$-measurable.
\end{proof}

For $F = \{ 1,4 \}$, the functions $b_I^F$ are illustrated in \Cref{fig:faithful-haar-system}.

\begin{figure}[t]
  \centering \includegraphics[width=0.65\textwidth]{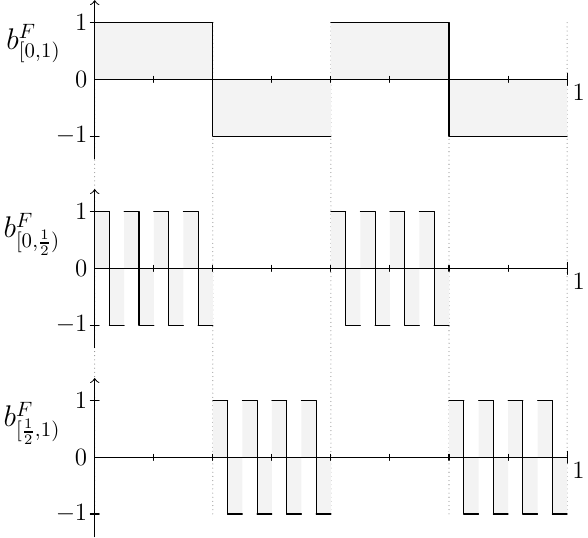}
  \caption{The functions $b_I^F$}
  \label{fig:faithful-haar-system}
\end{figure}

\begin{rem}
  A system $(b_I^F)_{I\in \mathcal{D}_n}$ of the form~\eqref{eq:22} is a special case of a so-called \emph{faithful Haar system}. This notion will be used extensively in \Cref{sec:complemented-embeddings}.
\end{rem}
\begin{rem}
\label{haar basis formulas}
Let $F = \{k_0<\cdots<k_n\}$ be a finite subset of $\mathbb{N}_0$.
\begin{enumerate}[label=(\alph*),]

\item\label{haar basis formulas a} For $g\in L_1^0$,
\[\mathbb{E}_Fg = \sum_{I\in\mathcal{D}_n}|I|^{-1}\langle b_I^F,g\rangle b_I^F.\]

\item\label{haar basis formulas b} For $g\in Y_n$,
\[J_Fg=  \sum_{I\in\mathcal{D}_n}|I|^{-1}\langle h_I,g\rangle b_I^{F}.\]

\end{enumerate}
\end{rem}

\begin{ntt}
  In the following, we will use the symbols $\lesssim$,$\gtrsim$, and $\asymp$ to suppress constants in inequalities: $a\lesssim b$ and $b\gtrsim a$ both mean that $a \le  C b$ for a constant $C > 0$, and $a\asymp b$ means that $a\lesssim b$ and $a \gtrsim b$.
\end{ntt}

\section{Haar system spaces and dyadic Hardy spaces}
We define Haar system spaces and explain the precise relationship between this definition and the various classical definitions of rearrangement-invariant spaces. We then define the associated Hardy spaces and introduce the notion of a square function-distributional embedding between their subspaces, which serves in this context as the analogue of a distributional embedding between subspaces of a Haar system space. As we point out, some of the strong properties of distributional embeddings do not carry over to the square-function analogue. Specifically, a square function-distributional copy of a complemented subspace of a Hardy space need not itself be complemented (see \Cref{rem:hardy}~\ref{rem:hardy:b}).

\begin{dfn}
\label{haar system space}
A Haar system space is the completion $X$ of the vector space $\langle\{ \chi_I:I\in\mathcal{D}\}\rangle = \langle\{\chi_{[0,1)}\}\cup\{h_I:I\in\mathcal{D}\}\rangle$ under a norm $\|\cdot\|_X$ satisfying:
\begin{enumerate}[label=(\alph*)]

\item $\|\chi_{[0,1)}\|_X = 1$ and

\item for $f,g\in \langle\{ \chi_I:I\in\mathcal{D}\}\rangle$ such that $\mathrm{dist}(|f|) = \mathrm{dist}(|g|)$, $\|f\|_X = \|g\|_X$.

\end{enumerate}
\end{dfn}

\begin{rem}
\label{comparing haar system to ri}
Haar system spaces whose norm is not equivalent to $\|\cdot\|_\infty$ provide a concise framework for separable rearrangement-invariant spaces (over the unit interval), with the Haar system highlighted as the main underlying structure. While there is no universally accepted definition of this class, a minimal formulation is to define a rearrangement-invariant space as an order ideal $X$ in $L_0$ containing $\chi_{[0,1)}$, equipped with a complete lattice norm $\|\cdot\|_X$ (i.e., $\|f\|_X\le \|g\|_X$ if $|f|\le |g|$ a.e.) such that $\|\chi_{[0,1)}\|_X = 1$ and whenever $f\in X$ and $g\in L_0$ has the same distribution as $f$, it follows that $g\in X$ and $\|g\|_X=\|f\|_X$ (see, e.g., \cite{rodin:semyonov:1975}). Henceforth, this is what we will mean by a rearrangement-invariant space. Given a Haar system space $X$, the identity map $id:(\langle\{\chi_I:I\in\mathcal{D}\}\rangle,\|\cdot\|_X)\to L_1\subset L_0$ extends to a continuous linear injection on $X$ (see, e.g., \cite[the introduction of Section 2.4]{speckhofer:2025:thesis}). We may thus naturally identify $X$ with a subspace of $L_0$. If $\|\cdot\|_X\not\asymp\|\cdot\|_\infty$, the image of $X$ in $L_0$ is a separable rearrangement-invariant space (see, e.g., \cite[Section~2.4]{speckhofer:2025:thesis}). Conversely, every separable rearrangement-invariant space admits $(\chi_{[0,1)})\cup(h_I)_{I\in\mathcal{D}}$ as a basis, and therefore coincides with a Haar system space whose norm is not equivalent to $\|\cdot\|_\infty$ (see \cite[Proposition~2.c.1]{lindenstrauss:tzafriri:1979:partII} or the proof of~\cite[Chapter~1, Theorem~5.5]{bennett:sharpley:1988}).

The Fatou property, which is commonly included in the definition of rearrangement-invariant spaces (see \Cref{Fatou property definition}), is not included in \Cref{haar system space} because a large part of the theory of Haar system spaces can be developed without it (see, e.g., \cite{lechner:motakis:mueller:schlumprecht:2022}, \cite{lechner:speckhofer:2025}, \cite{speckhofer:2025}, and \cite{MR5094444}).
Haar system spaces with $\|\cdot\|_X\not\asymp\|\cdot\|_\infty$ coincide with separable \emph{minimal} rearrangement-invariant function spaces over~$[0,1]$ in the terminology of~\cite{lindenstrauss:tzafriri:1979:partII} (see \Cref{rem:lindenstrauss-tzafriri}).
On the other hand, Haar system spaces with $\|\cdot\|_X\not\asymp\|\cdot\|_\infty$ that have the Fatou property are \emph{maximal} rearrangement-invariant in the sense of~\cite{lindenstrauss:tzafriri:1979:partII}.
We note, purely for context, that a Haar system space $X$ satisfies the Fatou property if and only if it does not contain an isomorphic copy of the Banach space $c_0$.
\end{rem}

\begin{ntt}
  For a Haar system space $X$ we denote $X^0 = \overline{\langle h_I : I\in\mathcal{D}\rangle}^{\|\cdot\|_X} = X\cap L_1^0$.
\end{ntt}
Henceforth, we will mostly work with the space $X^0$ instead of $X$, the main exception being norm computations in the Hardy space associated with~$X$ (see \Cref{dfn:hardy-space}). In fact, any Haar system space $X$ is isomorphic to its hyperplane $X^0$ (see \cite[Proposition 4.10]{lechner:speckhofer:2025}), so this is just a choice of notational convenience.

In the following proposition, we collect some basic results on Haar system spaces which will be used frequently. For proofs, we refer to~\cite[Section~2.4]{lechner:motakis:mueller:schlumprecht:2022} and~\cite{lechner:speckhofer:2025} (see also~\cite[Chapter~2]{speckhofer:2025:thesis} for a more detailed presentation containing all results and proofs).

\begin{prop}\label{pro:HS-1d}
  Let $X$ be a Haar system space. Then the following hold:
  \begin{enumerate}[label=(\alph*)]
    \item\label{pro:HS-1d:i} For every $f\in H:= \langle \{ \chi_I : I\in \mathcal{D} \} \rangle$, we have
          $\|f\|_1\leq \|f\|_X\leq \|f\|_{\infty}$.
    \item\label{pro:HS-1d:v} For all $f,g\in H$ with $|f|\leq |g|$, we have $\|f\|_X \leq \|g\|_X$.
    \item\label{pro:HS-1d:ii} The Haar system $(\chi_{[0,1)})\cup (h_I)_{I\in \mathcal{D}}$, in the usual linear
          order, is a monotone Schauder basis of~$X$.
    \item \label{pro:HS-1d:iii} $H$ naturally coincides with a subspace of $X^{*}$ (where $f\in H$ acts on $g\in H$ by $\langle f,g \rangle = \int fg$), and its closure in $X^{*}$ is also a Haar system space.
  \end{enumerate}
\end{prop}

\begin{rem}\label{rem:lindenstrauss-tzafriri}
  The above properties allow us to show that the Haar system spaces $X$ with $\|\cdot\|_X\not\asymp\|\cdot\|_\infty$ coincide with the separable \emph{minimal} rearrangement-invariant function spaces on~$[0,1]$, as defined in~\cite[Definition~2.a.1]{lindenstrauss:tzafriri:1979:partII}. Indeed, if $X$ is a Haar system space with $\|\cdot\|_X\not\asymp\|\cdot\|_\infty$, then by \cite[Theorem~2.4.3]{speckhofer:2025:thesis}, $X$ is a K\"othe function space. It follows from \cite[Theorem~2.4.8 and Remark~2.4.9]{speckhofer:2025:thesis} that condition~(i) of \cite[Definition~2.a.1]{lindenstrauss:tzafriri:1979:partII} holds (i.e., invariance of~$X$ and its norm under automorphisms of the unit interval). Condition~(ii) holds because the Haar system is a monotone Schauder basis of~$X$, so its linear span in $X^{*}$ is norming and $X$ is the closed span of simple integrable functions in its K\"othe bidual~$X''$ (note that by \cite[Proposition 2.4.2]{speckhofer:2025:thesis}, $L_{\infty}\subset X$). Condition~(iii) follows from the discussion in \Cref{comparing haar system to ri}, \cite[Proposition 2.4.2]{speckhofer:2025:thesis}, and \Cref{pro:HS-1d}~\ref{pro:HS-1d:i}.
\end{rem}

\begin{rem}
  Let $X$ be a Haar system space and let $Z$ be a subspace of $X$. Any distributional embedding $J:Z\to X$ is an isometry. If $\|\cdot\|_X\not\asymp\|\cdot\|_\infty$, then, by \Cref{comparing haar system to ri}, any distributional embedding $J:Z\to L_0$ satisfies $J(Z)\subset X$.
\end{rem}

\begin{lem}\label{lem:cond-exp}
  Let $X$ be a Haar system space with $\|\cdot \|_X \not\asymp \|\cdot \|_{\infty}$. Let $\mathcal{F}$ be a sub-$\sigma$-algebra of~$\mathcal{L}([0,1))$. Then the conditional expectation $\mathbb{E}^{\mathcal{F}}$ maps $X$ into itself and has norm one on $X$.
\end{lem}
\begin{proof}
 This follows from \cite[Theorem~2.a.4]{lindenstrauss:tzafriri:1979:partII} (recall that $X$ is a \emph{minimal} rearrangement-invariant space in the terminology used there). In fact, the proof can be adapted to working directly in the Haar system space~$X$ rather than passing to the maximal space~$X''$, after showing that the simple $\mathcal{F}$-measurable functions are norming in $(X\cap L_0(\mathcal{F}))^{*}$ (see \cite[Proposition~1.b.18]{lindenstrauss:tzafriri:1979:partII}).
\end{proof}

\begin{lem}\label{lem:distributional-copy-complemented}
  Let $X$ be a Haar system space with $\|\cdot \|_X \not\asymp \|\cdot \|_{\infty}$ and let $W$ be a $C$-complemented subspace of $X$ (or $X^0$) for some $C \ge 0$. Let $T\colon W\to L_0$ be a distributional embedding. Then $T[W]$ is also $C$-complemented in~$X$ (respectively, $X^0$).
\end{lem}
\begin{proof}
  Let $P\colon X\to X$ be a bounded linear projection onto $W$. By \Cref{distroembedding extension}, $T$ extends uniquely to an onto distributional embedding $\tilde{T}\colon L_0(\sigma(W))\to L_0(\mathcal{F})$, where $\mathcal{F} = \sigma(T[W])$. By \Cref{lem:cond-exp}, $Q = TP\tilde{T}^{-1}\mathbb{E}^{\mathcal{F}}|_X$ is a bounded linear projection from $X$ onto $T[W]$. The statement for $X^0$ follows from the same argument and the fact that $X^0$ is invariant under conditional expectation operators and distributional embeddings map subspaces of $X^0$ into $X^0$.
\end{proof}

Next, we introduce the (martingale) square function and use it to define the Hardy space over a Haar system space~$X$.
\begin{ntt}
Let $S:L_1^0\to L_0$ denote the square function, i.e., for $f = \sum_{I\in\mathcal{D}}a_Ih_I\in L_1^0$,
\[S(f) = \Big(\sum_{n=0}^\infty \big(\sum_{I\in\mathcal{D}^n}a_Ih_I\big)^2\Big)^{1/2} = \Big(\sum_{I\in\mathcal{D}}a_I^2\chi_I\Big)^{1/2}.\]
\end{ntt}

\begin{rem}
\label{square function continuity}
The square function satisfies the so-called weak type $(1,1)$ estimate
\begin{equation}
\label{weak type (1,1)}
\lambda\big|[S(f)\geq \lambda]\big| \leq \sqrt{e}\|f\|_1,
\end{equation}
where $f\in L_1^0$ and $\lambda >0$ (see, e.g., \cite[Theorem 8.2]{AdamOse:2012}). This suggests that $S:L_1^0\to L_0$ is well-defined and continuous, with the target space equipped with the topology of convergence in measure, because, for $f,g\in L_1^0$, $|S(f) - S(g)|\leq |S(f-g)|$ pointwise.
\end{rem}

\begin{dfn}\label{dfn:hardy-space}
For a Haar system space $X$, the dyadic Hardy space over $X$, denoted $H_X$, is the completion of $\langle\{h_I:I\in\mathcal{D}\}\rangle$ under the norm $\|f\|_{H_X} = \|S(f)\|_X$.
\end{dfn}

\begin{rem}
\label{haar system unconditional in hardy}
Let $X$ be a Haar system space.
\begin{enumerate}[label=(\alph*)]

\item\label{haar system unconditional in hardy a} The Haar system $(h_I)_{I\in\mathcal{D}}$ is, by definition, a 1-unconditional basis for $H_X$.

\item\label{haar system unconditional in hardy b} By \cite[Theorem 1.d.6 (ii)]{lindenstrauss:tzafriri:1979:partII}, $(h_I)_{I\in\mathcal{D}}$ is an unconditional basis for $X^0$ if and only if the norms $\|\cdot\|_X$ and $\|\cdot\|_{H_X}$ are equivalent.

\end{enumerate}
\end{rem}

Next, we collect some fundamental properties of Haar system Hardy spaces. For proofs, we refer to~\cite{lechner:speckhofer:2025} and~\cite[Chapter~2]{speckhofer:2025:thesis}.

\begin{prop}\label{pro:HSHS-1d}
  Let $X$ be a Haar system space and $Y = X^0$ or $H_X$.
  Then the following hold:
  \begin{enumerate}[label=(\alph*)]
    \item $\langle \{ h_I : I\in \mathcal{D} \} \rangle$ naturally coincides with a subspace of $Y^{*}$, and $(h_I)_{I\in \mathcal{D}}$ is $1$-unconditional in $H_X^{*}$.
    \item\label{pro:HSHS-1d:iii} For every $I\in \mathcal{D}$, we have $\|h_I\|_Y\|h_I\|_{Y^{*}} = |I|$. In particular, $\|h_I\|_Y = \|h_I\|_X$ and $\|h_I\|_{Y^{*}} = \|h_I\|_{X^{*}}$.
    \item\label{pro:HSHS-1d:disjointly-supp} If $f$ is a finite linear combination of disjointly supported Haar functions, then $\|f\|_{H_X} = \|f\|_X$.
    \item\label{pro:HSHS-1d:v} If $\mathcal{B}$ is a finite collection of pairwise disjoint dyadic intervals and $(\theta_K)_{K\in \mathcal{B}}\in \{ \pm 1 \}^{\mathcal{B}}$, then $\|\sum_{K\in \mathcal{B}}\theta_Kh_K\|_{Y^{*}}\le 1$.
  \end{enumerate}
\end{prop}

\begin{prop}
\label{hardy spaces without c_0}
For a Haar system space $X$, the identity map
\[id:\big(\langle\{h_I:I\in\mathcal{D}\}\rangle,\|\cdot\|_{H_X}\big)\to L_1^0\]
extends to a bounded linear injection on $H_X$. We may thus naturally identify $H_X$ with a subspace of $L_1^0$, and more precisely, we have
\begin{equation}
\label{Hardy space included}
H_X \subset \{f\in L_1^0: S(f) \in X\}.
\end{equation}
In particular, the defining formula for $\|\cdot\|_{H_X}$ extends to the completion of $\langle\{h_I:I\in\mathcal{D}\}\rangle$. That is, for every $f\in H_X$,
\begin{equation}
\label{square function well-defined in completion}
\|f\|_{H_X} = \|S(f)\|_X.
\end{equation}
If, furthermore, $\|\cdot\|_X\not\asymp\|\cdot\|_\infty$, then
\[H_X = \{f\in L_1^0: S(f) \in X\}.\]
\end{prop}

\begin{proof}
For every $f\in L_1^0$, $\|f\|_1\lesssim \|S(f)\|_1$ (see \cite{davis:1970} or \cite[Theorem 8.7]{AdamOse:2012}) and, for every $g\in X$, $\|g\|_1\le \|g\|_X$ (see, e.g., \cite[Proposition 2.3.1]{speckhofer:2025:thesis}). Consequently, for every $f\in \langle\{h_I:I\in\mathcal{D}\}\rangle$, $\|f\|_1\lesssim \|S(f)\|_X$, which implies boundedness of $id$. Since $id$ maps the Haar system onto itself, and this is a Schauder basis for both $L_1^0$ and $H_X$, the extension of $id$ is injective. To obtain \eqref{Hardy space included}, observe that for every $f\in H_X$ and $m,n\in\mathbb{N}_0$,
$\|S(\mathbb{E}_m f)-S(\mathbb{E}_n f)\|_X \le \|S(\mathbb{E}_n f-\mathbb{E}_m f)\|_X$.
Thus $(S(\mathbb{E}_n f))$ is Cauchy in $\|\cdot\|_X$, and therefore its pointwise limit $S(f)$ belongs to $X$. The same argument yields $\|f\|_{H_X}=\|S(f)\|_X$. If, in addition, $\|\cdot\|_X\not\asymp\|\cdot\|_\infty$ and $f\in L_1^0$ satisfies $S(f)\in X$, then for every $n\in\mathbb{N}$ we have $S(f - \mathbb{E}_n f)\in X$, and this sequence decreases pointwise to zero. By \cite[Proposition 2.4.7]{speckhofer:2025:thesis}, it follows that $\|S(f-\mathbb{E}_nf)\|_X\to 0$. Hence $(\mathbb{E}_{n}f)_{n=0}^\infty$ is Cauchy in $H_X$, and therefore $f\in H_X$.
\end{proof}

Distributional embeddings do not respect the Hardy norm; the appropriate analogue is given by square function-distributional embeddings.
\begin{dfn}
Let $Z$ be a subspace of $L_1^0$. A linear operator $J:Z\to L_1^0$ is called a square function-distributional embedding if, for every $f\in Z$,
\[\mathrm{dist}\big(S(Jf)\big) = \mathrm{dist}(Sf).\]
Denoting $W = J(Z)$, we then write $Z\equiv^{S\text{-dist}}W$ and say that $J$ is a square function-distributional isomorphism onto $W$.
\end{dfn}

\begin{rem}
Compositions and inverses of square function-distributional embeddings are also of this type. Furthermore, if $X$ is a Haar system space and $Z$ is a subspace of $H_X$, by \cref{square function well-defined in completion}, any square function-distributional embedding $J:Z\to H_X$ is an isometry. We note purely for context that, by \Cref{hardy spaces without c_0}, if $\|\cdot\|_X\not\asymp\|\cdot\|_\infty$, then any square function-distributional embedding $J:Z\to L_1^0$ automatically has image contained in $H_X$.
\end{rem}

\begin{rem}\phantom{a}\label{rem:hardy}
\begin{enumerate}[label=(\alph*)]

\item\label{rem:hardy:a} Square function-distributional embeddings differ from distributional embeddings. Two simple examples illustrate the gap. Let $(\mathcal{B}_I)_{I\in\mathcal{D}}$ be pairwise disjoint subsets of $\mathcal{D}$ such that, for each $I\in\mathcal{D}$, $\mathcal{B}_I$ consists of pairwise disjoint intervals and $\bigcup \mathcal{B}_I = I$. Then the map $h_I \mapsto \sum_{J\in \mathcal{B}_I} h_J$ defines a square function distributional embedding, but, unless $\mathcal{B}_I = \{I\}$ for all $I\in\mathcal{D}$, it is not a distributional embedding. Conversely, if $i\mapsto k_i$ is a not increasing injection on $\mathbb{N}_0$, then the extension $\tilde J$, in the sense of \Cref{distroembedding extension}, of the distributional embedding $r_i\stackrel{J}{\mapsto} r_{k_i}$ is not a square function-distributional embedding.

\item\label{rem:hardy:b} To our knowledge, there is no meaningful analogue of \Cref{distroembedding extension} for square function-distributional embeddings. Specifically, if $J:Z\to L_1^0$ is a square function-distributional embedding, one cannot in general expect the existence of a square function-distributional embedding $\tilde J:L_1^0(\sigma(Z))\to L_1^0$ extending $J$. This leads to the following phenomenon in a Haar system space $X$ with $\|\cdot \|_X \not\asymp \|\cdot \|_{\infty}$. If $Z$ is a complemented subspace of $X^0$, then any distributional copy of $Z$ in $X^0$ is also complemented (see \Cref{lem:distributional-copy-complemented}). However, there are situations, for instance when $X=L_1$, in which a subspace $Z\subset H_X$ is complemented, yet admits square function-distributional copies that are not complemented (see \Cref{thm:tree-spaces-uncomplemented}).

\end{enumerate}
\end{rem}

\begin{prop}
\label{standard square function distributional isomorphism}
Let $F = \{k_0<\dots< k_n\}$ be a finite subset of $\mathbb{N}_0$. Then $J_F:Y_n\to Y_F$ is a square function-distributional embedding. If $F$ is infinite, the same is true for $J_F:L_1^0\to Y_F$.
\end{prop}

\begin{proof}
Let $F = \{k_0<\dots< k_n\}$ and $\tilde J_F:L_0(\sigma(Y_n)) = \langle \chi_{[0,1)}\rangle+Y_n\to L_0$ denote the extension of $J_F$ defined in \Cref{distroembedding extension}. Then, for $f = \sum_{I\in\mathcal{D}_n}a_Ih_I\in Y_n$,
\begin{align*}
S(J_Ff) &= S\Big(\sum_{I\in\mathcal{D}_n}a_Ib^F_I\Big) = \Big(\sum_{I\in\mathcal{D}_n}a_I^2\chi_{\mathrm{supp}(b_I^F)}\Big)^{1/2} = \Big(\sum_{I\in\mathcal{D}_n}a_I^2\tilde J_F\chi_{\mathrm{supp}(h_I)}\Big)^{1/2}\\
        &=\tilde J_F\Big(\sum_{I\in\mathcal{D}_n}a_I^2\chi_{\mathrm{supp}(h_I)}\Big)^{1/2}\quad \text{(because $\tilde J_F$ is a homomorphism)}\\
  &= \tilde J_F\big(Sf\big).
\end{align*}
This yields $\mathrm{dist}\big(S(J_Ff)\big) = \mathrm{dist}(S(f))$.
The statement about infinite $F$ follows from the continuity of $S:L_1^0\to L_0$ and $f = \lim_n\mathbb{E}_nf$ in $L_1^0$.
\end{proof}

\begin{rem}
\label{norm-one projection and isometry}
Let $X$ be a Haar system space and $Y = X^0$ or $H_X$. Let $F = \{k_0<\cdots<k_n\}$ be a finite subset of $\mathbb{N}_0$. By \Cref{standard square function distributional isomorphism}, $J_F:Y_n\to Y_F$ is an isometry. Moreover, $\mathbb{E}_F:Y\to Y$ is a norm-one projection onto $Y_F$; for $Y = X^0$, this holds by \cite[Lemma~4.3]{lechner:speckhofer:2025} (or \Cref{lem:cond-exp} when~$\|\cdot\|_X\not\asymp\|\cdot\|_\infty$), and for $Y = H_X$, by the following computation: Let $x = \sum_{K\in  \mathcal{D}} a_K h_K$ be a finite linear combination. Then, using $|I|^{-1}\langle b_I^F, x \rangle b_I^F = \mathbb{E}_F \sum_{K\in \mathcal{B}_I^F} a_K h_K$ for all $I\in \mathcal{D}_n$ (see \Cref{haar basis formulas}~\ref{haar basis formulas a}) and conditional Jensen's inequality for the $\ell_2$-norm, we obtain
\begin{align*}
  \|\mathbb{E}_F x\|_{H_X}
  &= \Bigl\| \Bigl( \sum_{I\in \mathcal{D}_n} \bigl(  \mathbb{E}_F \sum_{K\in \mathcal{B}_I^F} a_Kh_K \bigr) ^2 \Bigr)^{1/2} \Bigr\|_X \le \Bigl\| \mathbb{E}_F \Bigl( \sum_{I\in \mathcal{D}_n} \bigl(  \sum_{K\in \mathcal{B}_I^F} a_K h_K \bigr)^2 \Bigr)^{1/2} \Bigr\|_X\\
  &\le \Bigl\| \Bigl( \sum_{I\in \mathcal{D}_n} \sum_{K\in \mathcal{B}_I^F} a_K^2h_K^2 \Bigr)^{1/2} \Bigr\|_X\le  \|x\|_{H_X}.
\end{align*}
See also \cite[Proposition 7.1]{lechner:speckhofer:2025} for an alternative proof (using an equivalent norm). The above statements also hold for infinite $F$, with $Y_F$ understood as $Y\cap L_0(\sigma_F)$.
\end{rem}

\section{Disjoint and independent sums in Haar system spaces and dyadic Hardy spaces}
\label{types of sums}

In this section, we introduce disjoint sums and independent sums of function spaces, which are used in the iterative construction of Bourgain-Rosenthal-Schechtman spaces. Because we work inside the centered space $L_1^0$, we use an augmented version of disjoint sums, as explained in \Cref{why augment}, and we define the corresponding analogue of independent sums in dyadic Hardy spaces.

\begin{dfn}
\label{def  compression}
For an interval $I = [a,b)\subset[0,1)$ we define $T_I:L_0\to L_0$ given by
\[\big(T_If\big)(t) = \left\{\begin{array}{ll}0&\text{ if }t\notin I\\
f\big(\frac{t-a}{b-a}\big)&\text{ if }t\in I.\end{array}\right.\]
\end{dfn}

\begin{rem}
\label{standard half compressions remark}
Let $f\in L_0$. Then, $T_{[0,1/2)}f$ and $T_{[1/2,1)}f$ are disjointly supported and
\[\mathrm{dist}(T_{[0,1/2)}f+T_{[1/2,1)}f) = \mathrm{dist}(f).\]
If $f\in L_1^0$, then, for $T = T_{[0,1/2)}$ or $T_{[1/2,1)}$, letting $S$ denote the square function operator,
\begin{equation}
\label{standard half compressions play nice with square function}
TS(f) = S(Tf).
\end{equation}
In particular, if $X$ is a Haar system space and $Y = X^0$ or $H_X$, then, for all $f\in Y$, $Tf\in Y$ and
\begin{equation}
\label{standard half compressions are 2-embeddings}
(1/2)\|f\|_Y\leq \|Tf\|_Y \leq \|f\|_Y.
\end{equation}
\end{rem}

We will later rely on the following observation, which holds both for Haar system spaces and for the associated Hardy spaces.
\begin{prop}\label{prop:isomorphic-to-square}
  Let $X$ be a Haar system space and $Y = X^0$ or $H_X$. Then $Y$ is isomorphic to its square and to its hyperplanes.
\end{prop}
\begin{proof}

  By~\cite[Proposition~4.10]{lechner:speckhofer:2025}, the spaces $X$, $X^0$, and $H_X$ are isomorphic to their hyperplanes. Now observe that $Y \simeq \mathbb{R} \oplus [h_I : I \subset [0,\frac{1}{2})] \oplus [h_I : I \subset [\frac{1}{2},1)]$, and the last two summands are isomorphic to~$Y$ by~\eqref{standard half compressions are 2-embeddings}.
\end{proof}

\begin{dfn}
\label{def augmented disjoint sum}
For two subspaces $Z$, $W$ of $L_0$ we define their disjoint sum as
\[(Z\oplus W)_{\mathrm{disj}} = T_{[0,1/2)}(Z)+T_{[1/2,1)}(W).\]
If $Z,W\subset L_1^0$, then $h_{[0,1)}\notin(Z\oplus W)_{\mathrm{disj}}$, and we define their augmented disjoint sum as
\[(Z\oplus W)_{\mathrm{disj}}^+ = \langle\{ h_{[0,1)}\}\rangle + (Z\oplus W)_{\mathrm{disj}},\]
which is a subspace of $L_1^0$.
\end{dfn}

\begin{rem}
\label{why augment}
In \cite{bourgain:rosenthal:schechtman:1981}, Bourgain-Rosenthal-Schechtman spaces are constructed using disjoint sums together with independent sums (formulated on different underlying measure spaces). In that setting, the spaces always contain the constant function and therefore are not subspaces of $L_1^0$. When working inside $L_1^0$, to preserve the distributional properties of the construction, disjoint sums must be replaced with  their augmented counterparts, since for any $Z,W\subset L_1$ containing the constant function,
\[(Z\oplus W)_{\mathrm{disj}}\cap L_1^0=\big((Z\cap L_1^0)\oplus (W\cap L_1^0)\big)_{\mathrm{disj}}^+.\]
The outcome of this adjustment is that if we denote by $R_\alpha^p$ a Bourgain-Rosenthal-Schechtman space constructed via the iteration in \cite{bourgain:rosenthal:schechtman:1981} (up to identifications of the underlying measure spaces) and by $R_\alpha^Y$, where $Y = L_p^0$, the corresponding space constructed in this paper, then $R_\alpha^p\cap L_1^0 = R_\alpha^Y$ (see \Cref{same spaces sans constant}).

Independent sums do not need to be adjusted to achieve $R_\alpha^p\cap L_1^0 = R_\alpha^Y$.

\end{rem}

\begin{rem}
\label{disjoint sum is really a sum}
Let $X$ be a Haar system space, $Y = X^0$ or $H_X$, and $Z$, $W$ be subspaces of $Y$. \Cref{standard half compressions remark} yields
\[(Z\oplus W)^+_{\mathrm{disj}}\simeq \mathbb R\oplus Z\oplus W.\]
\end{rem}

\begin{rem}
If $Z\equiv^\mathrm {dist}Z'$ and $W\equiv^\mathrm {dist} W'$ are subspaces of $L_1^0$ then
\[(Z\oplus W)^+_{\mathrm{disj}} \equiv^\mathrm{dist} (Z'\oplus W')^+_{\mathrm{disj}}.\]
The same is true replacing ``$\equiv^\mathrm {dist}$'' with ``$\equiv^{S\text{-dist}}$'' (use, e.g., \eqref{standard half compressions play nice with square function}).
\end{rem}

\begin{dfn}\label{dfn:independence}
\label{independence}~
Let $(X_n)_{n=1}^\infty$ be a sequence of subspaces of $L_1^0$.
\begin{enumerate}[label=(\alph*)]

\item We call $(X_n)_{n=1}^\infty$ independent if the $\sigma$-algebras $\sigma(X_n)$, $n\in\mathbb{N}$, are independent.

\item\label{square function-independence} We call $(X_n)_{n=1}^\infty$ square function-independent if, for every eventually zero sequence $(f_n)_{n=1}^\infty\in\prod_{n=1}^\infty X_n$,
\begin{enumerate}[label=(\roman*)]

\item\label{square function-independence a} $(Sf_n)_{n=1}^\infty\text{ is independent and }$

\item\label{square function-independence b} $S\big(\sum_{n=1}^\infty f_n\big) = \Big(\sum_{n=1}^\infty S^2(f_n)\Big)^{1/2}.$

\end{enumerate}
\end{enumerate}
\end{dfn}

\begin{rem}
\label{also applies for summable}
By \Cref{square function continuity}, if $(X_n)_{n=1}^\infty$ is square function-independent, then conclusions \ref{square function-independence a} and \ref{square function-independence b} in \Cref{dfn:independence} also hold for any sequence $(f_n)_{n=1}^\infty\in\prod_{n=1}^\infty X_n$ that is summable in $L_1^0$, and, by \Cref{hardy spaces without c_0}, for any sequence that is summable in a dyadic Hardy space $H_X$ over a Haar system space $X$.
\end{rem}

\begin{rem}\label{rem:faithful-embeddings}
Let $(F_n)_{n=1}^\infty$ be a sequence of pairwise disjoint subsets of $\mathbb{N}_0$. The sequence $(Y_{F_n})_{n=1}^\infty$ is independent and square function-independent. The latter is true because for any $f\in Y_{F_n}$, $S(f)\in L_0(\sigma_{F_n})$. In particular, the sequence of operators $(J_{F_n})_{n=1}^\infty$ has independent and square function-independent ranges.
\end{rem}

\begin{dfn}
\label{def independent sum}
Let $X$ be a Haar system space and let $Y=X^0$ or $H_X$. Let $(Z_n)_{n=1}^\infty$ be a sequence of closed subspaces of $Y$. Let $J_n:Z_n\to Y$, $n\in\mathbb{N}$, be embeddings such that:
\begin{enumerate}[label=(\alph*)]
\item if $Y=X^0$, the operators $J_n$ are distributional embeddings with independent ranges;
\item if $Y=H_X$, the operators $J_n$ are square function-distributional embeddings with square function-independent ranges.
\end{enumerate}
The independent sum of $(Z_n)_{n=1}^\infty$ is
\[\big(\oplus_{n=1}^\infty Z_n\big)_{Y,\mathrm{ind}} = \big[\big(J_n(Z_n)\big)_{n=1}^\infty\big].\]
\end{dfn}

\begin{rem}
\label{independent sums are unconditional sums}
Following \Cref{def independent sum}, $(J_n(Z_n))_{n=1}^\infty$ forms a suppression unconditional Schauder decomposition of $(\oplus_{n=1}^\infty Z_n)_{Y,\mathrm{ind}}$, and each $J_n(Z_n)$ is isometric to $Z_n$.
\end{rem}

\begin{prop}
\label{canonical properties of independent sums}
Let $X$ be a Haar system space and $Y = X^0$
\begin{enumerate}[label=(\alph*)]

\item\label{canonical properties of independent sums a} The choice of operators in \Cref{def independent sum} does not matter, up to distributional isomorphism: if $J_n'$, $n\in\N$, is another such sequence then
\[\big[\big(J_n(Z_n)\big)_{n=1}^\infty\big]\equiv^\mathrm{dist}\big[\big(J'_n(Z_n)\big)_{n=1}^\infty\big].\]

\item\label{canonical properties of independent sums b} If $(Z_n)_{n=1}^\infty$, $(Z_n')_{n=1}^\infty$ are closed subspaces of $Y$ such that, for all $n\in\mathbb{N}$, $Z_n\equiv^\mathrm{dist} Z_n'$, then
\[\big(\oplus_{n=1}^\infty Z_n\big)_{Y,\mathrm{ind}}\equiv^\mathrm{dist} \big(\oplus_{n=1}^\infty Z_n'\big)_{Y,\mathrm{ind}}.\]

\item\label{canonical properties of independent sums c} If $(Z_n)_{n=1}^\infty$ are closed subspaces of $Y$ and $\pi\colon \mathbb{N}\to \mathbb{N}$ is a bijection, then
\[\big(\oplus_{n=1}^\infty Z_n\big)_{Y,\mathrm{ind}}\equiv^\mathrm{dist} \big(\oplus_{n=1}^\infty Z_{\pi(n)}\big)_{Y,\mathrm{ind}}.\]

\item\label{canonical properties of independent sums d} If $(Z_m^n)_{m,n=1}^\infty$ are closed subspaces of $Y$, then
\[\Big(\oplus_{n=1}^\infty\big(\oplus_{m=1}^\infty Z_m^n\big)_{Y,\mathrm{ind}}\Big)_{Y,\mathrm{ind}}\equiv^\mathrm{dist} \big(\oplus_{m,n=1}^\infty Z_m^n\big)_{Y,\mathrm{ind}},\]
where the enumeration of the sum on the right-hand side does not matter.

\end{enumerate}
If  instead $Y = H_X$, the above statements remain valid, replacing ``distributional isomorphism'' with ``square function-distributional isomorphism'' and ``$\equiv^\mathrm{dist}$'' with ``$\equiv^{S\text{-dist}}$''.
\end{prop}

\begin{proof}
The statements when $Y = X^0$ are based on the following well-known probability fact: If $(f_n)_{n=1}^\infty$ and $(f_n')_{n=1}^\infty$ are two sequences of random variables, each of which is independent, and, for all $n\in\mathbb{N}$, $\mathrm{dist}(f_n) = \mathrm{dist}(f_n')$, then $\mathrm{dist}((f_n)_{n=1}^\infty) = \mathrm{dist}((f_n')_{n=1}^\infty)$. In particular, if $(Z_n)_{n=1}^\infty$ is an independent sequence of subspaces of $X^0$ and  $J_n:Z_n\to X^0$, $n\in\mathbb{N}$, are distributional embeddings with independent ranges, then the operator $J:[(Z_n)_{n=1}^\infty]\to X^0$ given by
\[J\Big(\sum_{n=1}^\infty f_n\Big) = \sum_{n=1}^\infty J_nf_n\]
is a distributional embedding. We will elaborate on the details of the statements for $Y=H_X$, starting with \ref{canonical properties of independent sums a}. By the second part of \Cref{independence} \ref{square function-independence} it follows that each of the sequences of spaces $(J_n(Z_n)\big)_{n=1}^\infty$ and $(J'_n(Z_n)\big)_{n=1}^\infty$ is linearly independent, and thus the map $J:\big\langle \big(J_n(Z_n)\big)_{n=1}^\infty\big\rangle\to \big\langle\big(J'_n(Z_n)\big)_{n=1}^\infty\big\rangle$ given by $J_nf_n \mapsto J_n'f_n$ is a linear surjection. Using both statements in \Cref{independence} \ref{square function-independence}, we get that for any eventually zero sequence $(f_n)_{n=1}^\infty\in\prod_{n=1}^\infty Z_n$,
\begin{align*}
S\Big(J\big(\sum_{n=1}^\infty J_nf_n\big)\Big) &= S\Big(\sum_{n=1}^\infty J_n'f_n\Big)  = \Big(\sum_{n=1}^\infty S^2(J_n'f_n)\Big)^{1/2}\equiv^\mathrm{dist} \Big(\sum_{n=1}^\infty S^2(J_nf_n)\Big)^{1/2}\\
& = S\Big(\sum_{n=1}^\infty J_nf_n\Big).
\end{align*}
Therefore, $J$ extends to a square function-distributional isomorphism. Statements~\ref{canonical properties of independent sums b} and~\ref{canonical properties of independent sums c} are easy applications of~\ref{canonical properties of independent sums a}, we thus prove~\ref{canonical properties of independent sums d}. We fix collections of subspaces $(Z_m^n)_{m,n=1}^\infty$ of $H_X$ and square function-distributional embeddings $J_m^n:Z_m^n\to H_X$ with square function-independent ranges to define
\[\big(\oplus_{m,n=1}^\infty Z_m^n\big)_{Y,\mathrm{ind}} = \big[\big(J_m^n(Z_m^n)\big)_{m,n=1}^\infty\big].\]
We can also define, for each $n\in\mathbb{N}$,
\[W_n:=\big(\oplus_{m=1}^\infty Z_m^n\big)_{Y,\mathrm{ind}} = [(J_m^n(Z^n_m))_{m=1}^\infty].\]
Using \Cref{also applies for summable} it is straightforward to show that the spaces $(W_n)_{n=1}^\infty$ are square function-independent, and we may therefore use the identity maps $I:W_n\to H_X$ to define
\[\big(\oplus_{n=1}^\infty W_n\big)_{Y,\mathrm{ind}} = [(W_n)_{n=1}^\infty].\]
We have therefore shown,
\[\Big(\oplus_{n=1}^\infty\big(\oplus_{m=1}^\infty Z_m^n\big)_{Y,\mathrm{ind}}\Big)_{Y,\mathrm{ind}} = [(W_n)_{n=1}^\infty] = \big[\big(J_m^n(Z_m^n)\big)_{m,n=1}^\infty\big] = \big(\oplus_{m,n=1}^\infty Z_m^n\big)_{Y,\mathrm{ind}}.\]
The equality holds because of the embeddings we chose to define the independent sums.
\end{proof}

\section{Bourgain-Rosenthal-Schechtman spaces as sums}

Using disjoint and independent sums, we define the Bourgain-Rosenthal-Schechtman spaces by transfinite induction and show that they are isomorphic to their squares and hyperplanes. This formulation is equivalent to that of \cite{bourgain:rosenthal:schechtman:1981}, but follows more closely the language of \cite{konstantos:motakis:2025} (see \Cref{same spaces sans constant}), with the notational adjustments introduced in \Cref{types of sums} to accommodate both Haar system and dyadic Hardy spaces.

\begin{dfn}
\label{BRS definition sums}
Let $X$ be a Haar system space and $Y = X^0$ or $H_X$. We define the subspaces $R_\alpha^{Y}$, $0\leq \alpha<\omega_1$, of $Y$, by transfinite recursion.
\begin{enumerate}[label=(\alph*)]

\item In the base case, we let $R_0^{Y} = \{0\}$.

\item In the successor case $\alpha = \beta+1$, we let
\[R^{Y}_\alpha = \big(R^{Y}_\beta\oplus R^{Y}_\beta\big)^+_{\mathrm{disj}}.\]

\item In the limit case we let
\[R^{Y}_\alpha = \big(\oplus_{\beta<\alpha}R^{Y}_\beta\big)_{Y,\mathrm{ind}}.\]

\end{enumerate}
\end{dfn}

\begin{rem}\label{rem:nth-successor-BRS-space}
  Note that if $\gamma$ is a countable ordinal, then, for every $n\in \mathbb{N}$,
  \begin{equation*}
    R_{\gamma + n}^{Y} = \langle \{ h_I : I\in \mathcal{D}_{n-1} \} \rangle + \sum_{I\in \mathcal{D}^n} T_I(R_\gamma^{Y}),
  \end{equation*}
  where $T_I$ is as in \Cref{def  compression}. The Bourgain-Rosenthal-Schechtman spaces $R_\alpha^Y$ for $\omega\le \alpha\le \omega\cdot 2$ are illustrated in \Cref{fig:BRS-infinite}.
\end{rem}

\begin{figure}[t]
  \centering \includegraphics[width=\textwidth]{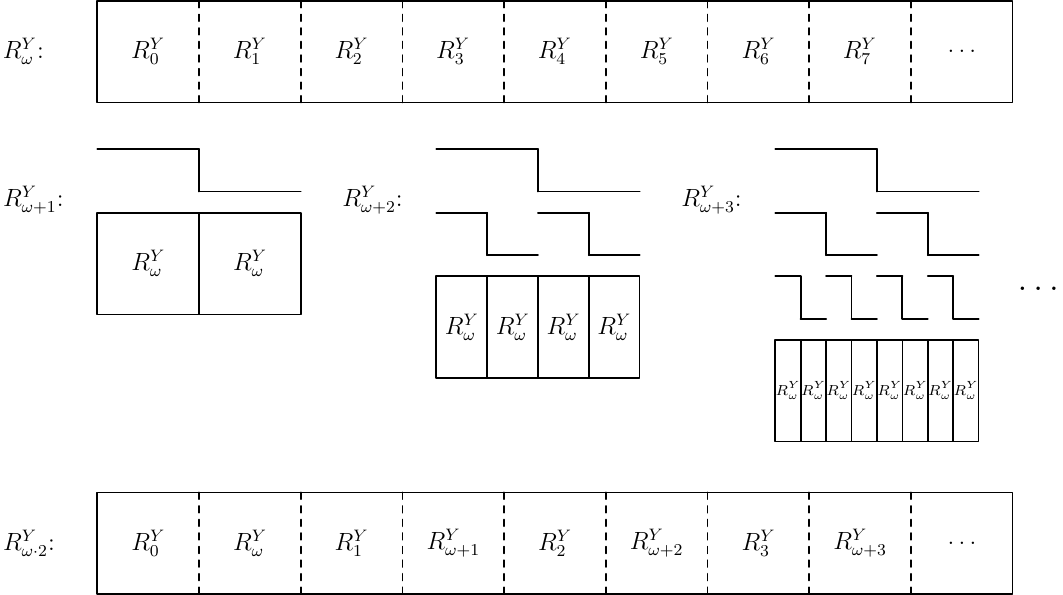}
  \caption{The first infinite-dimensional Bourgain-Rosenthal-Schechtman spaces. Disjoint and independent sums are represented by collections of boxes.}
  \label{fig:BRS-infinite}
\end{figure}

\begin{rem}
\label{same spaces sans constant}
By \cite[Remark 3.12]{konstantos:motakis:brs} (stated there for $X=L_p$), if for $0\le \alpha<\omega_1$ we denote by $R_\alpha^p$ the original Bourgain-Rosenthal-Schechtman space including the constant function obtained via the iterative construction in \cite{bourgain:rosenthal:schechtman:1981} (after identifying the underlying measure spaces, as described in \cite{konstantos:motakis:brs}), and by $R_\alpha^{L_p^0}$ the corresponding space constructed inside $L_1^0$, then $R_\alpha^p \cap L_1^0 = R_\alpha^{L_p^0}$.
\end{rem}

Our next goal is to prove that the spaces $R_\alpha^{Y}$ constructed above are isomorphic to their squares if $\alpha$ is infinite. We will need the following lemma.

 \begin{lem}\label{lem:R-omega-Y-0-isomorphic-to-its-square}
  Let $X$ be a Haar system space and $Y = X^0$ or $H_X$. Let $\gamma$ be a countable ordinal and suppose that $R_\gamma^{Y}$ is isomorphic to its square. Then the space $\bigl(\oplus_{n=0}^{\infty} R_{\gamma + n}^{Y} \bigr)_{Y,\mathrm{ind}}$ is isomorphic to its square and to its hyperplanes. In particular, this is true for $R_{\omega}^{Y} = \bigl(\oplus_{n=0}^{\infty} R_n^{Y} \bigr)_{Y,\mathrm{ind}}$.
\end{lem}
\begin{proof}
  Let $Z = (\oplus_{n=0}^{\infty} R_{\gamma + n}^{Y})_{Y,\mathrm{ind}}$.
  We can write
  \begin{equation*}
     Z = \bigl[\bigl(J_{F_n}(R_{\gamma + n}^{Y})\bigr)_{n=0}^{\infty}\bigr],
  \end{equation*}
  where $(F_n)_{n=0}^{\infty}$ is a sequence of pairwise disjoint infinite subsets of $\mathbb{N}_0$ and the embeddings $J_{F_n}$ are as in \Cref{notation of subsequences of the levels}~\ref{notation of subsequences of the levels b}. Now we define four projections on $\langle \{ b^{F_n}_I : I\in \mathcal{D}, n\in \mathbb{N}_0 \}  \rangle$ by putting, for $f = \sum_{n=0}^N \sum_{I\in \mathcal{D}} a_I^nb_I^{F_n}$,
  \begin{align*}
    P_0 f &= \sum_{I\in \mathcal{D}} a_I^0 b_I^{F_0},
    & P_1 f &= \sum_{n=1}^N a_{[0,1)}^n b_{[0,1)}^{F_n}, \\
    P_2 f &= \sum_{n=1}^N \sum_{\substack{I\in \mathcal{D}\\ I\subset [0,\frac{1}{2})}} a_I^n b_I^{F_n},
    & P_3 f &= \sum_{n=1}^N \sum_{\substack{I\in \mathcal{D}\\ I\subset [\frac{1}{2},1)}} a_I^n b_I^{F_n}.
  \end{align*}
  We will show that the projections $P_0$, $P_1$, $P_2$, and $P_3$ are bounded. If $Y = H_X$, this follows from the $1$-unconditionality of the Haar system. If $Y = X^0$, then it follows by repeatedly applying the following basic observation: If $n\in \mathbb{N}_0$ and $I\in \mathcal{D}$ are such that $a_J^n = 0$ for all $J \subset I$, then, for any scalar~$t$, $f + t b_I^{F_n}$ and $f - t b_I^{F_n}$ have the same distribution, so
  \begin{equation}\label{eq:5}
    \|f\|_X \le \frac{1}{2}\bigl( \|f + tb_I^{F_n}\|_X + \|f - tb_I^{F_n}\|_X \bigr) = \|f + tb_I^{F_n}\|_X.
  \end{equation}
  Thus, $P_0$, $P_1$, $P_2$, and $P_3$ extend to bounded linear projections on $[(J_{F_n}(Y))_{n=0}^{\infty}]$ that map $Z$ into itself (this follows from \Cref{rem:nth-successor-BRS-space}).
  Put $Z_0 = P_0(Z)$, $U = P_1(Z)$, $V = P_2(Z)$, and $W = P_3(Z)$.
  Then $Z\simeq Z_0 \oplus U \oplus V \oplus W$, where $Z_0$ is isometric to~$R_\gamma^{Y}$.

  We will now show that $U\simeq U \oplus U$, and $V\simeq W\simeq Z$.
  The first statement holds because $U = [r_{\min(F_n)} : n\in \mathbb{N}]$. In particular, we have $U \simeq U \oplus \mathbb{R}$, and hence, $Z\simeq Z \oplus \mathbb{R}$. To prove that $V\simeq W\simeq Z$, we define, for $I = [0,\frac{1}{2})$ or $[\frac{1}{2},1)$, the operator $\widetilde{T}_I$ on $\langle \{ b_I^{F_n} : I\in \mathcal{D}, n\in \mathbb{N}_0 \} \rangle$ by
  \begin{equation*}
    \widetilde{T}_I\Bigl( \sum_{n=0}^N J_{F_n}f_n \Bigr) = \sum_{n=0}^N J_{F_n}\bigl( T_I f_n\bigr)
  \end{equation*}
  for $f_0,\dots,f_N\in \langle \{ h_I : I\in \mathcal{D} \} \rangle$. Let $f = \sum_{n=0}^N J_{F_n} f_n$, and observe that $\widetilde{T}_{[0,\frac{1}{2})}f$ and $\widetilde{T}_{[\frac{1}{2},1)}f$ have the same distribution. The same is true for their square functions $S(\widetilde{T}_{[0,\frac{1}{2})}f)$ and $S(\widetilde{T}_{[\frac{1}{2},1)}f)$ (see \Cref{standard square function distributional isomorphism}). Thus,
  \begin{equation*}
    \|\widetilde{T}_{[0,\frac{1}{2})}f\|_Y
    = \frac{1}{2} \bigl( \|\widetilde{T}_{[0,\frac{1}{2})}f\|_Y + \|\widetilde{T}_{[\frac{1}{2},1)}f\|_Y \bigr)
    \ge \frac{1}{2} \|\widetilde{T}_{[0,\frac{1}{2})}f + \widetilde{T}_{[\frac{1}{2},1)}f\|_Y.
  \end{equation*}
  On the other hand, \eqref{eq:5} implies that
  \begin{equation*}
    \|\widetilde{T}_{[0,\frac{1}{2})}f\|_Y \le \|\widetilde{T}_{[0,\frac{1}{2})}f + \widetilde{T}_{[\frac{1}{2},1)}f\|_Y.
  \end{equation*}
  Note that $\widetilde{T}_{[0,\frac{1}{2})}f + \widetilde{T}_{[\frac{1}{2},1)}f$ has the same distribution as $f$, and the same is true for their square functions. Hence,
  \begin{equation*}
    \|f\|_Y = \|\widetilde{T}_{[0,\frac{1}{2})}f + \widetilde{T}_{[\frac{1}{2},1)}f\|_Y \asymp \|\widetilde{T}_{[0,\frac{1}{2})}f\|_Y,
  \end{equation*}
  and so, $\widetilde{T}_{[0,\frac{1}{2})}$ extends to an embedding of $[(J_{F_n}(Y))_{n=0}^{\infty}]$ into itself. The same is true for $\widetilde{T}_{[\frac{1}{2},1)}$.
  Now observe that $\widetilde{T}_{[0,\frac{1}{2})}(Z)$ is isometric to~$V$ (via the shift $b_I^{F_n}\mapsto b_I^{F_{n+1}}$), and $\widetilde{T}_{[\frac{1}{2},1)}(Z)$ is isometric to~$W$.

  Finally, since $Z_0$ and $U$ are complemented in $Z$ and isomorphic to their squares, we conclude that $Z \simeq Z \oplus U$ and $Z \simeq Z \oplus Z_0$. Hence,
  \begin{equation*}
    Z \simeq Z_0 \oplus U \oplus V \oplus W \simeq Z_0 \oplus U\oplus Z \oplus Z \simeq Z \oplus Z.\qedhere
  \end{equation*}
\end{proof}

\begin{prop}\label{prop:R-alpha-Y-0-isomorphic-to-its-square}
  Let $X$ be a Haar system space and $Y = X^0$ or $H_X$.
  Then, for every countable ordinal $\omega\le \alpha < \omega_1$ the space $R_\alpha^{Y}$ is isomorphic to its square and to its hyperplanes.
\end{prop}

\begin{proof}
  It can be easily shown by transfinite induction that for every $\omega \le  \alpha < \omega_1$, $R_\alpha^{Y}$ contains a complemented subspace isomorphic to $R_\omega^{Y}$. By \Cref{lem:R-omega-Y-0-isomorphic-to-its-square}, $R_\omega^{Y}\simeq R_\omega^{Y}\oplus \mathbb{R}$, hence, $R_\alpha^{Y}$ is isomorphic to its hyperplanes as well.

  Now we use transfinite induction to show that $R_\alpha^{Y}$ is isomorphic to its square for all $\omega \le \alpha < \omega_1$. For $\alpha = \omega$, this is proved in \Cref{lem:R-omega-Y-0-isomorphic-to-its-square}, so we assume that $\alpha > \omega$ and that the statement holds for all infinite countable ordinals $\beta < \alpha$.
  In the following, replace ``$\equiv^\mathrm {dist}$'' with ``$\equiv^{S\text{-dist}}$'' if $Y=H_X$.

  \noindent \textbf{Case 1.} $\alpha = \gamma + \omega$ for some limit ordinal $\gamma < \omega_1$.
  Using \Cref{canonical properties of independent sums}~\ref{canonical properties of independent sums d}, we can write
  \begin{align*}
    R_\alpha^{Y} &= \bigl(\oplus_{\beta < \alpha} R_\beta^{Y}\bigr)_{Y,\mathrm{ind}}\\
    &\equiv^\mathrm{dist} \Bigl( \bigl(\oplus_{\beta < \gamma} R_\beta^{Y}\bigr)_{Y,\mathrm{ind}} \oplus \bigl(\oplus_{n=0}^{\infty} R_{\gamma + n}^{Y} \bigr)_{Y,\mathrm{ind}} \Bigr)_{Y,\mathrm{ind}}\\
    &\simeq  \bigl(\oplus_{\beta < \gamma} R_\beta^{Y}\bigr)_{Y,\mathrm{ind}} \oplus \bigl(\oplus_{n=0}^{\infty} R_{\gamma + n}^{Y} \bigr)_{Y,\mathrm{ind}} \\
    &\simeq R_\gamma^{Y} \oplus \bigl(\oplus_{n=0}^{\infty} R_{\gamma + n}^{Y} \bigr)_{Y,\mathrm{ind}}.
  \end{align*}
  By the induction hypothesis, the space $R_\gamma^{Y}$ is isomorphic to its square, and by \Cref{lem:R-omega-Y-0-isomorphic-to-its-square}, the same is true for $\bigl(\oplus_{n=0}^{\infty} R_{\gamma + n}^{Y} \bigr)_{Y,\mathrm{ind}}$. Thus, $ R_\alpha^{Y} \simeq  R_\alpha^{Y} \oplus R_\alpha^{Y} $.

  \noindent \textbf{Case 2.} $\alpha = \sup\{ \beta < \alpha : \beta \text{ is a limit ordinal} \}$. Then we can pick a strictly increasing sequence $(\alpha_n)_{n=1}^{\infty}$ of countable limit ordinals with supremum $\alpha$. Let $B = \{ \beta < \alpha : \beta\ne \alpha_n \text{ for all }n\in \mathbb{N} \}$. Using \Cref{canonical properties of independent sums}~\ref{canonical properties of independent sums d} and~\ref{canonical properties of independent sums b}, we obtain
  \begin{align*}
    R_\alpha^{Y} = \bigl(\oplus_{\beta < \alpha} R_\beta^{Y}\bigr)_{Y,\mathrm{ind}}
    &\equiv^\mathrm{dist} \Bigl( \bigl(\oplus_{n=1}^{\infty} R_{\alpha_n}^{Y} \bigr)_{Y,\mathrm{ind}} \oplus \bigl(\oplus_{\beta\in B} R_{\beta}^{Y} \bigr)_{Y,\mathrm{ind}} \Bigr)_{Y,\mathrm{ind}}\\
    &\equiv^\mathrm{dist} \Bigl( \bigl(\oplus_{n=1,\beta < \alpha_n}^{\infty} R_\beta^{Y}\bigr)_{Y,\mathrm{ind}} \oplus \bigl(\oplus_{\beta\in B} R_{\beta}^{Y} \bigr)_{Y,\mathrm{ind}} \Bigr)_{Y,\mathrm{ind}}.
  \end{align*}
  In the first independent sum, each space $R_\beta^{Y}$, $\beta < \alpha$, appears infinitely many times. Thus, by \Cref{canonical properties of independent sums}~\ref{canonical properties of independent sums d} and~\ref{canonical properties of independent sums c}, we have
  \begin{equation}\label{eq:6}
    R_\alpha^{Y} \equiv^\mathrm{dist} \bigl(\oplus_{n=1,\beta < \alpha}^{\infty} R_\beta^{Y}\bigr)_{Y,\mathrm{ind}}.
  \end{equation}
  This implies that $R_\alpha^{Y} \equiv^\mathrm{dist} \bigl(R_\alpha^{Y} \oplus R_\alpha^{Y} \bigr)_{Y,\mathrm{ind}}$, so $R_\alpha^{Y}$ is isomorphic to its square.

  \noindent \textbf{Case 3.} $\alpha = \beta + 1$ is a successor ordinal. Then, by \Cref{disjoint sum is really a sum} and the induction hypothesis,
  \begin{align*}
    R^{Y}_\alpha = \big(R^{Y}_\beta\oplus R^{Y}_\beta\big)^+_{\mathrm{disj}}
    &\simeq \mathbb{R} \oplus R^{Y}_\beta\oplus R^{Y}_\beta\\
    &\simeq \mathbb{R} \oplus (R^{Y}_\beta\oplus R^{Y}_\beta)\oplus (\mathbb{R} \oplus  R^{Y}_\beta\oplus R^{Y}_\beta).\\
    &\simeq R^{Y}_\alpha \oplus R^{Y}_\alpha.\qedhere
  \end{align*}
\end{proof}

\section{Bourgain-Rosenthal-Schechtman spaces over disjointly branching trees}

We define the special classes of trees $\mathcal{T}$ used in the construction of the Bourgain-Rosenthal-Schechtman tree spaces $R_\mathcal{T}^{Y}$ and show that, for every $\alpha<\omega_1$, the space $R_\alpha^{Y}$ admits such a representation. This realization yields the same spaces as in \cite[Section 3]{bourgain:rosenthal:schechtman:1981}, where the tree spaces are defined over a different underlying measure space, while maintaining notational consistency with the present setting and allowing a direct visualization on $[0,1)$. Next, we show that, if $X$ is a Haar system space with non-trivial Boyd indices, then for every $\mathcal{T}$ the orthogonal projection onto $R_\mathcal{T}^{Y}$ is bounded. This was already proved by Ghawadrah in \cite{ghawadrah:2016}. In contrast to that paper, our argument does not proceed via the boundedness of the orthogonal projection on the Bourgain-Rosenthal-Schechtman spaces in $L_p$ combined with interpolation, but instead appeals to Stein's inequality \cite[Proposition 9.5, p.~243]{johnson:maurey:schechtman:tzafriri:1979} on $X$, whose proof itself relies on interpolation methods. Thus interpolation is not avoided entirely, but enters indirectly through Stein's inequality rather than through complementation in $L_p$. The overall approach remains close in spirit to the original ideas of Bourgain, Rosenthal, and Schechtman \cite{bourgain:rosenthal:schechtman:1981}, who already suggested Stein's inequality as an alternative ingredient for their proof (\cite[Remark 2, page 223]{bourgain:rosenthal:schechtman:1981}). The Bourgain-Rosenthal-Schechtman tree spaces also appear in \Cref{complemented ordinal section}, where we introduce and study an ordinal index characterizing the complementation of $Y=X$ or $H_X$ in an arbitrary separable Banach space, under the assumption that $X$ satisfies the Fatou property.

\begin{ntt}
We denote $\mathbb{N} = \{1,2,\ldots\}$ and $\mathbb{N}_0 = \{0,1,2,\ldots\}$. Consider the collection $[\mathbb{N}_0]^{<\infty}$ of finite subsets of $\mathbb{N}_0$ with the partial order ``$\sqsubseteq$'' of initial segments, i.e., $A\sqsubseteq B$ means $A\subset B$ and $\min(B\setminus A)>\max(A)$ (using the conventions $\min(\emptyset)=\infty$ and $\max(\emptyset) = -\infty$). Note that $([\mathbb{N}_0]^{< \infty},\sqsubseteq)$ is a tree (i.e., a partially ordered set such that the set of predecessors of any element is finite and linearly ordered).
\end{ntt}

\begin{dfn}
\label{definition disjointly branching}
A non-empty subcollection $\mathcal{T}$ of $[\mathbb{N}_0]^{<\infty}$ is called a disjointly branching tree if it satisfies the following:
\begin{enumerate}[label=(\alph*),]

\item\label{definition disjointly branching a} For every $A\in\mathcal{T}$ and $B\sqsubseteq A$, $B$ is in $\mathcal{T}$, i.e., $\mathcal{T}$ is downward closed.

\item\label{definition disjointly branching b}  For every $A,B\in\mathcal{T}$, $A\cap B\sqsubseteq A$.

\end{enumerate}
\end{dfn}
The former property yields, among other things, $\emptyset\in\mathcal{T}$. Given \ref{definition disjointly branching a}, the latter property can be equivalently formulated as follows: for every $A\in\mathcal{T}$, denote $S_A = \{l\in\mathbb{N}_0\setminus A: A\cup\{l\}\in\mathcal{T}\}$. Then, for every $A\in\mathcal{T}$, $\min(S_A)>\max(A)$ and the ``branching'' sets $S_A$, $A\in\mathcal{T}$, are pairwise disjoint.

\begin{prop}
\label{what tree isomorphisms look like}
Let $\mathcal{T}$, $\mathcal{S}$ be disjointly branching trees. The following are equivalent:
\begin{enumerate}[label=(\alph*),]

\item\label{what tree isomorphisms look like a} The partially ordered sets $\mathcal{T}$ and $\mathcal{S}$ are isomorphic, i.e., there exists a ``$\sqsubseteq$''-preserving bijection $\Phi:\mathcal{T}\to\mathcal{S}$ whose inverse is also ``$\sqsubseteq$''-preserving.

\item\label{what tree isomorphisms look like b} There exists a bijection $\phi:\cup\mathcal{T}\to\cup\mathcal{S}$ such that $\{ \phi(A) : A\in \mathcal{T} \} = \mathcal{S}$ and, for all $A\in\mathcal{T}$, $\phi|_A:A\to\phi(A)$ is increasing.

\end{enumerate}
More precisely, there exists a correspondence $\Phi\leftrightarrow\phi$ between the set of all isomorphisms from $\mathcal{T}$ to $\mathcal{S}$ and the set of all functions as in \ref{what tree isomorphisms look like b}.
\end{prop}

\begin{proof}
Assume that there exists a ``$\sqsubseteq$''-preserving bijection $\Phi:\mathcal{T}\to\mathcal{S}$ whose inverse is also ``$\sqsubseteq$''-preserving. It is not very hard to see that, for every $n\in\cup\mathcal{T}$, there exists a unique $A\in\mathcal{T}$ such that $n= \max(A)$. We simply define $\phi(n) = \max(\Phi(A))$. For the converse, if a $\phi$ such as in \ref{what tree isomorphisms look like b} exists, we define $\Phi:\mathcal{T}\to\mathcal{S}$ given by $\Phi(A) = \{\phi(k):k\in A\}$.
\end{proof}

\begin{dfn}
\label{disjointly branching tree operations}
We define two operations for constructing disjointly branching trees.
\begin{enumerate}[label=(\alph*),]

\item\label{disjointly branching tree operations a} For a disjointly branching tree $\mathcal{T}$, we define
\[\mathcal{T}^+ = \big\{\{0\}\cup\{k+1:k\in F\}:F\in\mathcal{T}\big\}\cup\{\emptyset\},\]
which is again a disjointly branching tree.

\item\label{disjointly branching tree operations b}

Let $S$ be a non-empty countable set, and let $(\mathcal{T}_s)_{s\in S}$ be a family of disjointly branching trees. We define their disjoint union, which is again a disjointly branching tree, as follows. For each $s\in S$, let $\mathcal{T}_s'$ be a disjointly branching tree isomorphic to $\mathcal{T}_s$, chosen so that the sets $\cup \mathcal{T}_s'$, $s\in S$, are pairwise disjoint, and let
\[\sqcup_{s\in S}\mathcal{T}_s = \cup_{s\in S} \mathcal{T}_s'.\]
For notational simplicity, we suppress the dependence of the disjoint union on the particular choice of copies $(\mathcal{T}'_s)_{s\in S}$, which, for the most part, is immaterial (see \Cref{remark on disjoint union}~\ref{remark on disjoint union b} and \Cref{tree BRS spaces of isomorphic trees}). 

\end{enumerate}
\end{dfn}

\begin{rem}\phantom{a}
\label{remark on disjoint union}
\begin{enumerate}[label=(\alph*)]

\item\label{remark on disjoint union a} The construction in \Cref{disjointly branching tree operations} \ref{disjointly branching tree operations b} is always possible. Fix a pairwise disjoint family $(M_s)_{s\in S}$ of subsets of $\mathbb{N}_0$ such that $\# M_s = \#\cup \mathcal{T}_s$ for all $s\in S$. Let $\phi_s:\cup \mathcal{T}_s\to M_s$ denote the order preserving bijection and define $\mathcal{T}_s'=\{\phi_s(F):F\in\mathcal{T}_s\}$.

\item\label{remark on disjoint union b} The construction of $\sqcup_{s\in S}\mathcal{T}_s$ does not depend on the specific choice of copies $(\mathcal{T}'_s)_{s\in S}$ in the following sense: choosing different copies $(\mathcal{T}''_s)_{s\in S}$ produces a disjoint union which is as a partially ordered set isomorphic to the original one, as follows easily, e.g., from \Cref{what tree isomorphisms look like}.

\end{enumerate}
\end{rem}

\begin{dfn}
\label{the specific tree hierarchy}
For every ordinal number $0\leq \alpha<\omega_1$ we recursively define a disjointly branching tree $\mathcal{T}_\alpha$ as follows:
\begin{enumerate}[,label=(\alph*)]

\item\label{the specific tree hierarchy a} $\mathcal{T}_0 =\{\emptyset\}$,

\item\label{the specific tree hierarchy b} For $0\leq \alpha<\omega_1$, $\mathcal{T}_{\alpha+1} = \mathcal{T}_\alpha^+$,

\item\label{the specific tree hierarchy c} For a limit ordinal number $\alpha<\omega_1$, we define
\[\mathcal{T}_\alpha = \sqcup_{0\leq \beta<\alpha}\mathcal{T}_\beta.\]

\end{enumerate}
\end{dfn}

\begin{exa}\label{exa:ordinal-trees}
\[\mathcal{T}_0 = \{\emptyset\},\;\mathcal{T}_1 = \{\emptyset,\{0\}\},\text{ and }\mathcal{T}_n = \{\emptyset,\{0\},\{0,1\},\ldots,\{0,1,\ldots,n-1\}\},\;n\in\mathbb{N}.\]
The tree $\mathcal{T}_\omega$ depends on the choice of copies $(\mathcal{T}'_n)_{n\in\mathbb{N}_0}$. If we let $M_0 = \emptyset$, $M_1 = \{0\}$, $M_2 = \{1,2\}$, $M_3 = \{3,4,5\}$, $M_4 = \{6,7,8,9\}$, \ldots, and follow \Cref{remark on disjoint union} \ref{remark on disjoint union a} then
\begin{align*}
\mathcal{T}_\omega = &\{\emptyset\}\cup\big\{\{0\}\big\}\cup\big\{\{1\},\{1,2\}\big\}\cup\big\{\{3\},\{3,4\},\{3,4,5\}\big\}\cup\cdots,\\
\mathcal{T}_{\omega+1} = &\{\emptyset\}\cup\big\{\{0\}\big\}\cup\big\{\{0,1\}\big\}\cup\big\{\{0,2\},\{0,2,3\}\big\}\cup\big\{\{0,4\},\{0,4,5\},\{0,4,5,6\}\big\}\cup\cdots,\text{ and}\\
\mathcal{T}_{\omega+2} =& \{\emptyset\}\cup\big\{\{0\}\big\}\cup\big\{\{0,1\}\big\}\cup\big\{\{0,1,2\}\big\}\cup\big\{\{0,1,3\},\{0,1,3,4\}\big\}\cup\\
&\big\{\{0,1,5\},\{0,1,5,6\},\{0,1,5,6,7\}\big\}\cup\cdots.
\end{align*}
This is illustrated in \Cref{fig:trees-T-alpha}.
\vspace{5mm}

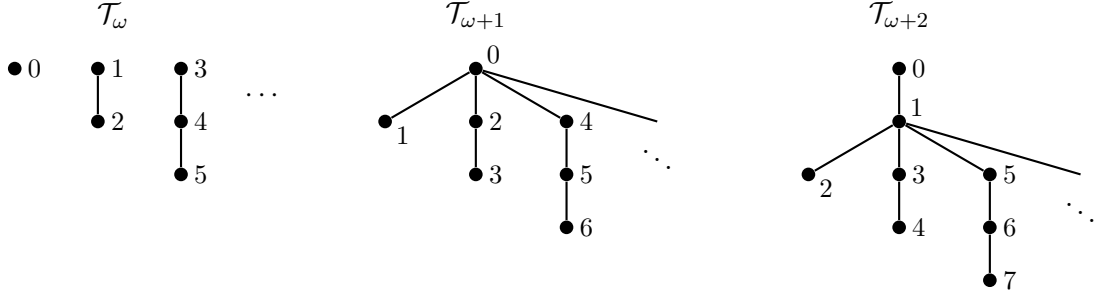
\begin{figure}[t]
  \centering
 \begin{tikzpicture}[
  vertex/.style={circle, fill, inner sep=0pt, minimum size=1.8mm},
  edge/.style={draw, thick},
  plain/.style={draw=none, fill=none, inner sep=0pt, minimum size=0pt},
  every label/.style={plain, font=\small},
]

\node[plain] at (1.3,0.7) {$\mathcal{T}_\omega$};
\node[vertex, label={[xshift=2pt]right:$0$}] (a0) at (0,0) {};
\node[vertex, label={[xshift=2pt]right:$1$}] (a1) at (1.1,0) {};
\node[vertex, label={[xshift=2pt]right:$2$}] (a2) at (1.1,-0.7) {};
\node[vertex, label={[xshift=2pt]right:$3$}] (a3) at (2.2,0) {};
\node[vertex, label={[xshift=2pt]right:$4$}] (a4) at (2.2,-0.7) {};
\node[vertex, label={[xshift=2pt]right:$5$}] (a5) at (2.2,-1.4) {};
\draw[edge] (a1) -- (a2);
\draw[edge] (a3) -- (a4) -- (a5);
\node[plain] at (3.3,-0.35) {$\cdots$};

\node[plain] at (6.1,0.7) {$\mathcal{T}_{\omega+1}$};
\node[vertex, label={[xshift=2pt,yshift=0pt]above right:$0$}] (b0) at (6.1,0) {};
\node[vertex, label={[xshift=2pt,yshift=0pt]below right:$1$}] (b1) at (4.9,-0.7) {};
\node[vertex, label={[xshift=2pt]right:$2$}] (b2) at (6.1,-0.7) {};
\node[vertex, label={[xshift=2pt]right:$3$}] (b3) at (6.1,-1.4) {};
\node[vertex, label={[xshift=2pt]right:$4$}] (b4) at (7.3,-0.7) {};
\node[vertex, label={[xshift=2pt]right:$5$}] (b5) at (7.3,-1.4) {};
\node[vertex, label={[xshift=2pt]right:$6$}] (b6) at (7.3,-2.1) {};
\coordinate (bdots) at (8.5,-0.7);
\draw[edge] (b0) -- (b1);
\draw[edge] (b0) -- (b2) -- (b3);
\draw[edge] (b0) -- (b4) -- (b5) -- (b6);
\draw[edge] (b0) -- (bdots);
\node[plain] at (8.5,-1.1) {$\ddots$};

\node[plain] at (11.7,0.7) {$\mathcal{T}_{\omega+2}$};
\node[vertex, label={[xshift=2pt]right:$0$}] (c0) at (11.7,0) {};
\node[vertex, label={[xshift=2pt,yshift=0pt]above right:$1$}] (c1) at (11.7,-0.7) {};
\node[vertex, label={[xshift=2pt,yshift=0pt]below right:$2$}] (c2) at (10.5,-1.4) {};
\node[vertex, label={[xshift=2pt]right:$3$}] (c3) at (11.7,-1.4) {};
\node[vertex, label={[xshift=2pt]right:$4$}] (c4) at (11.7,-2.1) {};
\node[vertex, label={[xshift=2pt]right:$5$}] (c5) at (12.9,-1.4) {};
\node[vertex, label={[xshift=2pt]right:$6$}] (c6) at (12.9,-2.1) {};
\node[vertex, label={[xshift=2pt]right:$7$}] (c7) at (12.9,-2.8) {};
\coordinate (cdots) at (14.1,-1.4);
\draw[edge] (c0) -- (c1);
\draw[edge] (c1) -- (c2);
\draw[edge] (c1) -- (c3) -- (c4);
\draw[edge] (c1) -- (c5) -- (c6) -- (c7);
\draw[edge] (c1) -- (cdots);
\node[plain] at (14.1,-1.8) {$\ddots$};

\end{tikzpicture}
 
  \caption{The trees $\mathcal{T}_\alpha$ for infinite $\alpha$. The labels along each finite branch form an element of the tree.}
  \label{fig:trees-T-alpha}
\end{figure}

\end{exa}

\begin{rem}
\label{pushing up}
For every $0\leq\alpha<\omega_1$ and $n\in\mathbb{N}$,
\[\mathcal{T}_{\alpha+n} = \mathcal{T}_n\cup\Big\{\{0,\ldots,n-1\}\cup\{k+n:k\in F\}:F\in\mathcal{T}_\alpha\Big\}.\]
\end{rem}

\begin{rem}
The construction of the hierarchy $(\mathcal{T}_\alpha)_{0\leq \alpha<\omega_1}$ depends on a choice made at each limit step. However, different choices produce a hierarchy $(\mathcal{T}'_\alpha)_{0\leq \alpha<\omega_1}$ such that, for every $0\leq \alpha<\omega_1$, the trees $\mathcal{T}_\alpha'$ and $\mathcal{T}_\alpha$ are isomorphic as partially ordered sets.
\end{rem}

For a disjointly branching tree, we now define the associated Bourgain-Rosenthal-Schechtman tree space in a Haar system or dyadic Hardy space and prove that the spaces $R_\alpha^{Y}$ can be realized in this form.

\begin{dfn}
\label{BRS definition trees}
Let $X$ be a Haar system space and $Y = X^0$ or $H_X$. For every disjointly branching tree $\mathcal{T}$, we define the subspace of $Y$
\[R^{Y}_\mathcal{T} = \big[(Y_F)_{F\in\mathcal{T}}\big].\]
\end{dfn}

The following notation assists us in proving that whenever the trees $\mathcal{T}$ and $\mathcal{S}$ are isomorphic then so are the spaces $R_\mathcal{T}^{Y}$ and $R_\mathcal{S}^{Y}$.

\begin{ntt}
\label{simple BRS basis}
Let $\mathcal{T}$ be a disjointly branching tree and let $k\in\cup\mathcal{T}$.
\begin{enumerate}[label=(\alph*),]

\item Denote by $F(k)$ the unique $F\in\mathcal{T}$ with maximum $k$ and define $h(k) = |F(k)|-1$.

\item For $I\in\mathcal{D}^{h(k)}$ define $b_I^k = b_I^{F(k)}$.

\end{enumerate}
\end{ntt}

\begin{prop}\label{prop:basis-of-tree-space}
Let $X$ be a Haar system space and $Y = X^0$ or $H_X$. For any disjointly branching tree  $\mathcal{T}$,
\[R^{Y}_\mathcal{T} = \big[\big((b_I^k)_{I\in\mathcal{D}^{h(k)}}\big)_{k\in\cup\mathcal{T}}\big].\]
\end{prop}

\begin{proof}
For every $k\in\cup\mathcal{T}$ and $I\in\mathcal{D}^{h(k)}$, $b_I^k = b_I^{F(k)}\in Y_{F(k)}\subset R^Y_\mathcal{T}$, which yields ``$\supset$''. The inverse inclusion follows from the easy fact that, for all non-empty $F\in\mathcal{T}$, $Y_F = [(b_I^F)_{I\in\mathcal{D}_{|F|-1}}] = [((b_I^k)_{I\in\mathcal{D}^{h(k)}})_{k\in F}]$.
\end{proof}

\begin{prop}
\label{tree BRS spaces of isomorphic trees}
Let $X$ be a Haar system space. If $\mathcal{T}$ and $\mathcal{S}$ are isomorphic disjointly branching trees then
\[R_\mathcal{T}^{X^0} \equiv^{\mathrm{dist}}R_\mathcal{S}^{X^0}\text{ and }R_\mathcal{T}^{H_X} \equiv^{S\text{-}\mathrm{dist}}R_\mathcal{S}^{H_X}.\]
\end{prop}

\begin{proof}
Let $((b^k_I)_{I\in\mathcal{D}^{h(k)}})_{k\in\cup\mathcal{T}}$ and $((\hat b^k_I)_{I\in\mathcal{D}^{\hat h(k)}})_{k\in\cup\mathcal{S}}$ be the systems from \Cref{simple BRS basis} associated with $\mathcal{T}$ and $\mathcal{S}$. For $k\in\cup \mathcal{T}$ and $m\in\cup\mathcal{S}$, let $F(k)$ and $\hat F(m)$ denote  the unique members of $\mathcal{T}$ and $\mathcal{S}$ with maxima $k$ and $m$.  Let $\phi:\cup\mathcal{T}\to\cup\mathcal{S}$ witness the isomorphism of $\mathcal{T}$ and $\mathcal{S}$ as in \Cref{what tree isomorphisms look like} \ref{what tree isomorphisms look like b}. Then, for every $k\in\cup\mathcal{T}$, $\phi(F(k)) = \hat F(\phi(k))$, and in particular,
\[h(k) = |F(k)| - 1 = |\hat F(\phi(k))| - 1 = \hat h(\phi(k)).\]
Hence, $((\hat b^k_I)_{I\in\mathcal{D}^{\hat h(k)}})_{k\in\cup\mathcal{S}} = ((\hat b^{\phi(k)}_I)_{I\in\mathcal{D}^{h(k)}})_{k\in\cup\mathcal{T}}$ and the correspondence  $b^k_I\mapsto \hat b^{\phi(k)}_I$ extends to a linear bijection between the linear spans of $((b^k_I)_{I\in\mathcal{D}^{h(k)}})_{k\in\cup\mathcal{T}}$ and the preceding system. We will next show that it is a distributional isomorphism.

Let $J:L_0(\sigma(r_k:k\in\cup\mathcal{T}))\to L_0$ denote the distributional embedding given by $r_k\mapsto r_{\phi(k)}$. It is worth noting that, since $\phi$ may fail to be globally increasing, $J$ is a priori not a square function-distributional embedding. However, for every $F\in\mathcal{T}$ and $k\in F$, the restriction $\phi|_F$ is increasing, so
\[Jr_k = r_{\phi(k)} = J_{\phi(F)}J_F^{-1}r_k.\]
Because distributional isomorphisms are homomorphisms, we deduce $J|_{Y_F} = J_{\phi(F)}J_F^{-1}|_{Y_F}$. In particular, for every $k\in\cup\mathcal{T}$ and $I\in \mathcal{D}^{h(k)}$,
\[Jb_I^k = J_{\hat F(\phi(k))}J_{F(k)}^{-1}b_I^k = \hat b_I^{\phi(k)},\]
and thus, $b^k_I\mapsto \hat b^{\phi(k)}_I$ extends to a distributional embedding. To conclude, we will show that it extends to a square function-distributional isomorphism as well.

Let us fix a collection of scalars $((a^k_I)_{I\in\mathcal{D}^{h(k)}})_{k\in\cup\mathcal{T}}$, with only finitely many non-zero terms, and let $f = \sum_{k\in\cup\mathcal{T}}\sum_{I\in\mathcal{D}^{h(k)}}a_I^k b_I^k$. Recall that $J$ is a homomorphism. Hence,
\begin{align*}
S(Jf) &= \Big(\sum_{k\in\cup\mathcal{T}}\big(\sum_{I\in\mathcal{D}^{h(k)}}a_I^k \hat b_I^{\phi(k)}\big)^2\Big)^{1/2} = \Big(\sum_{k\in\cup\mathcal{T}}\big(J\sum_{I\in\mathcal{D}^{h(k)}}a_I^k b_I^{k}\big)^2\Big)^{1/2}\\
      &= \Big(\sum_{k\in\cup\mathcal{T}}J\big(\sum_{I\in\mathcal{D}^{h(k)}}a_I^k b_I^{k}\big)^2\Big)^{1/2} = J\Big(\sum_{k\in\cup\mathcal{T}}\big(\sum_{I\in\mathcal{D}^{h(k)}}a_I^k b_I^{k}\big)^2\Big)^{1/2}\\
  &= JS(f),
\end{align*}
which yields $\mathrm{dist}(S(Jf)) = \mathrm{dist}(S(f))$.
\end{proof}

\begin{rem}
 In the Hardy case $Y = H_X$, although isomorphic trees $\mathcal{T}$ and $\mathcal{T}'$ give rise to square function-distributionally isomorphic subspaces $R_{\mathcal{T}}^Y$ and $R_{\mathcal{T}'}^Y$, the complementation of these spaces in~$Y$ may actually depend on the trees (see \Cref{rem:hardy}~\ref{rem:hardy:b} and \Cref{thm:tree-spaces-uncomplemented}).
\end{rem}

\begin{prop}
\label{showing these are really BRS spaces}
Let $X$ be a Haar system space. Then, for every $0\leq \alpha<\omega_1$,
\begin{equation*}
R_{\alpha}^{X^0} \equiv^\mathrm{dist} R_{\mathcal{T}_\alpha}^{X^0}\text{ and }R_{\alpha}^{H_X} \equiv^{S\text{-}\mathrm{dist}} R_{\mathcal{T}_\alpha}^{H_X}.
\end{equation*}
\end{prop}

\begin{proof}
We focus on the case $Y = H_X$, as the other case is proved identically by modifying the notation. Because $\mathcal{T}_0 = \{\emptyset\}$, we clearly have $R_{\mathcal{T}_0}^{Y} = R^{Y}_0$. We will show the following, which yields the desired conclusion from \Cref{BRS definition sums}:
\begin{enumerate}[label=(\alph*)]

\item\label{showing these are really BRS spaces proof a} For every $\alpha<\omega_1$,
\[R_{\mathcal{T}_{\alpha+1}}^{Y} = \big(R^{Y}_{\mathcal{T}_{\alpha}}\oplus R^{Y}_{\mathcal{T}_{\alpha}}\big)^+_{\mathrm{disj}}.\]

\item\label{showing these are really BRS spaces proof b} For every limit ordinal $\alpha<\omega_1$,
\[R^Y_{\mathcal{T}_{\alpha}}\equiv^{S\text{-}\mathrm{dist}}\big(\oplus_{\beta<\alpha}R^{Y}_{\mathcal{T}_{\beta}}\big)_{Y,\mathrm{ind}}.\]

\end{enumerate}
We will prove the following: for any finite subset $F$ of $\mathbb{N}_0$, letting $G = \{0\}\cup\{k+1:k\in F\}$,
\begin{equation}
\label{showing these are really BRS spaces eq1}
Y_G = \langle h_{[0,1)}\rangle + T_{[0,1/2)}(Y_F) + T_{[1/2,1)}(Y_F).
\end{equation}
Using \eqref{showing these are really BRS spaces eq1}, \cref{showing these are really BRS spaces proof a} is a rather straightforward verification based on \Cref{disjointly branching tree operations} \ref{disjointly branching tree operations a} and \Cref{BRS definition trees}. \Cref{showing these are really BRS spaces proof b} is a similar verification using the corresponding definitions. It, therefore, suffices to prove \eqref{showing these are really BRS spaces eq1}. Because the dimensions of both sides of \eqref{showing these are really BRS spaces eq1} are equal, it suffices to show that the right side is included in the left. Because $0\in G$, we have $h_{[0,1)} = r_0\in Y_G$. It is enough to show that $T_{[0,1/2)}Y_F\subset Y_G$, since the argument for $T_{[1/2,1)}Y_F\subset Y_G$ is identical. As $T_{[0,1/2)}$ is a homomorphism, it suffices to verify that $T_{[0,1/2)}r_k\in Y_G$ for each $k\in F$. Now, $T_{[0,1/2)}r_k = r_{k+1}\chi_{[0,1/2)}$ and this function is $\sigma_G$-measurable because $k+1\in G$, the set $[0,1/2)$ is $r_0$-measurable, and $0\in G$.
\end{proof}

\begin{rem}\label{rem:BRS-equals-tree-BRS}
  Note that in the above proof, specifically in \cref{showing these are really BRS spaces proof b}, the right-hand side can in fact be taken as the definition of $R_\alpha^{Y}$. This shows that, with appropriate choices at limit steps, one may ensure that for all $0 \leq \alpha < \omega_1$, we have $R_\alpha^{Y} = R_{\mathcal{T}_\alpha}^{Y}$. {
In particular, if $\alpha$ is a limit ordinal, we may assume that $R_{\mathcal{T}_\alpha}^{Y} = [(J^{\alpha}_\beta(R_{\mathcal{T}_\beta}^{Y}))_{\beta < \alpha}]$, where $J^{\alpha}_\beta\colon Y\to Y$ are (square function-) distributional embeddings with (square function-) independent ranges, given by $J^{\alpha}_\beta = J_{F^{\alpha}_\beta}$ for some sequence $(F^{\alpha}_\beta)_{\beta < \alpha}$ of pairwise disjoint infinite subsets of~$\mathbb{N}_0$ (see \Cref{rem:faithful-embeddings}).}
\end{rem}

\begin{ntt}
\label{orthogonal projection of BRS}
Let $\mathcal{T}$ be a disjointly branching tree. For every non-empty $F\in\mathcal{T}$, let $F' = F\setminus\{\max(F)\}$. We define the orthogonal projection $P_\mathcal{T}:\langle h_I:I\in\mathcal{D}\rangle \to \langle Y_F:F\in \mathcal{T} \rangle$ by letting
\[P_\mathcal{T}g = \sum_{F\in\mathcal{T}\setminus\{\emptyset\}}\Big(\mathbb{E}_F - \mathbb{E}_{F'}\Big)(g),\]
where $\mathbb{E}_\emptyset$ is the zero operator.
\end{ntt}

\begin{rem}
\label{orthoprojection formulas}
Taking $X = L_2$, $P_\mathcal{T}$ indeed extends to the orthogonal projection onto $R_\mathcal{T}^{X^0}$, and
\[P_\mathcal{T}g  = \sum_{k\in\cup\mathcal{T}}\mathbb{E}_{F(k)}\Big(\mathbb{E}_k-\mathbb{E}_{k-1}\Big)(g) = \sum_{k\in\cup\mathcal{T}}\sum_{I\in\mathcal{D}^{h(k)}}|I|^{-1}\langle b_I^k,g\rangle b_I^k.\]
The first expression follows from expanding $g = \sum_{k=0}^\infty (\mathbb{E}_k - \mathbb{E}_{k-1})g$ (here, $\mathbb{E}_{-1}$ is the zero operator), and, for each $k\in\cup\mathcal{T}$, the existence of a unique $F(k)\in\mathcal{T}$ with $k=\max(F(k))$. The second expression follows from  \Cref{haar basis formulas}~\ref{haar basis formulas a}.
\end{rem}

The following is essentially stated for $H^1$ in \cite[Remark 4, page 223]{bourgain:rosenthal:schechtman:1981}, and could be considered standard (see, e.g., \cite[Theorem 1]{rudin:1962} and \cite{MR114110}). We used ChatGPT~5.4 to retrieve the original argument and assist in adapting it to our setting. We include a proof for completeness, although some details are left to the reader. The argument is based on $R_\mathcal{T}^{Y}$ being translation invariant with respect to the Walsh group structure on $[0,1]$.

\begin{prop}
\label{BRS orthocomplemented or not at all}
Let $X$ be a Haar system space, $Y = X^0$ or $H_X$, and $\mathcal{T}$ be a disjointly branching tree. Then $R_\mathcal{T}^{Y}$ is complemented in $Y$ if and only if $P_\mathcal{T}$ extends to a bounded linear operator on $Y$. More precisely, $\|P_\mathcal{T}\| \leq \|Q\|$, for every linear projection $Q:Y\to Y$ onto $R_\mathcal{T}^{Y}$.
\end{prop}

\begin{proof}
This proof is more naturally stated in the context of the measure space $\{-1,1\}^\mathbb{N}$, instead of $[0,1]$. Nevertheless, we use the language of the latter for consistency. We focus on the case $Y = H_X$, because the case $Y = X^0$ is nearly identical. Assume that $Q:Y\to Y$ is an arbitrary bounded linear projection onto $R_\mathcal{T}^{Y}$.

Let $k\in\mathbb{N}_0$ and let $T_k:\langle\{h_I:I\in\mathcal{D}\}\rangle \to \langle\{h_I:I\in\mathcal{D}\}\rangle$ denote the linear operator such that, for every $I\in\mathcal{D}^k$, $T_kf$ swaps $f|_{I^+}$ with $f|_{I^-}$. Then $T_k$ is an onto distributional isomorphism with $T_k^2 = id$. Note that, for $k,k'\in\mathbb{N}_0$, $T_kT_{k'} = T_{k'}T_k$ and  furthermore, for every finite $F\subset\mathbb{N}_0$, $\mathbb{E}_FT_k = T_k\mathbb{E}_F$. Combined with $T_k$  being a distributional isomorphism, this implies that $T_k$ is a square function-distributional embedding and thus extends to an onto isometry $T_k:Y\to Y$. For every finite $A\subset\mathbb{N}_0$ we define the operator $T_A = \prod_{k\in A}T_k$. Note that, for any finite $A,B\subset \mathbb{N}_0$, $T_AT_B = T_{A\triangle B}$, and in particular $T_A^2 = id$.

For every $n\in\mathbb{N}$, let
\[P_n = \frac{1}{2^n}\sum_{A\subset\{0,\ldots,n-1\}}T_AQT_A.\]
Noting that, for all finite $A\subset\mathbb{N}_0$, $R_\mathcal{T}^{Y}$ is $T_A$-invariant, we obtain that $P_n$ is a bounded linear projection onto $R_\mathcal{T}^{Y}$ of norm at most $\|Q\|$. It can be shown that for every $I\in\mathcal{D}$, $(P_nh_I)_{n\in\mathbb{N}}$ is Cauchy (this is based on the set $\{T_Ah_I:A\in[\mathbb{N}_0]^{<\infty}\}$  being finite and $T_A\mathbb{E}_m = T_{A\cap \{0,\ldots,m\}}\mathbb{E}_m$), and therefore $(P_n)_{n=1}^\infty$ converges in the strong operator topology to a bounded linear projection  $P:Y\to R_\mathcal{T}^{Y}$, with $\|P\|\leq \|Q\|$. For a finite non-empty $A\subset\mathbb{N}_0$, consider the Walsh function $w_A = \prod_{k\in A}r_k$. We will show that, for every such $A$, $w_A$ is an eigenvector of $P$ (with eigenvalue $\lambda_A = 0$ or $1$, because $P$ is a projection). Because the Walsh functions are an orthonormal basis of $L_2^0$ (and their linear span is dense in $Y$), this suggests that $P$ is the orthogonal projection onto the closed linear span of $\{w_A:\lambda_A = 1\}$ (which coincides with $[\{w_A:\emptyset\neq A\subset F, F\in\mathcal{T}\}] = R_\mathcal{T}^{Y}$).

By approximation, it suffices to show that, for a finite non-empty $B\subset\mathbb{N}_0$ and fixed $\max(B)\leq m<n\in\mathbb{N}$,  the Walsh function $w_B$ is an eigenvector of $P_{m,n} = \mathbb{E}_mP_n\mathbb{E}_m$. The Walsh functions $w_A$, $\emptyset\neq A\subset\{0,\ldots,m\}$, form an orthonormal basis of $\langle\{h_I:I\in\mathcal{D}_m\}\rangle$. Expand
\begin{align}
\label{BRS orthocomplemented or not at all eq1}
P_{m,n}w_B = \sum_{\emptyset\neq A\subset\{0,\ldots,m\}}\lambda_Aw_A,
\end{align}
for some unique collection of coefficients $\lambda_A$, $\emptyset\neq A\subset\{0,\ldots,m\}$. For a fixed, but arbitrary, $k\leq m$, note that $T_kP_{m,n} = P_{m,n}T_k$ and, for $\emptyset\neq A\subset\{0,\ldots,m\}$, $T_kw_A = (-1)^{\delta_k(A)}w_A$, where $\delta_k(A) = 1$ if $k\in A$, and~$0$ otherwise. Applying $T_k$ to both sides of \eqref{BRS orthocomplemented or not at all eq1}, we obtain
\[P_{m,n}w_B = \sum_{\emptyset\neq A\subset\{0,\ldots,m\}}\lambda_A(-1)^{\delta_k(A)-\delta_k(B)}w_A.\]
By the uniqueness of the coefficients, $\lambda_A\neq 0$ implies $\delta_k(A) = \delta_k(B)$. Because $k\leq m$ was arbitrary, $A = B\cap \{0,\ldots,m\} =B$. In conclusion, $P_{m,n}w_B = \lambda_{B}w_{B}$.
\end{proof}

\begin{dfn}
\label{boyd def}
For a Haar system space $X$, the lower and upper Boyd indices of $X$ are
\begin{align*}
p_X = \lim_{n\to\infty} \frac{n}{\log_2(\|T^{-1}_{[0,2^{-n})}\|)}\text{ and }q_X = \lim_{n\to\infty} \frac{n}{\log_2(\|T_{[0,2^{-n})}\|^{-1})},
\end{align*}
where the appearing operators are $T_{[0,2^{-n})}\colon X\to X$ and $T^{-1}_{[0,2^{-n})}\colon T_{[0,2^{-n})}(X)\to X$ from \Cref{def  compression}.
\end{dfn}

\begin{rem}\phantom{a}
\label{boyd vs unconditionality}
\begin{enumerate}[label=(\alph*)]

\item\label{boyd vs unconditionality a} The Boyd indices are usually defined using ${[0,\theta)}$ with $\theta\to 0^+$ rather than the discretized limits involving ${[0,2^{-n})}$. This modification does not affect the indices (see, e.g., the \cite[page 131]{lindenstrauss:tzafriri:1979:partII}). Our formulation also accommodates the case $\|\cdot\|_X\asymp \|\cdot\|_\infty$, when $T_{[0,\theta)}$ may be ill-defined.

\item\label{boyd vs unconditionality c} By \cite[Theorem 2.c.6]{lindenstrauss:tzafriri:1979:partII}, the Haar system, including the constant function~$\chi_{[0,1)}$, is an unconditional basis for $X$ if and only if $p_X,q_X\in(1,\infty)$.

\end{enumerate}
\end{rem}

\begin{prop}\label{prop:complementation}
Let $X$ be a Haar system space with non-trivial Boyd indices, i.e., $p_X,q_X\in(1,\infty)$. Then, there exists a constant $C$ such that for every disjointly branching tree $\mathcal{T}$, $\|P_\mathcal{T}\|\leq C$, on $Y = X^0$ or $H_X$.
\end{prop}

\begin{proof}
By \Cref{haar system unconditional in hardy} \ref{haar system unconditional in hardy b} and \Cref{boyd vs unconditionality} \ref{boyd vs unconditionality c}, $\|\cdot\|_X\asymp\|\cdot\|_{H_X}$ and therefore, it suffices to prove the statement for the latter norm. By \cite[Proposition~9.5, p.~243]{johnson:maurey:schechtman:tzafriri:1979}, Stein's inequality holds in $X$, i.e., there exists a constant $C$ such that, for any $f_1,\ldots,f_n$ in $X$ and filtration $\mathcal{B}_1,\ldots,\mathcal{B}_n$ we have
\[\Big\|\Big(\sum_{k=1}^n\big(\mathbb{E}(f_k|\mathcal{B}_k)\big)^2\Big)^{1/2}\Big\|_X \leq C\Big\|\Big(\sum_{k=1}^nf_k^2\Big)^{1/2}\Big\|_X.\]
We will prove that $P_\mathcal{T}$ has norm at most $C^2$.

We will define two total orders on $\cup\mathcal{T}$, ``$\preccurlyeq$'' and ``$\curlyeqprec$''. Let $m,n\in\cup\mathcal{T}$. We will write $m\preccurlyeq n$ if one of the following holds:
\begin{equation}
\begin{gathered}
F(m)\sqsubseteq F(n)\text{ or}\\
F(m)\text{ and }F(n)\text{ are ``$\sqsubseteq$'' incomparable and }\min(F(m)\setminus F(n))<\min(F(n)\setminus F(m)).
\end{gathered}
\end{equation}
On the other hand, we will write $m\curlyeqprec n$ if one of the following holds:
\begin{equation}
\begin{gathered}
F(m)\sqsubseteq F(n)\text{ or}\\
F(m)\text{ and }F(n)\text{ are ``$\sqsubseteq$'' incomparable and }\min(F(m)\setminus F(n))>\min(F(n)\setminus F(m)).
\end{gathered}
\end{equation}
While the transitivity of ``$\preccurlyeq$'' and ``$\curlyeqprec$'' is not entirely obvious, one can convince oneself with a visual representation.

For every $k\in\cup\mathcal{T}$, we let $A(k) = \{m\in\cup\mathcal{T}:m\preccurlyeq k\}$ and $B(k) =  \{m\in\cup\mathcal{T}:m\curlyeqprec k\}$. Some of these sets may be infinite, but $(A(k))_{k\in\cup\mathcal{T}}$ and $(B(k))_{k\in\cup\mathcal{T}}$ are chains of subsets of $\mathbb{N}_0$, such that, for all $k\in\cup\mathcal{T}$, $A(k)\cap B(k) = F(k)$. To conclude, for $g\in  \langle h_I : I\in \mathcal{D} \rangle \subset  H_X$,
\begin{align*}
\|P_\mathcal{T}g\|_{H_X} &= \Big\|\Big(\sum_{k\in\cup\mathcal{T}}\big(\mathbb{E}_{F(k)}\big(\mathbb{E}_k-\mathbb{E}_{k-1}\big)(g)\big)^2\Big)^{1/2}\Big\|_X\\
&= \Big\|\Big(\sum_{k\in\cup\mathcal{T}}\big(\mathbb{E}_{A(k)}\mathbb{E}_{B(k)}\big(\mathbb{E}_k-\mathbb{E}_{k-1}\big)(g)\big)^2\Big)^{1/2}\Big\|_X\\
&\leq C\Big\|\Big(\sum_{k\in\cup\mathcal{T}}\big(\mathbb{E}_{B(k)}\big(\mathbb{E}_k-\mathbb{E}_{k-1}\big)(g)\big)^2\Big)^{1/2}\Big\|_X
\leq C^2\Big\|\Big(\sum_{k\in\cup\mathcal{T}}\big(\big(\mathbb{E}_k-\mathbb{E}_{k-1}\big)(g)\big)^2\Big)^{1/2}\Big\|_X\\
&\leq C^2\|g\|_{H_X}.\qedhere
\end{align*}
\end{proof}

\begin{rem}
In the space $X=L_p$, $1<p<\infty$, Stein's inequality holds with constant $\sqrt{p^*/2}$, where $p^{*} = \max\{ p, p/(p-1) \}$ (see \cite[Theorem~3.1]{asmar:montgomery-smith:1993}). Consequently, in $H^p$ we obtain $\|P_\mathcal{T}\|_{H^p\to H^p} \leq  p^*/2$, and in $L_p$, using \cite[Theorem~8.6]{AdamOse:2012},
\begin{equation}
\label{orthogonal projection norm bound in Lp}
\|P_\mathcal{T}\|_{L_p\to L_p} \leq (p^*-1)^2p^*/2.
\end{equation}
It is already mentioned in \cite[Remark~1, page~222]{bourgain:rosenthal:schechtman:1981} that this bound is achievable, although the authors of that paper chose to prove the weaker bound $(p^*-1)^2(p^*/2)^{3/2}$.

For an arbitrary Haar system space $X$, with $p_X,q_X\in(1,\infty)$, the constant in Stein's inequality is not, a priori, determined solely by $p_X$ and $q_X$.
\end{rem}

\section{An ordinal index detecting complementation}
\label{complemented ordinal section}

For a Haar system space $X$ and $Y = X^0$ or $H_X$, we define an ordinal index $i_Y[\cdot]:\mathcal{SB}\to[0,\omega_1]$ on the class of separable Banach spaces with the following property: for every such $W$, the inequality $i_Y[W]<\omega_1$ implies that $W$ contains no complemented subspace isomorphic to $Y$. If, in addition, $X$ satisfies the Fatou property, then $i_Y[W]=\omega_1$ guarantees the existence in $W$ of a complemented copy of $Y$. The definition of $i_Y[\cdot]$ is conceptually related to the local $L_p$ index from \cite{bourgain:rosenthal:schechtman:1981}, although the formulation differs slightly in its technical presentation. More importantly, $i_Y[\cdot]$ incorporates refinements that allow it to detect complemented copies of $Y$ without imposing structural assumptions on the ambient space $W$. This should be contrasted with the situation where $X=L_p$, or more generally, a reflexive Orlicz function space, and $W$ is a subspace of $Y$, in which case the mere isomorphic embedding of $Y$ into $W$ already yields a complemented copy. Using this, Ghawadrah proved in \cite{ghawadrah:2016} that the Bourgain-Rosenthal-Schechtman subspaces in reflexive Orlicz function spaces give rise to a transfinite hierarchy of pairwise non-isomorphic complemented subspaces. The complemented $Y$-index $i_Y[\cdot]$ allows us to obtain an analogous result, namely \Cref{thm:alternatives-ri}, for the broader class of Haar system spaces with non-trivial Boyd indices satisfying the Fatou property.

The definition of the complemented $Y$-index is based on the height of certain trees in a Banach space $W$ that encode the complexity of the existence of complemented copies of the subspaces $Y_n$ of $Y$ inside $W$. Recall that, for every $n\in\mathbb{N}_0$, $Y_n = \langle\{h_I:I\in\mathcal{D}_n\}\rangle$ and $\mathbb{E}_n:Y\to Y_n$ is the conditional expectation operator.

\begin{dfn}
\label{tree detecting complemented copy definition}
Let $X$ be a Haar system space, $Y = X^0$ or $H_X$, and $W$ be a Banach space. For $C\geq 1$, we define $\mathscr{T}(Y,W,C)$ as the collection of all finite sequences
\[\big((T_0,P_0),\ldots,(T_n,P_n)\big)\]
satisfying the following:
\begin{enumerate}[label=(\alph*)]

\item\label{tree detecting complemented copy definition a} For $0\leq k\leq n$, $T_k:Y_k\to W$ is a linear operator such that, for all $f\in Y_k$,
\[C^{-1}\|f\|\leq \|T_kf\| \leq C\|f\|.\]

\item\label{tree detecting complemented copy definition b} For $1\leq k\leq n$, $T_k|_{Y_{k-1}} = T_{k-1}$.

\item\label{tree detecting complemented copy definition c} For $0\leq k\leq n$, $P_k:W\to W$ is a projection onto $T_k(Y_k)$ of norm at most $C$.

\item\label{tree detecting complemented copy definition d} For all $0\leq k,m\leq n$, $P_kP_m = P_{k\wedge m}$.

\item\label{tree detecting complemented copy definition e} For $1\leq k\leq n$, $P_{k-1}T_k = T_{k-1}\mathbb{E}_{k-1}|_{Y_k}$.

\end{enumerate}
\end{dfn}

We also include the empty sequence $\emptyset$ in $\mathscr{T}(Y,W,C)$ for consistency with disjointly branching trees.
Condition~\ref{tree detecting complemented copy definition d} asserts, in part, that $(P_k|_{T_n(Y_n)})_{k=0}^n$ induces a decomposition of $T_n(Y_n)$ and~\ref{tree detecting complemented copy definition e} that this decomposition mirrors the one generated by $(\mathbb{E}_k)_{k=0}^n$ on $Y_n$.

\begin{rem}
While the formulation differs slightly, if we were to restrict the definition to sequences $(T_0,\ldots,T_n)$ satisfying only \Cref{tree detecting complemented copy definition}~\ref{tree detecting complemented copy definition a} and~\ref{tree detecting complemented copy definition b}, then, upon applying the ensuing \Cref{complemented Y index def} with the obvious modifications, we would recover precisely the local $L_p$ index of \cite{bourgain:rosenthal:schechtman:1981}. The inclusion of projections, together with conditions~\ref{tree detecting complemented copy definition c}--\ref{tree detecting complemented copy definition e}, allows the detection of complemented copies of $Y$ in $W$.
\end{rem}

\begin{ntt}
The collection $\mathscr{T}=\mathscr{T}(Y,W,C)$ in \Cref{tree detecting complemented copy definition} is a tree under the partial order of initial segments. Following \cite{bourgain:1980}, let $\mathscr{T}^{(0)}=\mathscr{T}$, given $\alpha$, let $\mathscr{T}^{(\alpha+1)}$ be the set of all non-maximal elements of $\mathscr{T}^{(\alpha)}$, and for a limit ordinal $\alpha$, set
\[\mathscr{T}^{(\alpha)}=\bigcap_{\beta<\alpha}\mathscr{T}^{(\beta)}.\]
If $\mathscr{T}$ is well-founded, meaning it contains no infinite chains, this transfinite derivation exhausts $\mathscr{T}$ at some ordinal stage. The least such ordinal is a successor ordinal, and its predecessor is called the order of $\mathscr{T}$, denoted by $o[\mathscr{T}]$, i.e.,
\[o[\mathscr{T}]=\min\bigl\{\alpha:\mathscr{T}^{(\alpha)}=\{ \emptyset \}\bigr\}.\]
If $\mathscr{T}$ is not well-founded, it is called ill-founded, and we put $o[\mathscr{T}] = \omega_1$.
\end{ntt}

\begin{rem}
    If we had not allowed the empty sequence in $\mathscr{T}$, we would define $o[\mathscr{T}]=\min\bigl\{\alpha:\mathscr{T}^{(\alpha)}= \emptyset \bigr\}$ as in~\cite{bourgain:1980} to compensate for a shift by one. These changes do not affect whether $\mathscr{T}$ is well-founded; hence, \cite[Proposition~3]{bourgain:1980}, which will play an important role in the next proof, remains valid in our setting.
\end{rem}

\begin{prop}
\label{well-founded iff countable ordinal}
Let $X$ be a Haar system space, $Y = X^0$ or $H_X$, $W$ be a separable Banach space, and $C\geq 1$. If $\mathscr{T}(Y,W,C)$ is well-founded, then $o[\mathscr{T}(Y,W,C)] <\omega_1$.
\end{prop}

\begin{proof}
For a fixed $k\in\mathbb{N}_0$, consider the space $\mathcal{L}(Y_k,W) = \{T:Y_k\to W\text{ linear and bounded}\}$, equipped with the norm topology, and the set $B_{\mathcal{L}(W)} = \{T:W\to W\text{ linear and }\|T\|\leq 1\}$, equipped with the strong operator topology. Since $Y_k$ is finite-dimensional and $W$ is separable, both spaces are Polish, and thus, so is the product
\[S_k =  \mathcal{L}(Y_k,W) \times CB_{\mathcal{L}(W)}.\]
By \cite[Proposition 3]{bourgain:1980}, it suffices to show that, for an arbitrary fixed $n\in\mathbb{N}_0$, the set $\mathscr{T}_n$ consisting of all $\big((T_0,P_0),\ldots,(T_n,P_n)\big)\in\mathscr{T}(Y,W,C)$ is closed in the product space $\prod_{k=0}^n S_k$. Recalling that composition is jointly continuous in the strong operator topology on bounded sets, it follows readily that if $\big((T^i_0,P^i_0),\ldots,(T^i_n,P^i_n)\big)_{i=1}^\infty$ in $\mathscr{T}_n$ converges to $\big((T_0,P_0),\ldots,(T_n,P_n)\big)$, then the limit also belongs to $\mathscr{T}_n$.
\end{proof}

\begin{dfn}
\label{complemented Y index def}
Let $X$ be a Haar system space, $Y = X^0$ or $H_X$, and $W$ be a separable Banach space. The complemented $Y$-index of $W$ is defined as follows:
\begin{equation*}
i_Y[W] = \left\{\begin{array}{ll}\omega_1&\text{ if, for some $C\geq 1$, } \mathscr{T}(Y,W,C)\text{ is ill-founded,}\\
\sup_{C\geq 1}o[\mathscr{T}(Y,W,C)]&\text{ otherwise.}\end{array}\right.
\end{equation*}
\end{dfn}

\begin{rem}\label{rem:index-omega1-ill-founded}
If $1\le C\le D$, then $\mathscr{T}(Y,W,C)\subset \mathscr{T}(Y,W,D)$ and hence $o[\mathscr{T}(Y,W,C)]\le o[\mathscr{T}(Y,W,D)]$. Consequently $i_Y(W)\le \omega_1$, with equality if and only if $\mathscr{T}(Y,W,C)$ is ill-founded for some $C\ge 1$.
\end{rem}

\begin{dfn}
\label{Fatou property definition}
We say that a Haar system space $X$ satisfies the Fatou property if whenever $(f_n)_{n=1}^{\infty}$ is a bounded sequence in $X$ and $0\le f_n \uparrow f$ almost everywhere for some $f\in L_0$, then $f\in X$ and $\|f_n\|_X\uparrow \|f\|_X$.
\end{dfn}

Let us note that if $\|\cdot\|_X \asymp \|\cdot\|_\infty$, then the Haar system space $X$ fails the Fatou property, as demonstrated, for example, by the sequence $f_n=\sum_{i=1}^n \chi_{[2^{-2i},\,2^{-2i+1})}$, $n\in\mathbb{N}$. More generally, a Haar system space $X$ fails the Fatou property if and only if it contains an isomorphic copy of the Banach space $c_0$. We do not prove this fact as it is not used in this paper.

\begin{prop}
\label{Doob's and Hardy analogue}
Let $X$ be a Haar system space satisfying Fatou's property.
\begin{enumerate}[label=(\alph*)]

\item  If $(f_n)_{n=1}^{\infty}$ is a supermartingale in $X$ with $\sup_n \|f_n\|_X < \infty$, then it converges almost everywhere to a function $f\in X$, and $\|f\|_X\le \liminf_n \|f_n\|_X$.

\item The Haar system is a boundedly complete basis for $H_X$.

\end{enumerate}
In particular, every dyadic martingale $(f_n)_{n=1}^\infty$ bounded in $X$ converges almost everywhere to an $f\in X$, and $\|f\|_X\le \liminf_n \|f_n\|_X$. The same holds with $X$ replaced by~$H_X$.
\end{prop}

\begin{proof}
By Doob's convergence theorem for supermartingales bounded in $L_1$, the sequence $(f_n)_{n=1}^{\infty}$ converges almost everywhere to a function $f\in L_1$. Now we apply \cite[Chapter~1, Lemma~1.5~(ii)]{bennett:sharpley:1988} (where Banach function spaces are assumed, by definition, to satisfy the Fatou property), i.e., Fatou's lemma in $X$, to conclude that $f\in X$ and $\|f\|_X \le \liminf_n \|f_n\|_X$.

Enumerate $\mathcal{D}$ lexicographically as $I_1<I_2<\cdots$, and let $(a_n)_{n=1}^\infty$ be scalars such that, letting $f_n = \sum_{k=1}^na_kh_{I_k}$, $\sup_n\|f_n\|_{H_X}<\infty$. Then, $(S(f_n))_{n=1}^\infty$ is a bounded and increasing sequence in $X$. Because $X$ satisfies Fatou's property, its limit, call it $g$, is in $X$. Because $\|\cdot\|_X\not\asymp\|\cdot\|_\infty$, by \cite[Proposition 2.4.7]{speckhofer:2025:thesis}, $\lim_n\|S(f_n)-g\|_X=0$. By the same argument, for every $m\in\mathbb{N}$, the increasing limit of $(S(f_n-f_m))_{n=m}^\infty$, namely $\sqrt{g^2 - S^2(f_m)}$, is in $X$ and $\lim_n\|S(f_n-f_m) - \sqrt{g^2 - S^2(f_m)}\|_X = 0$. Similarly, $\lim_m\|\sqrt{g^2 - S^2(f_m)}\|_X = 0$. In conclusion, $\lim_m\lim_n\|S(f_n-f_m)\|_X = 0$, so $(f_n)_{n=1}^\infty$ is Cauchy in $H_X$.
\end{proof}

\begin{thm}
\label{index main theorem}
Let $X$ be a Haar system space  and let $Y = X^0$ or $H_X$.
\begin{enumerate}[label=(\alph*)]

\item\label{index main theorem a} If a separable Banach space $W$ contains a complemented copy of a Banach space $V$ then
\[i_Y[V]\leq i_Y[W].\]

\item\label{index main theorem b} If $X$ satisfies the Fatou property then a separable Banach space $W$ contains a complemented copy of $Y$ if and only if
\[i_Y[W] = \omega_1.\]

\end{enumerate}
\end{thm}

\begin{proof}
The first assertion is straightforward, and we omit the proof. By \Cref{rem:index-omega1-ill-founded}, the second assertion is equivalent to the ill-foundedness of $\mathscr{T}(Y,W,C)$ for some $C\ge 1$.

Assume there exist an isomorphic embedding $T:Y\to W$ with $\|T\| = 1$ and a bounded linear projection $P:W\to W$ onto $T[Y]$. For $k\in\mathbb{N}_0$ define $T_k=T|_{Y_k}$ and $P_k=T\mathbb{E}_kT^{-1}P$. Then $(T_k,P_k)_{k\in\mathbb{N}_0}$ forms an infinite chain in $\mathscr{T}(Y,W,C)$ with $C=\|T^{-1}\|\|P\|$.

Conversely, let $C\ge 1$ and suppose $\mathscr{T}(Y,W,C)$ contains an infinite chain $(T_k,P_k)_{k\in\mathbb{N}_0}$. By \Cref{tree detecting complemented copy definition} \ref{tree detecting complemented copy definition b}, the operators $T_k$ extend each other on domains whose union is dense in $Y$, and by \Cref{tree detecting complemented copy definition} \ref{tree detecting complemented copy definition a}, they define an isomorphic embedding $T:Y\to W$.

We construct a bounded projection $P:W\to W$ onto $T(Y)$. For $w\in W$ set $f_k=T_k^{-1}P_kw$. By \Cref{tree detecting complemented copy definition} \ref{tree detecting complemented copy definition d} and \ref{tree detecting complemented copy definition e}, $\mathbb{E}_k f_{k+1}=f_k$, $k\in\mathbb{N}_0$, which is immediate upon rewriting \ref{tree detecting complemented copy definition e} as $T_{k-1}^{-1}P_{k-1}T_k=\mathbb{E}_{k-1}|_{Y_k}$. Thus $(f_k)_{k\in\mathbb{N}_0}$ is a bounded dyadic martingale, and by \Cref{Doob's and Hardy analogue}, there exists $f\in X$ (or $H_X$) such that $f_k\to f$ almost everywhere and $\|f\|\leq C^2\|w\|$. Define $Pw=T(f-\mathbb{E}f)$. Then $P:W\to W$ is a linear projection onto $T(Y)$ of norm at most $2C^3$.
\end{proof}

Analogously to the case of trees of finite sequences, for a disjointly branching tree $\mathcal{T}$ we define recursively
$\mathcal{T}^{(0)}=\mathcal{T}$,
$\mathcal{T}^{(\alpha+1)}=\{F \in\mathcal{T}^{(\alpha)} : F \text{ is non-maximal in }\mathcal{T}^{(\alpha)}\}$,
and, for limit ordinals~$\alpha$,
$\mathcal{T}^{(\alpha)}=\cap_{\beta<\alpha}\mathcal{T}^{(\beta)}$.
We then set
$o[\mathcal{T}]=\min\{\alpha:\mathcal{T}^{(\alpha)}=\{\emptyset\}\}$, or $o[\mathcal{T}] = \omega_1$ if $\mathcal{T}$ is ill-founded.
\begin{rem}\label{rem:order-of-ordinal-trees}
  It can be shown by a straightforward transfinite induction that $o[\mathcal{T}_\alpha] = \alpha$ for all $0\le \alpha < \omega_1$.
\end{rem}

\begin{prop}
\label{complementex Y index of tree BRS spaces}
Let $X$ be a Haar system space, $Y = X^0$ or $H_X$, and $\mathcal{T}$ be a disjointly branching tree. Then,
\[i_Y[R_\mathcal{T}^{Y}] \geq o[\mathcal{T}].\]
In particular, for every $0\leq \alpha<\omega_1$, $i_Y[R_\alpha^{Y}] \geq \alpha$.
\end{prop}

\begin{proof}
We define an order-preserving injection from $\mathcal{T}$ into $\mathscr{T}(Y,R_\mathcal{T}^{Y},1)$ as follows. Let $\emptyset\mapsto \emptyset$, and for $F=\{k_0<\cdots<k_n\}\in\mathcal{T}$, set
\[F \mapsto \big((T_0^F,P_0^F),\ldots,(T_n^F,P_n^F)\big),\]
where for $0\le m\le n$,
\begin{align*}
T_m^F &= J_{\{k_0,\ldots,k_m\}}:Y_m\to Y_{\{k_0,\ldots,k_m\}}\subset R_\mathcal{T}^{Y}\text{ and }\\
P_m^F &= \mathbb{E}_{\{k_0,\ldots,k_m\}}:R_\mathcal{T}^{Y}\to Y_{\{k_0,\ldots,k_m\}}.
\end{align*}

By \Cref{norm-one projection and isometry}, each $T_m^F$ is an isometry and each $P_m^F$ is a norm-one projection. Verification of \Cref{tree detecting complemented copy definition}~\ref{tree detecting complemented copy definition a}--\ref{tree detecting complemented copy definition e} is routine. For instance, condition \ref{tree detecting complemented copy definition e} reduces, for $1\le m\le n$, to
\[\mathbb{E}_{\{k_0,\ldots,k_{m-1}\}} J_{\{k_0,\ldots,k_m\}}=J_{\{k_0,\ldots,k_m\}} \mathbb{E}_{m-1}|_{Y_m},\]
which is immediate on $(h_I)_{I\in\mathcal{D}_m}$.

The ``in particular'' statement follows from the fact that $R_\alpha^{Y}$ is isometric to $R_{\mathcal{T}_\alpha}^{Y}$ (see \Cref{showing these are really BRS spaces}), together with \Cref{rem:order-of-ordinal-trees}.
\end{proof}

\begin{rem}
Let $X$ be a Haar system space and $Y = X^0$ or $H_X$. Then $(i_Y[R_\alpha^{Y}])_{\alpha<\omega_1}$ is a non-decreasing transfinite sequence of ordinals with supremum $\omega_1$. Indeed, a simple transfinite induction easily yields that, for every $\beta<\alpha<\omega_1$, $R_\alpha^{Y}$ contains a complemented subspace isomorphic to $R_\beta^{Y}$, and therefore, by \Cref{index main theorem} \ref{index main theorem a}, $i_Y[R_\beta^{Y}] \leq i_Y[R_\alpha^{Y}]$. The second part is from \Cref{complementex Y index of tree BRS spaces}.
\end{rem}

\section{Complemented embeddings into spaces with unconditional Schauder decompositions}
\label{sec:complemented-embeddings}

The aim of this section is to extend the result of Bourgain, Rosenthal, and Schechtman on complemented embeddings of $L_p$ into spaces with unconditional Schauder decompositions \cite[Theorem~1.1]{bourgain:rosenthal:schechtman:1981} to the setting of Haar system spaces and their associated Hardy spaces.
Recall that an unconditional Schauder decomposition of a Banach space~$Z$ is a sequence $(Z_n)_{n=1}^{\infty}$ of closed subspaces of~$Z$ such that every~$z\in Z$ admits a unique norm-convergent representation $z = \sum_{n=1}^{\infty} z_n$, $z_n\in Z_n$, and this series converges unconditionally.
We will prove the following theorem.

\begin{thm}\label{thm:1.1}
  Let $X$ be a Haar system space and $Y = X^0$ or $H_X$. Suppose that the Haar system is unconditional in $Y$.
  If $Y$ is isomorphic to a complemented subspace of a Banach space~$Z$ with an unconditional Schauder decomposition $(Z_i)_{i=1}^{\infty}$, then one of the following holds:
\begin{enumerate}[label=(\alph*),]
    \item\label{thm:1.1:a} There is an $i$ such that $Y$ is isomorphic to a complemented subspace of $Z_i$.
    \item A block basic sequence of the $Z_i$'s is equivalent to the Haar system in $Y$ and has closed linear span complemented in $Z$.
  \end{enumerate}
\end{thm}

Note that if $Y = H_X$, then the Haar system is always unconditional in~$Y$.
The above theorem has been proved in~\cite{bourgain:rosenthal:schechtman:1981} for the spaces $L_p$, $1 < p < \infty$, using structural properties of the vector-valued space $L_p(\ell_2)$, which is isomorphic to $L_p$. Ghawadrah~\cite{ghawadrah:2016} later obtained the same result for rearrangement-invariant spaces with non-trivial Boyd indices, using interpolation.
While we still follow the overarching proof structure of \cite{bourgain:rosenthal:schechtman:1981}, instead of passing to $X(\ell_2)$, our proof relies only on the space $Y$ and the Haar system itself, and moreover, it does not use interpolation. We employ methods that were originally developed to prove factorization results~\cite{lechner:speckhofer:2025} and are applicable to Haar system spaces and their Hardy versions alike.
The assumption that the Haar system is unconditional in $Y$ will only be needed in the final step of the proof, in particular to apply \Cref{lem:1.3} below. The other intermediate results hold in a more general setting (see \Cref{rem:assumptions-factorization-property++}).

In~\cite{bourgain:rosenthal:schechtman:1981}, it is mentioned that \Cref{thm:1.1} can also be proved for the space $L_1$, in which case only alternative~\ref{thm:1.1:a} occurs. The proof is based on the work of Enflo and Starbird~\cite{enflo:starbird:1979}. However, these methods are quite specific to the $L_1$ setting, and it is unclear whether the unconditionality of the Haar system can be dropped entirely in \Cref{thm:1.1}.

We will use the following diagonalization result, which is stated in \cite[Lemma~1.3]{bourgain:rosenthal:schechtman:1981} and follows from Tong's diagonalization theorem~\cite{MR262855} (see also \cite[Proposition 1.c.8]{lindenstrauss:tzafriri:1977}).

\begin{lem}\label{lem:1.3}
  Let $Y$ and $Z$ be Banach spaces with unconditional Schauder decompositions $(Y_i)_{i=1}^{\infty}$ and $(Z_i)_{i=1}^{\infty}$, respectively, and let $(P_i)_{i=1}^{\infty}$ (resp.\ $(Q_i)_{i=1}^{\infty}$) be the natural projections from~$Y$ (resp.\ $Z$) onto $Y_i$ (resp.\ $Z_i$). Then, if $T\colon Y\to Z$ is a bounded linear operator, so is $\widetilde{T} := \sum_{i=1}^{\infty} Q_i T P_i$. (In other words, there is a $K < \infty$ so that for all $x\in Y$, $\sum_{i=1}^{\infty} Q_i T P_i x$ converges and $\|\sum_{i=1}^{\infty} Q_i T P_i x \|\le  K \| x\|$. Here, $K$ is the product of $\|T\|$ and the unconditional constants of $(Y_i)_{i=1}^{\infty}$ and $(Z_i)_{i=1}^{\infty}$.)
\end{lem}

In the remainder of this section, we will use terminology from the factorization theory of operators on Banach spaces.
\begin{dfn}
    Let $S,T$ be bounded linear operators on a Banach space~$X$. We say that $S$ factors through~$T$ if there exist bounded linear operators $A,B\colon X\to X$ such that $S = ATB$.
  \end{dfn}
\begin{rem}\label{rem:factorization-subspace}
  Note that if the identity $I_X$ on a Banach space~$X$ factors through a bounded linear operator $T\colon X\to X$, then there exists a subspace~$Y$ of~$X$ with $Y$ isomorphic to~$X$, $T|_Y$ an isomorphism, and $T(Y)$ complemented in~$X$.
\end{rem}
\begin{dfn}
    We say that a Banach space~$X$ has the primary factorization property if for every bounded linear operator~$T\colon X\to X$, the identity~$I_X$ factors either through~$T$ or through~$I_X - T$.
\end{dfn}
\begin{dfn}
    Let $(e_n)_{n=1}^{\infty}$ be a Schauder basis of a Banach space~$X$ with biorthogonal functionals $(e_n^{*})_{n=1}^{\infty}$. We say that~$(e_n)_{n=1}^{\infty}$ has the factorization property if the identity~$I_X$ factors through every bounded linear operator~$T\colon X\to X$ which has large diagonal with respect to~$(e_n)_{n=1}^{\infty}$, i.e., $\inf_n |\langle e_n^{*}, Te_n \rangle| > 0$.
\end{dfn}

The next result is proved in \cite[Theorem~1.9]{bourgain:rosenthal:schechtman:1981} for $Y = L_p$, but in fact, the same proof works for every Banach space $Y$ which has the primary factorization property.

\begin{prop}\label{prop:pfp-complemented-subspace-direct-sum}
  Let $Y,Z_1,\dots,Z_n$ be Banach spaces. Suppose that $Y$ has the primary factorization property and is isomorphic to a complemented subspace of $Z_1\oplus \dots \oplus Z_n$. Then $Y$ is isomorphic to a complemented subspace of $Z_i$ for some $i$.
\end{prop}

\begin{dfn}\label{dfn:faithful-HS}
  For every $I\in \mathcal{D}$, let $\mathcal{B}_I$ be a non-empty finite subcollection of $\mathcal{D}$. Moreover, let $(\theta_K)_{K\in \mathcal{D}}\in \{ \pm 1 \}^{\mathcal{D}}$ be a family of signs, and let $(\eta_I)_{I\in \mathcal{D}}$ be a family of real numbers with $0\le \eta_I < 1$ for all $I$. For $I\in \mathcal{D}$, put $\tilde{h}_I = \sum_{K\in \mathcal{B}_I} \theta_Kh_K$. We say that $(\tilde{h}_I)_{I\in \mathcal{D}}$ is an $(\eta_I)_{I\in \mathcal{D}}$-faithful Haar system if the following conditions are satisfied:
\begin{enumerate}[label=(\alph*),]
    \item\label{cond:faithful-i-copy} For every $I\in \mathcal{D}$, the collection $\mathcal{B}_I$ consists of pairwise disjoint dyadic intervals, and $\mathcal{B}_I\cap \mathcal{B}_J = \emptyset$ for all $I\ne J\in \mathcal{D}$.
    \item\label{cond:faithful-ii-copy} For every $I\in \mathcal{D}$, we have $\mathcal{B}_{I^{\pm }}^{*}\subset [\tilde{h}_I = \pm 1]$ and $|\mathcal{B}_{I^{\pm }}^{*}|\ge (1-\eta_I) \cdot \tfrac12 |\mathcal{B}_I^{*}|$.
  \end{enumerate}
  If $\mathcal{B}_{[0,1)}^{*} = [0,1)$ and $\eta_I = 0$ for all $I\in \mathcal{D}$, we call $(\tilde{h}_I)_{I\in \mathcal{D}}$ a faithful Haar system.
  See \Cref{fig:almost-faithful-and-faithful} for examples.
\end{dfn}

\begin{prop}\label{thm:operators-AB-copy}
  Let $X$ be a Haar system space and $Y = X^0$ or $H_X$, and suppose that $(\tilde{h}_I)_{I\in \mathcal{D}}$ is an $(\eta_I)_{I\in \mathcal{D}}$-faithful Haar system for some family of real numbers $0 \le\eta_I < 1$ such that
  \begin{equation}\label{eq:3}
    \sigma := \sum_{I\in \mathcal{D}} \eta_I < 1.
  \end{equation}
  Then the operators $A,B\colon Y\to Y$ given by
  \begin{equation}\label{eq:9}
    Bx = \sum_{I\in \mathcal{D}} \frac{\langle h_I, x \rangle}{|I|} \tilde{h}_I\qquad \text{and} \qquad
    Ax = \sum_{I\in \mathcal{D}} \frac{\langle \tilde{h}_I, x \rangle}{|\mathcal{B}_I^{*}|} h_I
  \end{equation}
  are well-defined and bounded, and they satisfy $AB = I_Y$. In particular, $(\tilde{h}_I)_{I\in \mathcal{D}}$ is equivalent to $(h_I)_{I\in \mathcal{D}}$ in $Y$, and $[\tilde{h}_I]_{I\in \mathcal{D}}$ is complemented in $Y$ by the projection $BA$.
  Moreover, we have $|\mathcal{B}_I^{*}| \ge (1-\sigma)|\mathcal{B}_{[0,1)}^{*}| |I|$ for all $I\in \mathcal{D}$.
\end{prop}
\begin{proof}
  This is proved in \cite[Theorem~7.3]{lechner:speckhofer:2025} (see also \cite[Theorem 3.5.3]{speckhofer:2025:thesis} for more details). In our main results, the most important case is $Y = X^0$, where $X$ is a Haar system space in which the Haar system is unconditional. In this case, the proof is much shorter, and we have chosen to include it for convenience.

  First, we assume that $(\tilde{h}_I)_{I\in \mathcal{D}}$ is a faithful Haar system.
  Then $B$ is a distributional isomorphism on $\langle\{ h_I : I\in \mathcal{D} \}  \rangle$ and hence an isometry on~$X^0$.
  We can extend the operators $B$ and~$A$ from $\langle \{ h_I : I\in \mathcal{D} \} \rangle$ to $\langle \{ \chi_I : I\in \mathcal{D} \} \rangle$ as follows: For $\alpha\in \mathbb{R}$, $x_0\in \langle \{ h_I : I\in \mathcal{D} \} \rangle$, and $x = \alpha \chi_{[0,1)} + x_0$, we put $\bar{B}x = \alpha \chi_{[0,1)} + Bx_0$ and $\bar{A}x = \alpha \chi_{[0,1)} + Ax_0$. Then the boundedness of $\bar{A}$, and hence of $A$, follows by duality: By \Cref{pro:HS-1d}~\ref{pro:HS-1d:iii}, $\bar{B}$ is also an isometry in the $X^{*}$-norm,
  and for all $x,y\in \langle \{ \chi_{[0,1)} \}\cup  \{ h_I : I\in \mathcal{D} \} \rangle$, we have $\langle y, \bar{A}x \rangle = \langle \bar{B}y, x \rangle$.

  If $(\tilde{h}_I)_{I\in \mathcal{D}}$ is not faithful, then we can extend it to a faithful Haar system $(\hat{h}_I)_{I\in \mathcal{D}}$ with Haar supports~$\hat{\mathcal{B}}_I$ such that $\mathcal{B}_I\subset \hat{\mathcal{B}}_I$ for all $I\in \mathcal{D}$ (such a system can be constructed inductively: if $\hat{h}_{\pi(I)}$ has already been defined for some $I\in \mathcal{D}$, we extend the collection $\mathcal{B}_I$ to $\hat{\mathcal{B}}_I$ by inserting sufficiently small, disjoint dyadic intervals which precisely cover the set $[\hat{h}_{\pi(I)} = \varepsilon(I)]\setminus \mathcal{B}_I^{*}$, and which do not yet appear in any previously constructed collection; see \Cref{fig:almost-faithful-and-faithful}).
  Now let $\hat{A},\hat{B}$ denote the operators associated with $(\hat{h}_I)_{I\in \mathcal{D}}$. We already know that they are bounded. Moreover, let $\mathcal{C} = \bigcup_{I\in \mathcal{D}} \mathcal{B}_I$, and let $R$ denote the Haar multiplier defined by $R h_K = h_K$ for $K\in \mathcal{C}$ and $Rh_k = 0$ for $K\in \mathcal{D}\setminus \mathcal{C}$. Thus, $R \hat{h}_I = \tilde{h}_I$ for all $I\in \mathcal{D}$, and by unconditionality, $R$ is bounded on $X$. Now we can write $B = R \hat{B}$ and $A = M \hat{A} R$, where $M$ is the linear extension of
  \begin{equation*}
    Mh_I = \frac{|I|}{|\mathcal{B}_I^{*}|} h_I,\qquad I\in \mathcal{D}.
  \end{equation*}
  Hence, it only remains to prove that $\sup_{I\in \mathcal{D}} |I|/|\mathcal{B}_I^{*}| < \infty$.
  This follows by repeatedly applying \Cref{dfn:faithful-HS}~\ref{cond:faithful-ii-copy}: For $I\in \mathcal{D}$ with $|I| = 2^{-n}$, we have
  \begin{equation*}
    |\mathcal{B}_I^{*}| \ge \Bigl( \prod_{J\in \mathcal{D}: J\supsetneq I} (1-\eta_J) \Bigr) \frac{|\mathcal{B}_{[0,1)}^{*}|}{2^n}
    \ge \Bigl( 1 - \sum_{J\in \mathcal{D}: J\supsetneq I} \eta_J \Bigr) \frac{|\mathcal{B}_{[0,1)}^{*}|}{2^n} \ge (1-\sigma) |\mathcal{B}_{[0,1)}^{*}| |I|.
  \end{equation*}
  Thus, $A$ and $B$ extend to bounded linear operators on $X^0$, given by the two norm-convergent series in~\eqref{eq:9}.
\end{proof}

\begin{figure}[t]
  \centering \includegraphics[width=\textwidth]{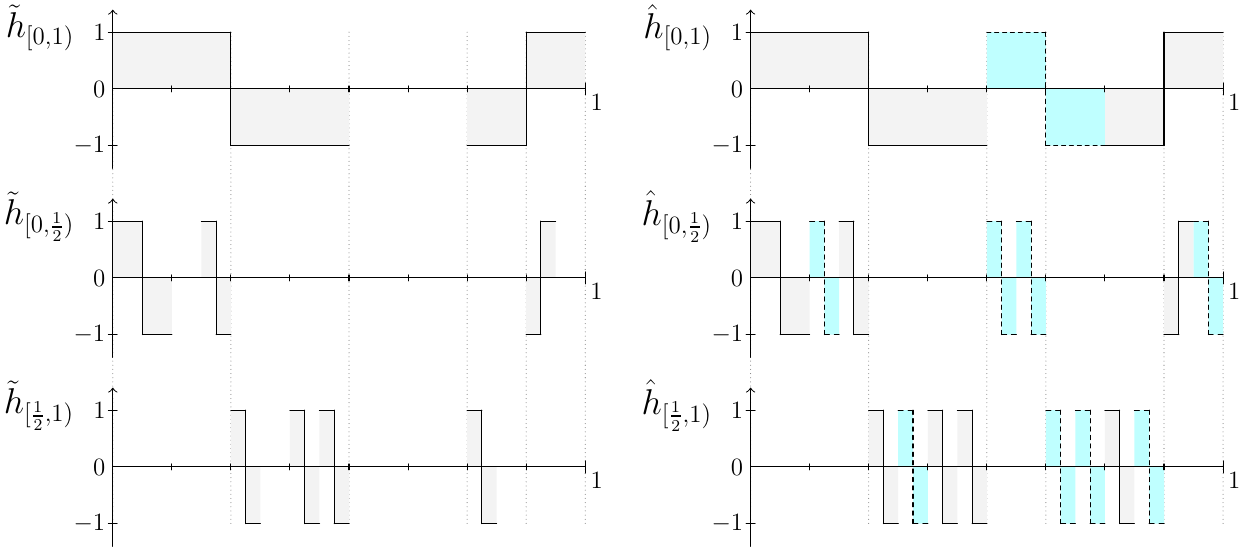}
  \caption{An $(\eta_I)_{I\in \mathcal{D}}$-faithful Haar system (left) is extended to a faithful Haar system (right).}
  \label{fig:almost-faithful-and-faithful}
\end{figure}

In order to construct an ``almost'' faithful Haar system that almost diagonalizes a given operator, we will need the following lemma.
\begin{lem}\label{lem:close-to-diagonal-operator}
  Let $X$ be a Haar system space and $Y = X^0$ or $H_X$, and let $(\tilde{h}_I)_{I\in \mathcal{D}}$ be an $(\eta_I)_{I\in \mathcal{D}}$-faithful Haar system for some $(\eta_I)_{I\in \mathcal{D}}$, with Haar supports $\mathcal{B}_I$. Suppose that $\sigma < 1$, where $\sigma$ is defined as in~\eqref{eq:3}. Moreover, let $\varepsilon > 0$, and let $(\varepsilon_{I,J})_{I,J\in \mathcal{D}}$ be a family of positive numbers such that
  \begin{equation}\label{eq:13}
    \sum_{I,J\in \mathcal{D}} \frac{\varepsilon_{I,J}}{|I| |J|} < \varepsilon.
  \end{equation}
  If $T\colon Y\to Y$ is a bounded linear operator and $|\langle \tilde{h}_I, T \tilde{h}_J \rangle|\le \varepsilon_{I,J}$ for all $I\ne J\in \mathcal{D}$, then the linear extension of $Dh_I = |\mathcal{B}_I^{*}|^{-1}\langle \tilde{h}_I, T \tilde{h}_I\rangle  h_I$, $I\in \mathcal{D}$,
  defines a bounded Haar multiplier on~$Y$ such that $\|ATB - D\| \le \varepsilon (1-\sigma)^{-1}|\mathcal{B}_{[0,1)}^{*}|^{-1}$, where $A$ and~$B$ are the associated operators from \Cref{thm:operators-AB-copy}.
\end{lem}
\begin{proof}
  For every $J\in \mathcal{D}$, we have
  \begin{equation*}
    \|(ATB - D) h_J\|_Y = \Bigl\| \sum_{\substack{I\in \mathcal{D}\\ I\ne J}} \frac{\langle \tilde{h}_I, T \tilde{h}_J \rangle}{|\mathcal{B}_I^{*}|}h_I \Bigr\|_Y
    \le \sum_{\substack{I\in \mathcal{D}\\ I\ne J}} \frac{|\langle \tilde{h}_I, T \tilde{h}_J \rangle|}{|\mathcal{B}_I^{*}|} \|h_I\|_Y
    \le \sum_{\substack{I\in \mathcal{D}\\ I\ne J}} \frac{\varepsilon_{I,J}}{|\mathcal{B}_I^{*}|}.
  \end{equation*}
  For every finite linear combination $x = \sum_{J\in \mathcal{D}} a_Jh_J\in Y$, we have $|a_J| = |J|^{-1} |\langle h_J,x \rangle|\le |J|^{-1}\|x\|_Y$ for all $J$ (note that $\|h_J\|_{Y^{*}}\le 1$ by \Cref{pro:HSHS-1d}~\ref{pro:HSHS-1d:v}). Thus, we obtain
  \begin{equation*}
    \|(ATB - D)x\|_Y
    \le \sum_{J\in \mathcal{D}} |a_J|\|(ATB - D)h_J\|_Y
    \le \|x\|_Y \sum_{I,J\in \mathcal{D}} \frac{\varepsilon_{I,J}}{|\mathcal{B}_I^{*}| |J|}
    \le \frac{\varepsilon \|x\|_Y}{(1-\sigma)|\mathcal{B}_{[0,1)}^{*}|},
  \end{equation*}
  where we used~\eqref{eq:13} and the last statement of \Cref{thm:operators-AB-copy}.
\end{proof}

Next, we establish a slightly stronger version of the factorization property of the Haar basis of~$Y$, along with some other preliminary results needed to prove \Cref{thm:1.1}.

\begin{dfn}
  Let $\mathcal{A}\subset \mathcal{D}$ be a collection of dyadic intervals.
  By $\limsup \mathcal{A}$ we denote the set of all points $t\in [0,1)$ which are contained in infinitely many intervals $I\in \mathcal{A}$.
  Moreover, we define
  \begin{equation*}
    \mathcal{G}_0(\mathcal{A}) = \{ I\in \mathcal{A} : I \text{ is maximal with respect to inclusion} \}
  \end{equation*}
  and, for $n\ge 1$,
  \begin{equation*}
    \mathcal{G}_n(\mathcal{A}) = \mathcal{G}_0\Bigl( \mathcal{A}\setminus \bigcup_{k=0}^{n-1}\mathcal{G}_k(\mathcal{A}) \Bigr).
  \end{equation*}
  Note that the ``generations'' $\mathcal{G}_n(\mathcal{A})$ are pairwise disjoint collections and each of them consists of pairwise disjoint intervals. If $n\ge 1$ and $I\in \mathcal{G}_n(\mathcal{A})$, then there exists a unique interval $J\in \mathcal{G}_{n-1}(\mathcal{A})$ such that $I\subset J^+$ or $I\subset J^-$. In particular, $\mathcal{G}_0(\mathcal{A})^{*}\supset \mathcal{G}_1(\mathcal{A})^{*}\supset \cdots$.
  See \Cref{fig:generations} for an example.
  Also note that $\mathcal{G}_n(\mathcal{A})^{*}$ is the set of points in $[0,1)$ contained in at least $n+1$ intervals from $\mathcal{A}$, hence the intersection of all sets $\mathcal{G}_n(\mathcal{A})^{*}$ is $\limsup(\mathcal{A})$.
\end{dfn}

\begin{figure}[t]
  \centering \includegraphics[width=0.85\textwidth]{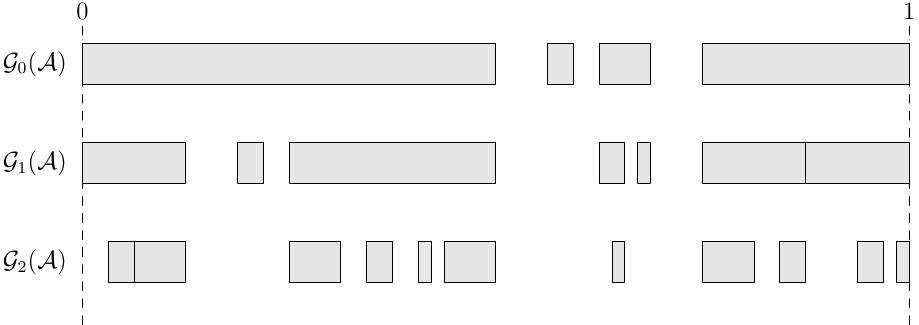}
  \caption{The collections $\mathcal{G}_n(\mathcal{A})$}
  \label{fig:generations}
\end{figure}

The next lemma contains the main inductive step in the construction of an almost faithful Haar system and will be applied twice in the remainder of this section. Essentially, it asserts that it is always possible to construct a block of the Haar system (with constraints on its support and Haar support) that almost annihilates finitely many vectors and functionals.
\begin{lem}\label{lem:choice-of-B_I}
  Let $X$ be a Haar system space and $Y = X^0$ or $H_X$. Suppose that the sequence of Rademacher functions is weakly null in $Y$. Let $\mathcal{A}\subset \mathcal{D}$, let $S = \limsup(\mathcal{A})$, and let $\Gamma \subset [0,1)$ be a finite union of dyadic intervals such that $|\Gamma \cap S| > (1 - \tilde{\eta})|\Gamma|$ for some $0 < \tilde{\eta} < 1$. Moreover, let $\mathcal{E}\subset \mathcal{D}$, $F\subset Y$, and $G\subset Y^{*}$ be finite, and let $\eta,\varepsilon > 0$. Then there exists a finite collection $\mathcal{B}\subset \mathcal{A}\setminus \mathcal{E}$ of pairwise disjoint dyadic intervals with the following properties:
\begin{enumerate}[label=(\alph*),]
    \item\label{lem:choice-of-B_I:a} $\mathcal{B}^{*}\subset \Gamma$
    \item\label{lem:choice-of-B_I:b} $|\mathcal{B}^{*}| > (1 - \tilde{\eta}) |\Gamma|$
    \item\label{lem:choice-of-B_I:c} $|\mathcal{B}^{*}\cap S| > (1 - \eta) |\mathcal{B}^{*}|$
    \item\label{lem:choice-of-B_I:d} For every choice of signs $\theta\in \{ \pm 1 \}^{\mathcal{B}}$, the function $\tilde{h} = \sum_{K\in \mathcal{B}} \theta_Kh_K$ satisfies
  \begin{equation}
    \max_{g\in G} |\langle g, \tilde{h} \rangle| < \varepsilon
    \qquad \text{and} \qquad
    \max_{f\in F} |\langle \tilde{h}, f \rangle| < \varepsilon.
  \end{equation}
  \end{enumerate}
\end{lem}
\begin{proof}
  The collection $\mathcal{C} = \{ K\in \mathcal{A} : K\subset \Gamma \}$ satisfies $\limsup(\mathcal{C}) = \Gamma \cap S$, so for all $n\in \mathbb{N}_0$, we have
  \begin{equation}\label{eq:1}
    |\mathcal{G}_n(\mathcal{C})^{*}| \ge |\limsup \mathcal{C}| = |\Gamma \cap S| > (1 - \tilde{\eta})|\Gamma|.
  \end{equation}
  Moreover, for all sufficiently large $n$, we have $|\mathcal{G}_n(\mathcal{C})^{*}\cap S| > (1 - \eta) |\mathcal{G}_n(\mathcal{C})^{*}|$.
  If $Y = X^0$, then by \cite[Lemma~4.3]{lechner:motakis:mueller:schlumprecht:2022}, there exists $K_0\in \mathcal{D}$ such that for any finite collection $\mathcal{B}\subset \mathcal{D}$ of pairwise disjoint dyadic intervals with $K\ge K_0$ for all $K\in \mathcal{B}$ and every choice of signs $\theta\in \{ \pm 1 \}^{\mathcal{B}}$, we have
 \begin{equation}\label{eq:2}
    \max_{g\in G}\,  \Bigl| \Bigl\langle g, \sum_{K\in \mathcal{B}} \theta_Kh_K \Bigr\rangle \Bigr| <  \varepsilon
    \qquad \text{and} \qquad
    \max_{f\in F}\,  \Bigl| \Bigl\langle \sum_{K\in \mathcal{B}} \theta_Kh_K, f \Bigr\rangle \Bigr| <  \varepsilon.
  \end{equation}
  The proof in~\cite{lechner:motakis:mueller:schlumprecht:2022} also carries over to the case $Y = H_X$.
  By choosing $n$ large enough, we may assume that $K \ge K_0$ for all $K\in \mathcal{G}_n(\mathcal{C})$ and that $\mathcal{G}_n(\mathcal{C})\subset \mathcal{A}\setminus \mathcal{E}$. Now we take $\mathcal{B}$ to be a finite subcollection of $\mathcal{G}_n(\mathcal{C})$ which still satisfies $|\mathcal{B}^{*} \cap S| > (1 - \eta) |\mathcal{B}^{*}|$ as well as $|\mathcal{B}^{*}| > (1 - \tilde{\eta}) |\Gamma|$. Then properties~\ref{lem:choice-of-B_I:a}--\ref{lem:choice-of-B_I:c} are satisfied, and~\ref{lem:choice-of-B_I:d} follows from~\eqref{eq:2}.
\end{proof}

\begin{rem}
  The assumption that the sequence of Rademacher functions is weakly null in~$Y$ is always satisfied if $Y = H_X$. If $Y = X^0$, then this assumption is equivalent to $\|\cdot \|_X \not\asymp \|\cdot \|_{\infty}$ (see \cite[Proposition~2.3.2]{speckhofer:2025:thesis}).
\end{rem}

The next result is a slightly stronger version of the factorization property of the Haar basis in~$Y$.

\begin{thm}\label{thm:factorization-property++}
  Let $X$ be a Haar system space and $Y = X^0$ or $H_X$, and suppose that the Haar system is unconditional in~$Y$. Let
  $T\colon Y\to Y$ be a bounded operator, and suppose that there exists $\delta > 0$ such that if
  \begin{equation*}
    \mathcal{A} = \{ I\in \mathcal{D} : |\langle h_I, T h_I \rangle| \ge \delta |I| \},
  \end{equation*}
  then $\limsup \mathcal{A}$ has positive Lebesgue measure.
  Then the identity on $Y$ factors through $T$. In particular, there exists a subspace $Z$ of $Y$ with $Z$ isomorphic to $Y$, $T|_Z$ an isomorphism, and $T(Z)$ complemented in $Y$.
\end{thm}
\begin{proof}
  Define
  \begin{equation*}
    \mathcal{A}_1 = \{ I\in \mathcal{D} : \langle h_I, T h_I \rangle \ge \delta |I| \}
     \qquad\text{and}\qquad
    \mathcal{A}_2 = \{ I\in \mathcal{D} : \langle h_I, T h_I \rangle \le -\delta |I| \}.
  \end{equation*}
  Then $\mathcal{A}_1\cup \mathcal{A}_2 = \mathcal{A}$, so either $\limsup \mathcal{A}_1$ or $\limsup \mathcal{A}_2$ has positive measure. We assume that $|\limsup \mathcal{A}_1| > 0$ (otherwise, replace $T$ by $-T$), and put $S = \limsup \mathcal{A}_1$. Let $(\eta_I)_{I\in \mathcal{D}}$ be a family in $(0,1)$ with $\sigma < 1$, where $\sigma$ is defined as in~\eqref{eq:3}. We will construct an $(\eta_I)_{I\in \mathcal{D}}$-faithful Haar system from the Haar functions $(h_K)_{K\in \mathcal{A}_1}$ which almost diagonalizes~$T$. To this end, we also fix $0 < \varepsilon < \delta C^{-1}(1-\sigma)|S|/2$
  (where $C$ is the unconditional constant of the Haar system in $Y$) and a family $(\varepsilon_{I,J})_{I,J\in \mathcal{D}}$ of positive numbers satisfying~\eqref{eq:13}. For each $I\in \mathcal{D}$, we choose $\varepsilon_I > 0$ such that $\varepsilon_I < \min(\varepsilon_{I,J},\varepsilon_{J,I})$ for all $J\in \mathcal{D}$ with $J < I$.
  First, we explain how the Haar supports $\mathcal{B}_I$ of $\tilde{h}_I$ are defined:
  \begin{itemize}
    \item If $I = [0,1)$, we apply \Cref{lem:choice-of-B_I} to $\mathcal{A} = \mathcal{A}_1$, $\Gamma = [0,1)$, $\tilde{\eta} = 1 - |S|/2$, $\mathcal{E} = F = G = \emptyset$, $\eta = \eta_{[0,1)}/2$, and $\varepsilon = 1$. This yields a finite collection $\mathcal{B}_{[0,1)}\subset \mathcal{A}_1$ with $|\mathcal{B}_{[0,1)}^{*}| \ge |S|/2$ and $|\mathcal{B}_{[0,1)}^{*}\cap S| > (1 - \eta_{[0,1)}/2)|\mathcal{B}_{[0,1)}^{*}|$.
    \item If $I\ne [0,1)$, we assume that for every $J<I$, we have already constructed $\mathcal{B}_J$ and $\tilde{h}_J$ such that $|\mathcal{B}_J^{*}\cap S| >  (1-\eta_J/2)|\mathcal{B}_J^{*}|$.
          Put $\Gamma = [\tilde{h}_{\pi(I)} = \varepsilon(I)]$. Then the above inequality implies $|\Gamma\cap S| > (1-\eta_{\pi(I)})|\Gamma|$.
          Now we apply \Cref{lem:choice-of-B_I} to $\mathcal{A} = \mathcal{A}_1$, $\Gamma$ as defined above, $\tilde{\eta} = \eta_{\pi(I)}$, $\mathcal{E} = \bigcup_{J < I}\mathcal{B}_J$, $F = \{ T \tilde{h}_J : J < I \}$, $G = \{ T^{*} \tilde{h}_J : J < I \}$, $\eta = \eta_I/2$, and $\varepsilon = \varepsilon_I$. This yields a finite collection $\mathcal{B}_I\subset \mathcal{A}_1\setminus \bigcup_{J < I}\mathcal{B}_J$ with $\mathcal{B}_I^{*}\subset \Gamma$ such that $|\mathcal{B}_I^{*}\cap S| > (1-\eta_I/2)|\mathcal{B}_I^{*}|$ as well as $|\mathcal{B}_I^{*}| > (1-\eta_{\pi(I)}) \cdot \frac{1}{2}|\mathcal{B}_{\pi(I)}^{*}|$. Hence, conditions~\ref{cond:faithful-i-copy} and~\ref{cond:faithful-ii-copy} of \Cref{dfn:faithful-HS} are indeed satisfied. Moreover, for every choice of signs $\theta\in \{ \pm 1 \}^{\mathcal{B}_I}$, the function $\tilde{h}_I = \sum_{K\in \mathcal{B}_I} \theta_Kh_K$ satisfies
  \begin{equation}\label{eq:10}
    \max_{J\in \mathcal{D} : J<I} |\langle \tilde{h}_J, T \tilde{h}_I \rangle| < \varepsilon_I
    \qquad \text{and} \qquad
    \max_{J\in \mathcal{D} : J<I} |\langle \tilde{h}_I, T \tilde{h}_J \rangle| < \varepsilon_I.
  \end{equation}
\end{itemize}

  After constructing $\mathcal{B}_I$ for a given $I\in \mathcal{D}$, we choose the signs $\theta\in \{ \pm 1 \}^{\mathcal{B}_I}$ uniformly at random and note that since $\mathcal{B}_I\subset \mathcal{A}_1$, the randomized function $\tilde{h}_I = \sum_{K\in \mathcal{B}_I} \theta_Kh_K$ satisfies
  \begin{equation*}
    \mathbb{E}_{\theta\in \{ \pm 1 \}^{\mathcal{B}_I}} \langle \tilde{h}_I, T \tilde{h}_I \rangle = \sum_{K\in \mathcal{B}_I} \langle h_K, T h_K \rangle \ge \delta |\mathcal{B}_I^{*}|.
  \end{equation*}
  Hence, we can choose at least one sign vector $\theta\in \{ \pm 1 \}^{\mathcal{B}_I}$ such that the resulting function $\tilde{h}_I$ satisfies
  \begin{equation}\label{eq:11}
    \langle \tilde{h}_I, T \tilde{h}_I \rangle \ge \delta |\mathcal{B}_I^{*}|.
  \end{equation}
  Having constructed the $(\eta_I)_{I\in \mathcal{D}}$-faithful Haar system $(\tilde{h}_I)_{I\in \mathcal{D}}$, we now use \Cref{thm:operators-AB-copy}, \eqref{eq:10}, and \Cref{lem:close-to-diagonal-operator} to conclude that the associated operators $A,B$ and the Haar multiplier $D\colon Y\to Y$ defined by
  \begin{equation*}
    Dh_I = \frac{\langle \tilde{h}_I, T \tilde{h}_I \rangle}{|\mathcal{B}_I^{*}|} h_I,\qquad I\in \mathcal{D},
  \end{equation*}
  are bounded and satisfy $\|ATB - D\|\le \varepsilon(1-\sigma)^{-1}|\mathcal{B}_{[0,1)}^{*}|^{-1} < \delta C^{-1}$. Moreover, by~\eqref{eq:11}, $D$ has $\delta$-large positive entries, so its inverse satisfies $\|D^{-1}\| \le \delta^{-1}C$.
  Thus, $\|D^{-1}ATB - I_Y\| < \delta C^{-1} \delta^{-1} C = 1$, so $D^{-1}ATB$ is invertible, and the identity $I_Y$ factors through $T$. The subspace $Z$ can be chosen as $B(Y)$.
\end{proof}

\begin{rem}\label{rem:assumptions-factorization-property++}
  The statement of \Cref{thm:factorization-property++} also holds without assuming that the Haar system is unconditional in $Y$; it suffices that the standard Rademacher sequence is weakly null in~$Y$. This can be proved by first reducing $T$ to an operator with $\delta$-large diagonal as above and then applying \cite[Theorem~3.2]{lechner:speckhofer:2025}, i.e., the factorization property of the Haar basis of~$Y$.
\end{rem}

As a consequence, we immediately obtain the primary factorization property of $Y$ (which corresponds to \cite[Corollary~1.8]{bourgain:rosenthal:schechtman:1981} and has also been proved in~\cite{lechner:speckhofer:2025} without assuming unconditionality).
\begin{cor}\label{cor:primary-factorization-property}
  Let $X$ be a Haar system space and $Y = X^0$ or $H_X$. Suppose that the Haar system is unconditional in~$Y$.
  Then $Y$ has the primary factorization property.
\end{cor}
\begin{proof}
  Let $T\colon Y\to Y$ be a bounded linear operator. For every $I\in \mathcal{D}$, we either have $\langle h_I, Th_I \rangle \ge \frac12 |I|$ or $\langle h_I, (I_Y - T)h_I \rangle \ge \frac12 |I|$. Hence, for either $S = T$ or $S = I_Y - T$, the collection
$    \mathcal{A} = \{ I\in \mathcal{D} : \langle h_I, S h_I \rangle \ge \tfrac12 |I| \}$
  satisfies $|\limsup \mathcal{A}| > 0$. Now apply \Cref{thm:factorization-property++} to $\mathcal{A}$.
\end{proof}

\begin{lem}\label{lem:sublemma-1-copy}
  Let $X$ be a Haar system space and $Y = X^0$ or $H_X$. Suppose that the Haar system is unconditional in~$Y$ and that $Y$ is isomorphic, via an embedding $\iota\colon Y\to Z$, to a complemented subspace of a Banach space $Z$ with an unconditional Schauder decomposition $(Z_i)_{i=1}^{\infty}$. Let $U\colon Z\to Y$ be an operator with $U\iota = I_Y$. Moreover, let $Q_i\colon Z\to Z_i$ denote the natural projections associated with $(Z_i)_{i=1}^{\infty}$, and for $F\subset \mathbb{N}$, put $Q_F = \sum_{i\in F} Q_i$. Assume that no $Z_i$ contains a complemented subspace isomorphic to $Y$. Then, for every $n\in \mathbb{N}$, the collection
  \begin{equation*}
    \mathcal{A} = \{ I\in \mathcal{D} : \langle h_I, UQ_{[n,\infty)}\iota h_I\rangle > \tfrac{1}{2} |I| \}
  \end{equation*}
  satisfies $|\limsup \mathcal{A}| = 1$.
\end{lem}
\begin{proof}
  Suppose that $|\limsup \mathcal{A}| < 1$. Fix $n\ge 2$ (for $n=1$, the assertion is immediate), and put
  \begin{equation*}
    \mathcal{B} = \{ I\in \mathcal{D} : \langle h_I, UQ_{[1,n)}\iota h_I\rangle  \ge \tfrac{1}{2} |I| \}.
  \end{equation*}
  Then, since $U(Q_{[1,n)} + Q_{[n,\infty)})\iota = I_Y$, we see that $\mathcal{A}\cup \mathcal{B} = \mathcal{D}$, and hence, $\limsup\mathcal{A}\cup \limsup\mathcal{B} = [0,1)$. Since $|\limsup \mathcal{A}| < 1$, this implies $|\limsup\mathcal{B}| > 0$. Hence, $T = UQ_{[1,n)}\iota$ satisfies the hypotheses of \Cref{thm:factorization-property++}. Thus, there is a subspace $E$ of $Y$ with $E$ isomorphic to $Y$, $T|_E = UQ_{[1,n)}\iota|_E$ an isomorphism, and $T(E)$ complemented in $Y$. It follows easily that $Q_{[1,n)}\iota(E)$ is complemented in $Z$ and hence in $Q_{[1,n)}(Z)\simeq  Z_1 \oplus \dots \oplus Z_{n-1}$: Indeed, if $P$ is a projection from~$Y$ onto $T(E)$, then a projection from~$Z$ onto $Q_{[1,n)}\iota(E)$ is given by $\Pi = Q_{[1,n)}\iota(T|_E)^{-1}PU$.
  But this implies that $Y$ is isomorphic to a complemented subspace of $Z_1\oplus \dots\oplus Z_{n-1}$, so by \Cref{cor:primary-factorization-property} and \Cref{prop:pfp-complemented-subspace-direct-sum}, $Y$ is isomorphic to a complemented subspace of $Z_i$ for some~$i$, contradicting the hypothesis. Hence, we must have $|\limsup \mathcal{A}| = 1$.
\end{proof}

\begin{proof}[\normalfont\textbf{Proof of \Cref{thm:1.1}}]
Assume that $Y$ is isomorphic, via an embedding $\iota\colon Y\to Z$, to a complemented subspace of $Z$, and let $U\colon Z\to Y$ be an operator with $U\iota = I_Y$. Let $C$ denote the unconditional constant of the basis $(h_I)_{I\in \mathcal{D}}$ of $Y$. Let $(Z_i)_{i=1}^{\infty}$ be an unconditional decomposition of $Z$ with associated projections $(Q_i)_{i=1}^{\infty}$. Suppose that there is no $i$ with $Y$ isomorphic to a complemented subspace of~$Z_i$. Our goal is to construct a block sequence of the spaces $Z_i$ which is equivalent to the Haar system in $Y$ and has closed linear span complemented in $Z$.

To this end, we will construct a ``blocking'' of the decomposition $(Z_i)_{i=1}^{\infty}$, resulting in a new decomposition with projections $\bar{Q}_k$, $k\in \mathbb{N}$. To be precise, we will choose integers $1 = m_0 < m_1 < m_2 < \dots$ and put $\bar{Q}_k = Q_{[m_{k-1},m_k)} = \sum_{l = m_{k-1}}^{m_k - 1} Q_l$ for $k\in \mathbb{N}$. This defines an unconditional Schauder decomposition of $Z$ with the same (or smaller) unconditional constant. We will relabel the projections as $(\bar{Q}_I)_{I\in \mathcal{D}}$ via the usual bijection between $\mathbb{N}$ and $\mathcal{D}$.
At the same time, we will construct an almost faithful Haar system $(\tilde{h}_I)_{I\in \mathcal{D}}$ and an unconditional Schauder decomposition of $Y$ with associated projections $(P_I)_{I\in \mathcal{D}}$, with $P_I \tilde{h}_J = \delta_{I,J} \tilde{h}_J$ for all $I,J\in \mathcal{D}$, such that:
\begin{enumerate}[label=(\alph*),]
  \item\label{cond:i} $(\tilde{h}_I)_{I\in \mathcal{D}}$ satisfies the assumptions of \Cref{thm:operators-AB-copy}, and thus, it is equivalent to the Haar system $(h_I)_{I\in \mathcal{D}}$ in $Y$ and $[\tilde{h}_I]_{I\in \mathcal{D}}$ is complemented in $Y$.
  \item\label{cond:ii} $\langle \tilde{h}_I, U \bar{Q}_I \iota\tilde{h}_I \rangle > \frac{1}{2}|\mathcal{B}_I^{*}|$ for all $I\in \mathcal{D}$, where $\mathcal{B}_I$ is the Haar support of $\tilde{h}_I$.
  \item\label{cond:iii} $(\tilde{h}_I)_{I\in \mathcal{D}}$ ``almost diagonalizes'' the operator $UV$, where $V = \sum_I \bar{Q}_I \iota P_I$.
\end{enumerate}
To this end, fix a family $(\eta_I)_{I\in \mathcal{D}}$ in $(0,1)$ such that $\sigma < 1$, where $\sigma$ is defined as in~\eqref{eq:3}.
We also fix $0 < \varepsilon <  C^{-1}(1-\sigma)/2$ and a family $(\varepsilon_{I,J})_{I,J\in \mathcal{D}}$ of positive numbers satisfying~\eqref{eq:13}. For each $I\in \mathcal{D}$, we choose $\varepsilon_I > 0$ such that $\varepsilon_I < \min(\varepsilon_{I,J},\varepsilon_{J,I})$ for all $J\in \mathcal{D}$ with $J < I$.
Now let $k\in \mathbb{N}$, and let $I$ be the $k$th dyadic interval in the usual enumeration of $\mathcal{D}$.\medskip

\noindent\textbf{Case $k = 1$.}
If $k = 1$, i.e., $I =  [0,1)$, we simply put $\mathcal{B}_{[0,1)} = \{ [0,1) \}$ and $\tilde{h}_{[0,1)} = h_{[0,1)}$. Then $h_{[0,1)} = U\iota h_{[0,1)} = \lim_{n\to \infty} U Q_{[1,n)}\iota h_{[0,1)}$. So we can choose $m_1 > 1$ such that $\|UQ_{[1,m_1)}\iota h_{[0,1)} - h_{[0,1)}\| < 1/2$, hence $\langle h_{[0,1)}, U \bar{Q}_{[0,1)}\iota h_{[0,1)} \rangle > 1/2$, where $\bar{Q}_{[0,1)}=Q_{[1,m_1)}$ (note that $\|h_{[0,1)}\|_{Y^{*}} = 1$).\medskip

\noindent\textbf{Case $k > 1$.} Assume that we have already chosen $m_0 < m_1 < \dots < m_{k-1}$ and constructed $\tilde{h}_J$ for all $J<I$. Moreover, put $\Gamma = [\tilde{h}_{\pi(I)}= \varepsilon(I)]$. We will use the abbreviation $n = m_{k-1}$.
  Now define
  \begin{equation*}
    \mathcal{A} = \{ K\in \mathcal{D} : \langle h_K, UQ_{[n,\infty)}\iota h_K \rangle > \tfrac12 |K| \}.
  \end{equation*}
  Since we assumed that there is no $i$ with $Y$ isomorphic to a complemented subspace of $Z_i$, \Cref{lem:sublemma-1-copy} implies that $|\limsup (\mathcal{A})| = 1$. Now we apply \Cref{lem:choice-of-B_I} to $\mathcal{A}$ and $\Gamma$ as defined above, $\tilde{\eta} = \eta_{\pi(I)}$, $\mathcal{E} = \bigcup_{J < I} \mathcal{B}_J$, $F = \{ U \bar{Q}_J \iota\tilde{h}_J : J < I \}$, $G = \{ (UQ_{[n,\infty)}\iota)^{*} \tilde{h}_J : J < I \}$, $\eta = 1$, and $\varepsilon = \varepsilon_I$. This yields a finite collection $\mathcal{B}_I\subset \mathcal{A}\setminus \mathcal{E}$ with $\mathcal{B}_I^{*}\subset \Gamma$ and $|\mathcal{B}_I^{*}| \ge (1-\eta_{\pi(I)})|\Gamma| = (1-\eta_{\pi(I)})\cdot \frac12 |\mathcal{B}_{\pi(I)}^{*}|$. Hence, conditions~\ref{cond:faithful-i-copy},~\ref{cond:faithful-ii-copy} of \Cref{dfn:faithful-HS} are indeed satisfied. Moreover, for every choice of signs $\theta\in \{ \pm 1 \}^{\mathcal{B}_I}$, the function $\tilde{h}_I = \sum_{K\in \mathcal{B}_I} \theta_Kh_K$ satisfies
  \begin{equation}\label{eq:14}
    \max_{J\in \mathcal{D} : J<I} |\langle \tilde{h}_J, UQ_{[n,\infty)}\iota\tilde{h}_I \rangle| < \varepsilon_I
    \qquad \text{and} \qquad
    \max_{J\in \mathcal{D} : J<I} |\langle \tilde{h}_I, U \bar{Q}_J \iota\tilde{h}_J \rangle| < \varepsilon_I.
  \end{equation}

  Next, since $\mathcal{B}_I\subset  \mathcal{A}$, as in \Cref{thm:factorization-property++}, we can find a sign vector $\theta\in \{ \pm 1 \}^{\mathcal{B}_I}$ such that $\tilde{h}_I = \sum_{K\in \mathcal{B}_I} \theta_Kh_K$ satisfies
  \begin{equation}\label{eq:15}
    \langle \tilde{h}_I, UQ_{[n,\infty)}\iota\tilde{h}_I \rangle > \tfrac12  |\mathcal{B}_I^{*}|.
  \end{equation}

  Finally, since $Q_{[n,\infty)}y = \lim_{m\to \infty} Q_{[n,m)}y$ for all $y\in Z$, we can choose $m_k\in \mathbb{N}$ sufficiently large such that $m_k > n = m_{k-1}$ and such that after putting $\bar{Q}_I = Q_{[m_{k-1},m_k)}$, we still have
  \begin{equation}\label{eq:16}
    \max_{J\in \mathcal{D} : J<I} |\langle \tilde{h}_J, U \bar{Q}_I \iota\tilde{h}_I \rangle| < \varepsilon_I
    \qquad \text{and} \qquad
    \max_{J\in \mathcal{D} : J<I} |\langle \tilde{h}_I, U \bar{Q}_J \iota\tilde{h}_J \rangle| < \varepsilon_I
  \end{equation}
  as well as
  \begin{equation}\label{eq:17}
    \langle \tilde{h}_I, U \bar{Q}_I \iota\tilde{h}_I \rangle > \tfrac12  |\mathcal{B}_I^{*}|,
  \end{equation}
  cf.~\eqref{eq:14} and~\eqref{eq:15}. This completes the construction of the system $(\tilde{h}_I)_{I\in \mathcal{D}}$ and the projections $(\bar{Q}_I)_{I\in \mathcal{D}}$. Finally, choose a partition of $\mathcal{D}$ into finite subcollections $\mathcal{E}_I\subset \mathcal{D}$, $I\in \mathcal{D}$, such that $\mathcal{B}_I\subset \mathcal{E}_I$ for all $I\in \mathcal{D}$. The collections $\mathcal{E}_I$ define a $C$-unconditional Schauder decomposition of $Y$ with associated projections $P_I$ such that $P_I \tilde{h}_J = \delta_{I,J} \tilde{h}_J$ for all $I,J\in \mathcal{D}$.

  Next, consider the operator $V = \sum_{I\in \mathcal{D}} \bar{Q}_I \iota P_I$ and note that $V$ is bounded according to \Cref{lem:1.3}. Observe that $V \tilde{h}_I = \bar{Q}_I \iota\tilde{h}_I$ for all $I\in \mathcal{D}$. Hence, we may replace both $U \bar{Q}_I \iota$ and $U \bar{Q}_J \iota$ by the operator $UV$ in~\eqref{eq:16} and~\eqref{eq:17} (which also holds for $I = [0,1)$). Thus, we can apply \Cref{lem:close-to-diagonal-operator}, which implies that the Haar multiplier $D$ defined by
  \begin{equation*}
    D h_I = \frac{\langle \tilde{h}_I, UV \tilde{h}_I \rangle}{|\mathcal{B}_I^{*}|} h_I,\qquad I\in \mathcal{D},
  \end{equation*}
  is bounded on $Y$ and satisfies $\|AUVB - D\| < \varepsilon (1-\sigma)^{-1} < C^{-1}/2$, where $A,B$ are the operators associated with $(\tilde{h}_I)_{I\in \mathcal{D}}$. Moreover, by~\eqref{eq:17}, the entries of $D$ are larger than $1/2$, which implies that $D^{-1}$ is bounded and satisfies $\|D^{-1}\|\le 2C$. Hence, we have $\|D^{-1}AUVB - I_Y\| < 1$, so there exists an operator $R$ on $Y$ such that $I_Y = RD^{-1}AUVB$. This implies that $(VBh_I)_{I\in \mathcal{D}} = (\bar{Q}_I \iota\tilde{h}_I)_{I\in \mathcal{D}}$ is equivalent to $(h_I)_{I\in \mathcal{D}}$ in $Y$, and its closed linear span is complemented in $Z$ by the projection $VBRD^{-1}AU$.
\end{proof}

\section{Proof of the main result}\label{sec:Proof of the main result}

{
We now prove our main dichotomy mentioned in the introduction (\Cref{thmA:alternatives}) and the corresponding result for Hardy spaces. 
The proof proceeds by distinguishing two cases. If, for some $\alpha < \omega_1$, the space $R_\alpha^{Y}$ contains a complemented copy of~$Y$, then we use the result on complemented embeddings into spaces with unconditional decompositions (\Cref{thm:1.1}) to show that all spaces $R_\alpha^Y$, $\omega \le  \alpha < \omega_1$, are isomorphic to~$Y$. Otherwise, we extract a subfamily $(R_{\alpha_\xi}^Y)_{\xi < \omega_1}$ along which the complemented $Y$-index from \Cref{complemented ordinal section} is strictly increasing, which implies that these spaces are pairwise non-isomorphic.
}
\begin{thm}\label{thm:alternatives-ri}
  Let $X$ be a Haar system space with the Fatou property and non-trivial Boyd indices, and let $Y = X^0$. Then, one of the following exclusive alternatives holds:
  \begin{enumerate}[label=(\alph*),]
    \item\label{thm:alternatives:a} There exists a strictly increasing transfinite sequence of countable ordinals $(\alpha_\xi)_{\xi< \omega_1}$ such that the spaces $R_{\alpha_\xi}^{Y}$, $\xi < \omega_1$, are pairwise non-isomorphic.
    \item\label{thm:alternatives:b} For every ordinal $\omega \le \alpha < \omega_1$, the space $R_\alpha^{Y}$ is isomorphic to $Y$.
  \end{enumerate}
  Moreover,~\ref{thm:alternatives:b} is equivalent to~\ref{thm:alternatives:c} and~\ref{thm:alternatives:d} below.
    \begin{enumerate}[label=(\alph*),start=3]
    \item\label{thm:alternatives:c} For some $\alpha < \omega_1$, the space $R_\alpha^{Y}$ contains a complemented copy of~$Y$.
    \item\label{thm:alternatives:d} There exists an independent sequence of functions in $Y$ equivalent to the Haar system whose closed linear span is complemented in $Y$.
  \end{enumerate}
\end{thm}

The assumption that $X$ has non-trivial Boyd indices guarantees that the Haar system is unconditional in~$X$ and that for every disjointly branching tree $\mathcal{T}$, the tree space $R_{\mathcal{T}}^{Y}$ is complemented in~$Y$. If we consider the Hardy space~$H_X$ instead of~$X^0$, then unconditionality holds automatically, but complementation needs to be imposed as an additional assumption.

\begin{thm}\label{thm:alternatives-hardy}
  Let $X$ be a Haar system space with the Fatou property and $Y = H_X$, and assume that for every disjointly branching tree $\mathcal{T}$, the space $R_{\mathcal{T}}^{Y}$ is complemented in~$Y$. Then, one of the exclusive alternatives~\ref{thm:alternatives:a} and~\ref{thm:alternatives:b} in \Cref{thm:alternatives-ri} holds, and~\ref{thm:alternatives:b} is equivalent to~\ref{thm:alternatives:c}.
\end{thm}
\begin{rem}
  In fact, it suffices to assume complementation only for certain realizations of the trees $\mathcal{T}_\alpha$, $\alpha < \omega_1$; see the proof below. However, we do not know if this complementation assumption already implies that $X$ has non-trivial Boyd indices and hence $X^0 = H_X$.
\end{rem}

\begin{proof}[Proof of~\Cref{thm:alternatives-ri} and~\Cref{thm:alternatives-hardy}] Let $X$ be a Haar system space that satisfies the assumptions of \Cref{thm:alternatives-ri} or \Cref{thm:alternatives-hardy}, and let $Y = X^0$ or $H_X$, respectively. Then the Haar system is unconditional in $Y$, and all spaces $R_{\mathcal{T}_\alpha}^{Y}$, $\alpha < \omega_1$, are complemented in~$Y$ (if $Y = X^0$, this follows from \Cref{prop:complementation}).
\begin{itemize}[leftmargin=1.5em, labelsep=1.5em, align=left, itemsep=0.5em]
  \item[\ref{thm:alternatives:b}$\Longrightarrow$\ref{thm:alternatives:c}.] This is immediate.
  \item[\ref{thm:alternatives:c}$\Longrightarrow$\ref{thm:alternatives:b}.]
        Suppose that~\ref{thm:alternatives:c} is true. We will prove that $R_\omega^{Y}$ contains a complemented copy of~$Y$. Then, as $R_\omega^{Y}$ also embeds complementably into $R_\beta^{Y}$ for all $\omega \le \beta < \omega_1$ and both~$Y$ and~$R_\beta^{Y}$ are isomorphic to their squares (see \Cref{prop:isomorphic-to-square}, and \Cref{prop:R-alpha-Y-0-isomorphic-to-its-square}), we obtain~\ref{thm:alternatives:b} using Pe{\l}czy\'nski's decomposition method.

  Let $\alpha$ be the smallest countable ordinal such that $R_\alpha^{Y}$ contains a complemented copy of~$Y$. Clearly, $\alpha$ must be infinite, and it cannot be a successor ordinal: If $\alpha = \beta + 1$, we would have $R_\alpha^{Y} = (R_\beta^{Y} \oplus R_\beta^{Y})_{\mathrm{disj},Y}^+ \simeq R_\beta^{Y} \oplus R_\beta^{Y} \oplus \mathbb{R}$, so by \Cref{cor:primary-factorization-property} and \Cref{prop:pfp-complemented-subspace-direct-sum}, the space~$R_\beta^{Y}$ would also contain a complemented copy of $Y$, contradicting the minimality of~$\alpha$.
        Hence, $\alpha$ is a limit ordinal, and by \Cref{showing these are really BRS spaces}, we have $R_\alpha^{Y}\equiv R_{\mathcal{T}_\alpha}^{Y}$. Moreover, by \Cref{rem:BRS-equals-tree-BRS}, the tree $\mathcal{T}_\alpha$ may be chosen such that
        \begin{equation*}
         R_{\mathcal{T}_\alpha}^{Y} = [(J^{\alpha}_\beta(R_{\mathcal{T}_\beta}^{Y}))_{\beta < \alpha}]
       \end{equation*}
        for a suitable choice of (square function-) distributional embeddings $J^{\alpha}_\beta\colon Y\to Y$, $\beta < \alpha$, with (square function-) independent ranges. The spaces $J^{\alpha}_\beta(R_{\mathcal{T}_\beta}^{Y})$, $\beta < \alpha$, form an unconditional Schauder decomposition of $R_{\mathcal{T}_\alpha}^{Y}$. Since none of them contains a complemented copy of $Y$, \Cref{thm:1.1} implies that a block basic sequence $(x_n)_{n=1}^{\infty}$ of the spaces $J^{\alpha}_\beta(R_{\mathcal{T}_\beta}^{Y})$, $\beta < \alpha$, is equivalent to the Haar system in $Y$ and has closed linear span complemented in $R_{\mathcal{T}_\alpha}^{Y}$. In passing, we note that this implies~\ref{thm:alternatives:d} if $Y = X^0$, since $R_{\mathcal{T}_\alpha}^{Y}$ is in turn complemented in $Y$.

        Now write
  \begin{equation*}
    x_n = \sum_{\beta \in \sigma_n}  y_\beta,\qquad n\in \mathbb{N},
  \end{equation*}
  where the sets $\sigma_n$ are pairwise disjoint finite sets of ordinals smaller than $\alpha$ and $y_\beta = J^{\alpha}_\beta z_\beta$ for some $z_\beta \in R_{\mathcal{T}_\beta}^{Y}$. We put $z_\beta = y_\beta =0$ for $\beta\notin \cup_n \sigma_n$. By the principle of small perturbations, we may assume that each vector~$z_\beta$ is a finite linear combination of blocks of the standard Haar system (note that these are dense in $R_{\mathcal{T}_\beta}^{Y}$ by \Cref{prop:basis-of-tree-space}), while $[x_n]_{n=1}^{\infty}$ is still complemented in~$R_{\mathcal{T}_\alpha}^{Y}$ by a projection $P\colon R_{\mathcal{T}_\alpha}^{Y}\to R_{\mathcal{T}_\alpha}^{Y}$. In particular, there exist pairwise distinct numbers $k(\beta)\in \mathbb{N}$, $\beta < \alpha$, such that $z_\beta\in Y_{k(\beta) - 1} = R_{k(\beta)}^{Y}$ for all~$\beta$.
  Now recall that $R_\omega^{Y} = [(J^{\omega}_k(R_k^{Y}))_{k=0}^{\infty} ]$ for certain (square function-) distributional embeddings $J_k^{\omega}$ with (square function-) independent ranges, and define the operator $U\colon [y_\beta]_{\beta < \alpha} \to R_\omega^{Y}$ by
  \begin{equation*}
    Uy_\beta = \tilde{y}_\beta := J_{k(\beta)}^{\omega}z_\beta,\qquad \beta < \alpha.
  \end{equation*}
  Then $U$ is an isometry. In particular, the sequence $(\tilde{x}_n)_{n=1}^{\infty}$ defined by
  \begin{equation*}
    \tilde{x}_n = U x_n = \sum_{\beta \in \sigma_n} \tilde{y}_\beta \in R_\omega^{Y},\qquad  n\in \mathbb{N},
  \end{equation*}
  is equivalent to $(x_n)_{n=1}^{\infty}$ in $R_{\mathcal{T}_\alpha}^{Y}$, which is in turn equivalent to the Haar system in~$Y$.
  Now it only remains to show that $[\tilde{x}_n]_{n=1}^{\infty}$ is complemented in $R_\omega^{Y}$. To this end, we define the operator $V\colon R_\omega^{Y}\to Y$ by setting, for every $k\in \mathbb{N}_0$ and $z\in R_k^{Y}$,
  \begin{equation*}
    V(J^{\omega}_k z) =
    \begin{cases}
      J^{\alpha}_\beta z, & \text{if } k = k(\beta) \text{ for some } \beta < \alpha, \\
      0, & \text{if } k \notin \{ k(\beta) : \beta < \alpha \}.
    \end{cases}
  \end{equation*}
  Then $V$ is bounded, as it is obtained by projecting $R_\omega^{Y}$ onto a subsequence of its canonical unconditional decomposition and then applying (square function-) distributional isomorphisms with (square function-) independent ranges to the individual components. Moreover, note that $V \tilde{y}_\beta = y_\beta$ for all $\beta < \alpha$, and hence,
  \begin{equation*}
    V \tilde{x}_n = x_n,\qquad n\in \mathbb{N}.
  \end{equation*}
  Finally, let $Q$ be a bounded linear projection from $Y$ onto $R_{\mathcal{T}_\alpha}^{Y}$, and define $\widetilde{P}\colon R_\omega^{Y}\to R_\omega^{Y}$ by
  \begin{equation*}
    \widetilde{P} = UPQV.
  \end{equation*}
  Then $\widetilde{P}(R_\omega^{Y}) \subset U([x_n]_{n=1}^{\infty})\subset [\tilde{x}_n]_{n=1}^{\infty}$, and for every $n\in \mathbb{N}$,
  \begin{equation*}
    \widetilde{P} \tilde{x}_n = UPQx_n = Ux_n = \tilde{x}_n.
  \end{equation*}
  Thus, $\widetilde{P}$ is a projection from $R_\omega^{Y}$ onto $[\tilde{x}_n]_{n=1}^{\infty}$.

  \item[\ref{thm:alternatives:c}$\Longleftrightarrow$\ref{thm:alternatives:d} for $Y = X^0$.]
        In the previous part of the proof, we have already seen that \ref{thm:alternatives:c}$\Longrightarrow$\ref{thm:alternatives:d} if $Y = X^0$.
        Next, suppose that $Y = X^0$ and~\ref{thm:alternatives:d} holds, i.e., there exists an independent sequence $(x_n)_{n=1}^{\infty}$ in $Y$ equivalent to the Haar system whose closed linear span is complemented in $Y$. { Let $J_n\colon Y\to Y$, $n\in \mathbb{N}$, be distributional embeddings with independent ranges. The sequence $(x_n')_{n=1}^{\infty}$ in $(\oplus_{n=1}^{\infty} Y)_{Y,\mathrm{ind}} = [J_n(Y)]_{n=1}^{\infty}$, defined by $x_n' = J_n x_n$ for all $n$, is also independent and equivalent to the Haar system, and its closed linear span is complemented in $Y$, since distributional embeddings preserve complementation (see \Cref{lem:distributional-copy-complemented}). Now approximate the functions $x_n'$ by elements of~$J_n(R_{k(n)}^Y)$ for sufficiently large natural numbers $k(1)<k(2)<\cdots$. By mapping these elements distributionally into the corresponding components of $R_\omega^Y = (\oplus_{k=1}^{\infty} R_k^Y)_{Y,\mathrm{ind}}$, we see that $R_\omega^Y$
        contains a sequence spanning a complemented copy of~$Y$, which, as above, implies~\ref{thm:alternatives:b}.}

  \item[\ref{thm:alternatives:a}$\vee $\ref{thm:alternatives:b}.]
        Finally, assume that~\ref{thm:alternatives:b}, and hence~\ref{thm:alternatives:c}, is false. \Cref{index main theorem}~\ref{index main theorem b} implies that $i_Y[R_\alpha^{Y}] < \omega_1$ for all $\alpha < \omega_1$. We construct $(\alpha_\xi)_{\xi\in \omega_1}$ as follows: Put $\alpha_0 = 0$. Now fix $1 \le \xi < \omega_1$, and assume that $\alpha_\eta$ has already been defined for all $\eta < \xi$. Then $\beta := \sup_{\eta < \xi} i_Y[R_{\alpha_\eta}^{Y}]$ is a countable ordinal, and we put $\alpha_\xi = \beta + 1$. By \Cref{complementex Y index of tree BRS spaces}, we have $i_Y[R_{\beta + 1}^{Y}] \ge \beta + 1$. In particular, for all $\eta < \xi$, we have $i_Y[R_{\alpha_\xi}^{Y}] > i_Y[R_{\alpha_\eta}^{Y}]$, which, by \Cref{index main theorem}~\ref{index main theorem a}, implies that $R_{\alpha_\xi}^{Y}$ and $R_{\alpha_\eta}^{Y}$ are non-isomorphic. Hence,~\ref{thm:alternatives:a} is true.\qedhere
\end{itemize}
\end{proof}

\section{Unboundedness of orthogonal projections}
\label{sec:unboundedness}

The main goal of this section is to prove that the Bourgain-Rosenthal-Schechtman tree spaces fail to be complemented in certain Haar system spaces and Hardy spaces because the orthogonal projections onto them are unbounded.
In particular, we prove this result for $L_1$ and $H^1$, and also for a certain Orlicz function space ``close'' to $L_1$ and its associated Hardy space.
In $L_1$, the orthogonal projection is even unbounded as an operator from $H^1$ into $L_1$; this is already hinted at in~\cite[p.~223]{bourgain:rosenthal:schechtman:1981}.
Before stating the main result of this section, we introduce the necessary notation and define the aforementioned Orlicz space.

\begin{ntt}
  For $a > 0$ and two functions $f,g\colon [a,\infty)\to (0,\infty)$, we say $f(s)\sim g(s)$ as $s\to \infty$ if $f(s) = g(s)(1 + o(1))$, where $o(1)$ denotes a function tending to $0$ as $s\to \infty$.
\end{ntt}
\begin{ntt}\label{ntt:orlicz-space}
  Throughout this section, let $\Phi\colon [0,\infty)\to [0,\infty)$ be a  strictly increasing  Young's function (see \cite[Chapter~4, Section~8]{bennett:sharpley:1988} for the definition) with $\Phi(1)=1$ and $\Phi(s) \sim s (\log s)^{\gamma}$ as $s \to \infty$, for some $\gamma > 0$. Let $L_{\Phi}$ be the Orlicz function space on $[0,1)$ generated by $\Phi$, equipped with the Luxemburg norm. Since $\Phi$ satisfies the $\Delta_2$-condition at infinity, the simple functions are dense in~$L_\Phi$ (see \cite[p.~120]{lindenstrauss:tzafriri:1979:partII}). This implies that the linear span of the Haar system and~$\chi_{[0,1)}$ is dense in~$L_\Phi$, and therefore, $L_\Phi$ is a Haar system space.
\end{ntt}

  The next theorem is the main result of this section. It contains examples of spaces in which the Bourgain-Rosenthal-Schechtman tree spaces are uncomplemented or where the complementation depends on the realization of the trees, i.e., on the particular choice of the trees~$\mathcal{T}_s'$ forming the disjoint sum in \Cref{disjointly branching tree operations}~\ref{disjointly branching tree operations b}.
  Part~\ref{thm:tree-spaces-uncomplemented:c} of \Cref{thm:tree-spaces-uncomplemented} follows from a result of M\"uller and Schechtman~\cite{mueller:schechtman:1989}.
  For the definition of $o[\mathcal{T}]$, see above \Cref{rem:order-of-ordinal-trees}.

\begin{thm}\label{thm:tree-spaces-uncomplemented}
  Let $Y =L_1^0$, $H^1$, $L_{\Phi}^0$, or $H_{L_{\Phi}}$, where $\Phi$ is as in \Cref{ntt:orlicz-space} for some $0 < \gamma < 1/2$. Then the following hold:
  \begin{enumerate}[label=(\alph*)]
    \item\label{thm:tree-spaces-uncomplemented:a} For every well-founded disjointly branching tree $\mathcal{T}$ with $o[\mathcal{T}]\geq \omega\cdot 2$, the space $R_{\mathcal{T}}^{Y}$ is uncomplemented in $Y$. In particular, this holds for the space $R_{\mathcal{T}_\alpha}^Y$ if $\omega\cdot 2 \le \alpha < \omega_1$.
    \item\label{thm:tree-spaces-uncomplemented:b} For every $n\in \mathbb{N}_0$, there exist realizations of $\mathcal{T}_{\omega + n}$ for which $R_{\mathcal{T}_{\omega + n}}^Y$ is uncomplemented in~$Y$.
    \item\label{thm:tree-spaces-uncomplemented:c} For every $n\in \mathbb{N}_0$, there exist realizations of $\mathcal{T}_{\omega + n}$ (such as the one in \Cref{exa:ordinal-trees}) for which $R_{\mathcal{T}_{\omega + n}}^{H^1}$ is complemented in~$H^1$.
  \end{enumerate}
\end{thm}

\begin{rem}
  Note that if $Y = L_1^0$ or $L_{\Phi}^0$, then by \Cref{lem:distributional-copy-complemented}, the conclusion of~\ref{thm:tree-spaces-uncomplemented:b} holds for all realizations of $\mathcal{T}_{\omega + n}$ since the resulting spaces $R_{\mathcal{T}_{\omega + n}}^Y$ are distributionally isomorphic.
\end{rem}

{
\begin{rem}
  In the cases where the tree spaces in \Cref{thm:tree-spaces-uncomplemented} are uncomplemented in~$Y$, this is proved by showing that the corresponding orthogonal projection onto such a subspace is unbounded. Thus, by duality, these results also carry over to the closure of the linear span of the Haar system in~$L_{\infty}$ or in the dual space of~$L_\Phi$.
\end{rem}
}

Before proving \Cref{thm:tree-spaces-uncomplemented}, we first collect some important properties of the Orlicz space~$L_{\Phi}$, including its Boyd indices and fundamental function.
\begin{dfn}
  Let $X$ be a Haar system space with $\|\cdot\|_X\not\asymp\|\cdot\|_\infty$ or a rearrangement-invariant space. The fundamental function $\varphi_X$ of $X$ is defined by $\varphi_X(t) = \|\chi_{[0,t)}\|_X$ for $t\in (0,1]$.
\end{dfn}

\begin{rem}\label{rem:chat-gpt-orlicz}
  The specific form of the function $\Phi$ in \Cref{ntt:orlicz-space} was suggested by ChatGPT~5.2 when prompted to identify a Young's function for which the first lower bound in \Cref{prop:unboundedness} tends to infinity. For this purpose, ChatGPT also provided the asymptotic expansion of its fundamental function $\varphi_X$, along with a proof sketch that we made rigorous in the following lemma.
\end{rem}

\begin{lem}\label{lem:orlicz-fundamental-function}
  Let $\Phi\colon [0,\infty)\to [0,\infty)$ be a Young's function as in \Cref{ntt:orlicz-space}. Then the fundamental function of the Orlicz space $L_{\Phi}$ satisfies $\varphi_{L_{\Phi}}(t) \sim t (\log \frac{1}{t})^{\gamma}$ as $t \to 0$. Moreover, $p_{L_{\Phi}} = q_{L_{\Phi}} = 1$, and $\|T_{[0,t)}\|_{L_{\Phi}\to L_{\Phi}} \asymp  t(\log\frac{1}{t})^{\gamma}$ for $0 < t < 1/2$, where the implied constants depend only on~$\Phi$.
\end{lem}
\begin{proof}
  The fundamental function of the Orlicz space $L_{\Phi}$ is given by the formula
  \begin{equation*}
    \varphi_{L_{\Phi}}(t) = \frac{1}{\Phi^{-1}(1/t)}, \qquad t\in (0,1].
  \end{equation*}
  Now let $y = \Phi(s)$ for some $s > 0$. Then, as $s \to \infty$, we have $y \sim  s (\log s)^{\gamma}$, and thus,
  \begin{equation}\label{eq:8}
    \log y = \log s + \gamma \log \log s + o(1).
  \end{equation}
  This implies that $\log y \sim \log s$. Plugging $\log s = \log y\, (1 + o(1))$ into the term $\log \log s$ in~\eqref{eq:8} yields
  \begin{equation*}
    \log y = \log s + \gamma \log\log y + o(1).
  \end{equation*}
  Solving for $s$, we obtain
  \begin{equation*}
    \Phi^{-1}(y) = s \sim y (\log y)^{-\gamma}\qquad \text{as }s\to \infty
  \end{equation*}
  or, equivalently, as $y\to \infty$.
  This implies the first claimed result. The second statement follows from \cite[Chapter~4, Theorem~8.18]{bennett:sharpley:1988}.

  To prove the last statement, we use~(8.37) from the proof of \cite[Chapter~4, Theorem~8.18]{bennett:sharpley:1988}, i.e., $\|T_{[0,t)}\|_{L_{\Phi}\to L_{\Phi}}\le k(2t)$ for $0 < t < 1/2$, where
  \begin{equation*}
    k(t) = \sup \Bigl\{ \frac{\Phi^{-1}(s)}{\Phi^{-1}(s/t)} : 1\le s < \infty \Bigr\},\qquad 0 < t < 1.
  \end{equation*}
  We know that
  \begin{equation*}
    \Phi^{-1}(s)\sim s (\log s)^{-\gamma} \sim s (\log(e + s))^{-\gamma}\qquad \text{as } s\to \infty,
  \end{equation*}
  so
  \begin{equation*}
    s\mapsto \frac{\Phi^{-1}(s)}{s \log(e + s)^{-\gamma}}
  \end{equation*}
  is a continuous function on $[1,\infty)$ that converges to $1$ as $s\to \infty$. Thus, there exist $m,M > 0$ such that
  \begin{equation*}
    m s (\log(e + s))^{-\gamma} \le  \Phi^{-1}(s) \le M s (\log(e + s))^{-\gamma}, \qquad s\in [1,\infty).
  \end{equation*}
  Hence, for every $0 < t < 1$ and $s\in [1,\infty)$,
 \begin{equation*}
    \frac{\Phi^{-1}(s)}{\Phi^{-1}(s/t)} \le \frac{M}{m} t \biggl[ \frac{\log(e+s/t)}{\log(e+s)} \bigg]^{\gamma}
                                       \le \frac{M}{m} t \biggl[ \frac{\log(e+s) + \log(1/t)}{\log(e+s)} \bigg]^{\gamma}
                                       \le \frac{M}{m} t \Bigl[ 1 + \log(1/t) \Bigr]^{\gamma},
 \end{equation*}
 which implies that $k(t) \lesssim t(1 + \log\frac{1}{t})^{\gamma}$. Thus, for $0 < t <1/2$, we obtain $\|T_{[0,t)}\|_{L_{\Phi}\to L_{\Phi}}\le k(2t) \lesssim t (\log\frac{1}{t})^{\gamma}$. The reverse inequality holds because $\|T_{[0,t)}\|_{L_{\Phi}\to L_{\Phi}} \ge  \varphi_{L_{\Phi}}(t) \gtrsim t (\log\frac{1}{t})^{\gamma}$.
\end{proof}

To obtain the main result of this section, we will show that every well-founded disjointly branching tree of sufficiently large order contains a certain structure, defined below.

\begin{dfn}\label{dfn:yardstick-family}
  Let $m,k\in \mathbb{N}$. A family $(F_r)_{r=1}^m$ of pairwise disjoint finite subsets of $\mathbb{N}_0$ is called a $k$-yardstick family if $|F_r| \ge k + 1$ for all $r$ and there exists an $n\in \mathbb{N}$ such that for all $r=1,\dots,m$ and all $l\in F_r\setminus \{ \max F_r \}$, we have $l < n \le \max F_r$.
\end{dfn}

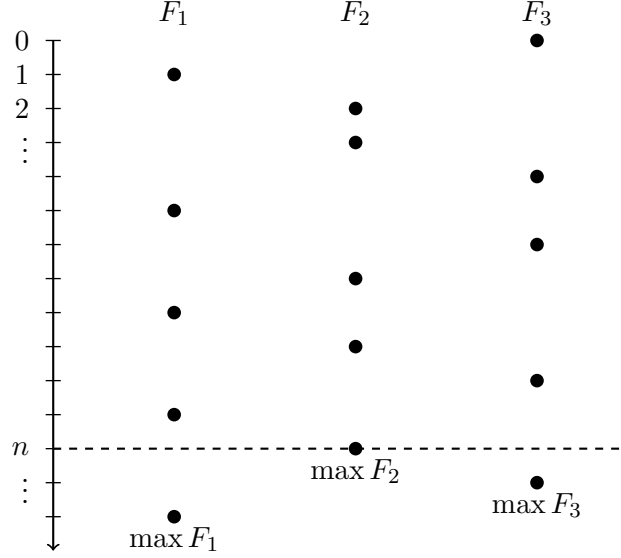
\begin{figure}[t]
  \centering
 \begin{tikzpicture}[
  x=1cm,
  y=-0.45cm,
  vertex/.style={circle, fill, inner sep=0pt, minimum size=1.8mm},
  tick/.style={line width=0.5pt},
]

\draw[->, thick] (0,0) -- (0,15.0);

\foreach \y in {0,...,14}
  \draw[tick] (-0.10,\y) -- (0.10,\y);

\node[anchor=east] at (-0.18,0) {$0$};
\node[anchor=east] at (-0.18,1) {$1$};
\node[anchor=east] at (-0.18,2) {$2$};
\node[anchor=east] at (-0.18,3) {$\vdots$};
\node[anchor=east] at (-0.18,12) {$n$};
\node[anchor=east] at (-0.18,13) {$\vdots$};

\draw[thick,dashed] (0,12) -- (7.5,12);

\node at (1.6,-0.8) {$F_1$};
\node at (4.0,-0.8) {$F_2$};
\node at (6.4,-0.8) {$F_3$};

\foreach \y in {1,5,8,11,14}
  \node[vertex] at (1.6,\y) {};

\foreach \y in {2,3,7,9,12}
  \node[vertex] at (4.0,\y) {};

\foreach \y in {0,4,6,10,13}
  \node[vertex] at (6.4,\y) {};

\node[anchor=north] at (1.6,14) {$\max F_1$};
\node[anchor=north] at (4.0,12) {$\max F_2$};
\node[anchor=north] at (6.4,13) {$\max F_3$};
  
\end{tikzpicture}
 
  \caption{A $4$-yardstick family $(F_1,F_2,F_3)$}
  \label{fig:yardstick-family}
\end{figure}

\Cref{fig:yardstick-family} shows a $4$-yardstick family with $m = 3$.
Now, the key observation is that if a tree contains yardstick families, then the orthogonal projection onto the associated tree space has a large norm, assuming that the underlying space is sufficiently ``close'' to $L_1$ or $H^1$. This is made precise in the following result.

\begin{prop}\label{prop:unboundedness}
  Let $X$ be a Haar system space with $\|\cdot\|_X\not\asymp\|\cdot\|_\infty$.
  Let $k\in \mathbb{N}$ and $m = 2^k$, and let $(F_r)_{r=1}^m$ be a $k$-yardstick family. 
  Then the orthogonal projection $P$ from $\langle \{ h_I : I\in \mathcal{D}  \} \rangle$ onto $Y_{F_1} + \dots + Y_{F_m}$ satisfies the following lower bounds:
  \begin{equation}\label{eq:4}
    \|P\|_{H_X\to H_X}\gtrsim \frac{m^{-m}}{\varphi_X(m^{-m})} \sqrt{m},
   \end{equation}
  and, if $(r_n)_{n=0}^{\infty}$ in $X$ is equivalent to the unit vector basis of $\ell_2$, then
  \begin{equation}\label{eq:7}
   \|P\|_{X^0\to X^0} \gtrsim \frac{m^{-m}}{\|T_{[0, m^{-m})}\|_{X^0\to X^0}} \sqrt{m}.
  \end{equation}
   Both estimates remain valid if the codomain of $P$ is replaced by $X^0$ and $H_X$, respectively.
  The implied constants in~\eqref{eq:4} and~\eqref{eq:7} depend only on the underlying space~$X$.
\end{prop}

\begin{proof}
  For $r = 1,\dots,m$, let $|F_r| = k_r + 1$, and let $\max F_r$ denote the largest and $s_r$ the second-largest element of $F_r$.
  Let $n\in \mathbb{N}$ be as in \Cref{dfn:yardstick-family} and note that $s_r \le n - 1$.
  Choose a collection $\mathcal{B}_r\subset \mathcal{D}^{k_r}$ such that the function
  \begin{equation*}
    f_r = \sum_{I\in \mathcal{B}_r} b_I^{F_r},
  \end{equation*}
  satisfies $|\supp f_r| = 2^{-k} =: \rho$. Define
  \begin{equation*}
    A = \bigcap_{r=1}^m \supp f_r.
  \end{equation*}
  Note that for every $I\in \mathcal{D}^{k_r}$, the set $\supp b_I^{F_r} = [b_{\pi(I)}^{F_r} = \varepsilon(I)]$ is $\sigma(\mathcal{D}^{s_r + 1})\subset \sigma(\mathcal{D}^{n})$-measurable, so $\supp f_r$ is $\sigma(\mathcal{D}^n)$-measurable as well for every $r$. Hence, $A$ is $\sigma(\mathcal{D}^n)$-measurable and, by independence, $|A| = \rho^m$. Now define
  \begin{equation*}
    f = \sum_{r=1}^m f_r \chi_A.
  \end{equation*}
  Since $\max F_r \ge n$ for all $r$, we see that every $K$ in the Haar support of $f_r$ is either a subset of~$A$ or disjoint from $A$. Thus,
  \begin{equation*}
    f = \sum_{r=1}^m \sum_{\substack{K\in \mathcal{D}^{\max F_r}\\ K\subset A}} h_K = \chi_A\sum_{r=1}^m r_{\max F_r}\qquad \text{and} \qquad S(f) =  \sqrt{m}\, \chi_A.
  \end{equation*}
  This implies that $\|f\|_Y = \sqrt{m}\, \|\chi_A\|_X$ if $Y = H_X$, and $\|f\|_Y \lesssim  \sqrt{m}\, \|T_{[0,|A|)}\|_{Y\to Y}$ if $Y = X^0$, assuming that $(r_n)_{n=1}^{\infty}$ in $X$ is equivalent to the unit vector basis of $\ell_2$.

  The orthogonal projection onto $Y_{F_1} + \dots + Y_{F_m}$ is given by
  \begin{equation}\label{eq:20}
    Px = \sum_{r=1}^m \mathbb{E}_{F_r}x = \sum_{r=1}^m \sum_{I\in \mathcal{D}_{k_r}} |I|^{-1} \langle b_I^{F_r}, x \rangle b_I^{F_r},\qquad  x\in \langle \{ h_I : I\in \mathcal{D} \} \rangle.
  \end{equation}
  Observe that
  \begin{equation*}
    \langle b_I^{F_r}, b_J^{F_s} \chi_A \rangle = 0
  \end{equation*}
  for $r,s\in \{ 1,\dots,m \}$ and $I\in \mathcal{D}_{k_r}$, $J\in \mathcal{B}_s$, unless $r = s$
  and $I = J$. Thus, for the function~$f$ defined above, we obtain
  \begin{equation*}
    Pf = \sum_{r=1}^m \sum_{I\in \mathcal{B}_r} |I|^{-1} \langle b_I^{F_r}, b_I^{F_r} \chi_A \rangle b_I^{F_r}.
  \end{equation*}
  By independence, for $I\in \mathcal{B}_r$, we have
  \begin{equation*}
    \langle b_I^{F_r}, b_I^{F_r} \chi_A \rangle = |\supp(b_I^{F_r})\cap A| = |I|\cdot \frac{|A|}{|\supp f_r|} = \rho^{-1}|I| |A|,
  \end{equation*}
  which implies that
  \begin{equation*}
    Pf = \rho^{-1}|A| \sum_{r=1}^m \sum_{I\in \mathcal{B}_r} b_I^{F_r} = \rho^{-1}|A|\sum_{r=1}^m f_r
  \end{equation*}
  and
  \begin{equation*}
    S(Pf) = \rho^{-1}|A| \Bigl( \sum_{r=1}^m |f_r| \Bigr)^{1/2}.
  \end{equation*}
  The functions $|f_r|$ are $\{ 0,1 \}$-valued, independent, and ``almost'' disjointly supported in the following sense: We have $|\supp f_r| = \rho =  \frac{1}{m}$ for all $r=1,\dots,m$, which implies that
  \begin{equation*}
    \bigl|\bigl[|f_r| = 1 \wedge |f_s| = 0\ \forall s\ne r\bigr]\bigr| \ge  \rho \Bigl( 1 - \frac{1}{m}\Bigr)^m \ge \frac{\rho}{4}
  \end{equation*}
  for all $r$. Hence, there exist $m$ pairwise disjoint sets, each of measure at least $\rho/4$, on which
  \begin{equation*}
    |Pf| = S(Pf) = \rho^{-1}|A|.
  \end{equation*}
  Thus, for any Haar system space $X$ and $Y = X^0$ or $H_X$, we have
  \begin{equation*}
    \|Pf\|_Y \ge \rho^{-1}|A|\cdot  \|\chi_{[0,\, m \rho/4 )}\|_X.
  \end{equation*}
  Altogether, we obtain
  \begin{equation}\label{eq:18}
    \frac{\|Pf\|_Y}{\|f\|_{H_X}} \gtrsim \frac{\rho^{-1}|A|\cdot \|\chi_{[0,m\rho/4)}\|_X}{\sqrt{m}\|\chi_A\|_X} \gtrsim  \frac{\varphi_X(m\rho)/(m\rho)}{\varphi_X(\rho^m)/\rho^m} \sqrt{m}.
  \end{equation}
  Since $m\rho = 1$, the last inequality simplifies to~\eqref{eq:4}.
  In passing, we note that $\varphi_X(t)/t = 1/\varphi_{X'}(t)$ for $t\in (0,1]$, where $X'$ is the associate space of $X$.
  To obtain~\eqref{eq:7}, we need to replace $\|f\|_{H_X}$ with $\|f\|_X$ and $\|\chi_A\|_X$ with $\|T_{[0,|A|)}\|_{X^0\to X^0}$ in~\eqref{eq:18}.
\end{proof}

Note that by \Cref{lem:orlicz-fundamental-function}, we have $q_{L_{\Phi}} < \infty$, which implies that the sequence of Rademacher functions in~$L_{\Phi}$ is equivalent to the unit vector basis of~$\ell_2$ (see the discussion preceding \cite[Theorem~2.b.4]{lindenstrauss:tzafriri:1979:partII}; see also~\cite{rodin:semyonov:1975}).
Thus, combining \Cref{prop:unboundedness} with \Cref{lem:orlicz-fundamental-function}, we immediately obtain the following result:

\begin{cor}\label{cor:norm-P-tends-to-infty}
  If $Y =L_1^0$, $H^1$, $L_{\Phi}^0$, or $H_{L_{\Phi}}$, where $\Phi$ is as in \Cref{ntt:orlicz-space} for some $\gamma < 1/2$, then the projection $P$ in \Cref{prop:unboundedness} satisfies $\|P\|_{Y\to Y}\to \infty$ as $k\to \infty$.
\end{cor}

Now it remains to show that trees of sufficiently large order give rise to arbitrarily large yardstick families (which are not necessarily contained in the trees themselves, but related to them). To illustrate the general idea of our construction, we first give a direct proof for the trees~$\mathcal{T}_\alpha$, where $\alpha$ is a limit ordinal, before stating and proving the general result.

\begin{exa}\label{exa:T-alpha-yardstick}
  Let $\alpha \ge \omega\cdot 2$ be a countable limit ordinal. Then for all $m,k \in \mathbb{N}$, there exists a $k$-yardstick family $F_1,\ldots,F_m$ in $\mathcal{T}_\alpha$.
\end{exa}
\begin{proof}
  The construction used in this proof is illustrated in \Cref{fig:yardstick-from-tree}.
  Recall that $\mathcal{T}_\alpha = \sqcup_{0\leq \beta<\alpha}\mathcal{T}_\beta = \cup_{0\le \beta < \alpha} \mathcal{T}_\beta'$, where for each $\beta < \alpha$, $\mathcal{T}_\beta'$ is a disjointly branching tree isomorphic to $\mathcal{T}_\beta$, chosen so that the sets $\cup \mathcal{T}_\beta'$, $\beta < \alpha$, are pairwise disjoint. For $r = 1,\dots,m$, let $G_r$ be the unique element of $\mathcal{T}_{\omega + k + r}'$ with cardinality $k + r$. Then $G_r$ has an infinite successor set in $\mathcal{T}_{\omega + k + r}'$, as follows from the structure of $\mathcal{T}_{\omega + k + r}$. Thus, for every $r$, we can find $g_r\in \mathbb{N}$ such that $g_r > \max(G_1\cup \dots\cup G_m)$ and $F_r := G_r\cup \{ g_r \}\in \mathcal{T}_{\omega + k + r}'$. Then $F_1,\dots,F_m$ is a $k$-yardstick family in $\mathcal{T}_\alpha$.
\end{proof}

\begin{figure}[t]
  \centering
\begin{tikzpicture}[
  vertex/.style={circle, draw, fill=none, inner sep=0pt, minimum size=1.8mm},
  marked/.style={vertex, fill=black},
  edge/.style={draw, thick},
  plain/.style={draw=none, fill=none, inner sep=0pt, minimum size=0pt},
  every label/.style={plain, font=\small},
]

\node[plain] at (2.2,0.7) {$\mathcal{T}'_{\omega+2}$};
\node[marked, label={[xshift=2pt]right:$0$}] (a0) at (2.2,0) {};
\node[marked, label={[xshift=2pt]above right:$1$}] (a1) at (2.2,-0.7) {};
\node[vertex, label={[xshift=2pt]below right:$5$}] (a2) at (0.7,-1.4) {};
\node[vertex, label={[xshift=2pt]right:$9$}] (a3) at (1.7,-1.4) {};
\node[vertex, label={[xshift=2pt]right:$10$}] (a4) at (1.7,-2.1) {};
\node[vertex, label={[xshift=2pt]right:$11$}] (a5) at (2.7,-1.4) {};
\node[vertex, label={[xshift=2pt]right:$12$}] (a6) at (2.7,-2.1) {};
\node[vertex, label={[xshift=2pt]right:$13$}] (a7) at (2.7,-2.8) {};
\node[marked, label={[xshift=2pt]right:$18$}] (a8) at (3.7,-1.4) {};
\node[vertex, label={[xshift=2pt]right:$19$}] (a9) at (3.7,-2.1) {};
\node[vertex, label={[xshift=2pt]right:$20$}] (a10) at (3.7,-2.8) {};
\node[vertex, label={[xshift=2pt]right:$21$}] (a11) at (3.7,-3.5) {};
\coordinate (adots) at (4.7,-1.4);
\draw[edge] (a0) -- (a1);
\draw[edge] (a1) -- (a2);
\draw[edge] (a1) -- (a3) -- (a4);
\draw[edge] (a1) -- (a5) -- (a6) -- (a7);
\draw[edge] (a1) -- (a8) -- (a9) -- (a10) -- (a11);
\draw[edge] (a1) -- (adots);
\node[plain] at (4.7,-1.8) {$\ddots$};

\begin{scope}[xshift=-7mm]
\node[plain] at (8.6,0.7) {$\mathcal{T}'_{\omega+3}$};
\node[marked, label={[xshift=2pt]right:$2$}] (b0) at (8.6,0) {};
\node[marked, label={[xshift=2pt]right:$3$}] (b1) at (8.6,-0.7) {};
\node[marked, label={[xshift=2pt]above right:$4$}] (b2) at (8.6,-1.4) {};
\node[vertex, label={[xshift=2pt]below right:$6$}] (b3) at (7.1,-2.1) {};
\node[vertex, label={[xshift=2pt]right:$7$}] (b4) at (8.1,-2.1) {};
\node[vertex, label={[xshift=2pt]right:$8$}] (b5) at (8.1,-2.8) {};
\node[marked, label={[xshift=2pt]right:$22$}] (b6) at (9.1,-2.1) {};
\node[vertex, label={[xshift=2pt]right:$23$}] (b7) at (9.1,-2.8) {};
\node[vertex, label={[xshift=2pt]right:$24$}] (b8) at (9.1,-3.5) {};
\node[vertex, label={[xshift=2pt]right:$28$}] (b9) at (10.1,-2.1) {};
\node[vertex, label={[xshift=2pt]right:$29$}] (b10) at (10.1,-2.8) {};
\node[vertex, label={[xshift=2pt]right:$30$}] (b11) at (10.1,-3.5) {};
\node[vertex, label={[xshift=2pt]right:$31$}] (b12) at (10.1,-4.2) {};
\coordinate (bdots) at (11.1,-2.1);
\draw[edge] (b0) -- (b1) -- (b2);
\draw[edge] (b2) -- (b3);
\draw[edge] (b2) -- (b4) -- (b5);
\draw[edge] (b2) -- (b6) -- (b7) -- (b8);
\draw[edge] (b2) -- (b9) -- (b10) -- (b11) -- (b12);
\draw[edge] (b2) -- (bdots);
\node[plain] at (11.1,-2.5) {$\ddots$};
\end{scope}

\begin{scope}[xshift=-14mm]
\node[plain] at (15,0.7) {$\mathcal{T}'_{\omega+4}$};
\node[marked, label={[xshift=2pt]right:$14$}] (c0) at (15,0) {};
\node[marked, label={[xshift=2pt]right:$15$}] (c1) at (15,-0.7) {};
\node[marked, label={[xshift=2pt]right:$16$}] (c2) at (15,-1.4) {};
\node[marked, label={[xshift=2pt]above right:$17$}] (c3) at (15,-2.1) {};
\node[marked, label={[xshift=2pt]below right:$25$}] (c4) at (13.5,-2.8) {};
\node[vertex, label={[xshift=2pt]right:$26$}] (c5) at (14.5,-2.8) {};
\node[vertex, label={[xshift=2pt]right:$27$}] (c6) at (14.5,-3.5) {};
\node[vertex, label={[xshift=2pt]right:$32$}] (c7) at (15.5,-2.8) {};
\node[vertex, label={[xshift=2pt]right:$33$}] (c8) at (15.5,-3.5) {};
\node[vertex, label={[xshift=2pt]right:$34$}] (c9) at (15.5,-4.2) {};
\node[vertex, label={[xshift=2pt]right:$35$}] (c10) at (16.5,-2.8) {};
\node[vertex, label={[xshift=2pt]right:$36$}] (c11) at (16.5,-3.5) {};
\node[vertex, label={[xshift=2pt]right:$37$}] (c12) at (16.5,-4.2) {};
\node[vertex, label={[xshift=2pt]right:$38$}] (c13) at (16.5,-4.9) {};
\coordinate (cdots) at (17.5,-2.8);
\draw[edge] (c0) -- (c1) -- (c2) -- (c3);
\draw[edge] (c3) -- (c4);
\draw[edge] (c3) -- (c5) -- (c6);
\draw[edge] (c3) -- (c7) -- (c8) -- (c9);
\draw[edge] (c3) -- (c10) -- (c11) -- (c12) -- (c13);
\draw[edge] (c3) -- (cdots);
\node[plain] at (17.5,-3.2) {$\ddots$};
\end{scope}
\end{tikzpicture}
  \caption{Constructing a $2$-yardstick family (labels of the black nodes) in a tree~$\mathcal{T}_\alpha$}
  \label{fig:yardstick-from-tree}
\end{figure}
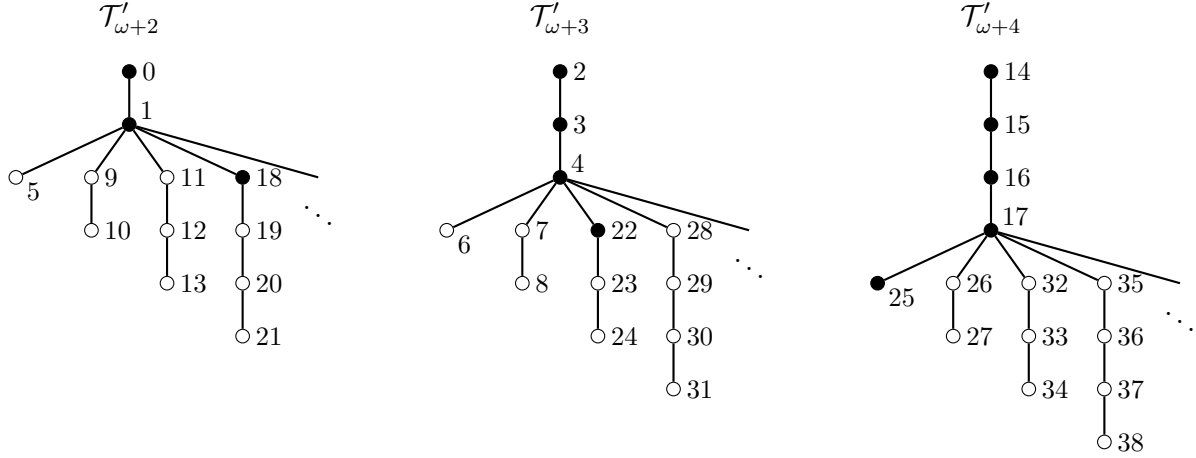

In general, to obtain an arbitrarily large yardstick family from a tree $\mathcal{T}$ of order at least $\omega\cdot 2$, we have to pass to the extension tree $\mathcal{T}_F$ of a suitable element $F\in \mathcal{T}$.

\begin{dfn}
Let $\mathcal{T}$ be a disjointly branching tree. For every $F\in\mathcal{T}$, we define
\[\mathcal{T}_F = \{G\in[\mathbb{N}_0]^{<\infty}: F\cap G = \emptyset \text{ and }F\cup G\in\mathcal{T}\}.\]
\end{dfn}

\begin{prop}\label{prop:existence-of-yardstick-families-in-trees}
  Let $\mathcal{T}$ be a well-founded disjointly branching tree such that $o[\mathcal{T}]\geq \omega\cdot 2$. Then, for every $k,m\in\mathbb{N}$, there exist $F\in\mathcal{T}$ and a $k$-yardstick family $F_1,\ldots,F_m$ in $\mathcal{T}_F$.
\end{prop}

\begin{proof}
We prove the claimed result in several steps.
\setcounter{clm}{0}
\begin{clm}
\label{computing by pruning}
For every $F\in\mathcal{T}$,
\[o[\mathcal{T}_F] = \max\{\alpha:F\in\mathcal{T}^{(\alpha)}\}.\]
In particular, $o[\mathcal{T}]\geq o[\mathcal{T}_F] + \#F$.
\end{clm}
\noindent \emph{Proof.} This is proved with an induction on $\mathcal{T}$ starting at the leaves. The second part is then obvious.\hfill$\scriptstyle\square$

\begin{clm}
For every $\alpha\leq o[\mathcal{T}]$ there exists $F\in\mathcal{T}$ such that $o[\mathcal{T}_F] = \alpha$.
\end{clm}

\noindent \emph{Proof.} One just picks a maximal $F\in\mathcal{T}^{(\alpha)}$.\hfill$\scriptstyle\square$

\begin{clm}
\label{infinitely many successors}
If $n\in\mathbb{N}_0$ and $o[\mathcal{T}] = \omega+n$, then there exists $F\in\mathcal{T}$ such that $\#F = n$ and the set $S_F = \{k>\max(F):F\cup\{k\}\in \mathcal{T}\}$ is infinite.
\end{clm}

\noindent \emph{Proof.} If $n = 0$, this means that there are infinitely many singletons in $\mathcal{T}$. If $\mathcal{T}$ contained only finitely many singletons, say, $\{ k_1 \},\ldots,\{ k_m \}$, then, letting $\mathcal{T}_i = \{ \emptyset \} \cup \{F\in\mathcal{T}:\min(F) = k_i\}$, we would have
\[o[\mathcal{T}] = \max_{1\leq i\leq m} o[\mathcal{T}_i] = \max_{1\leq i\leq m}o[\mathcal{T}_{\{k_i\}}]+1,\]
which is an infinite successor ordinal.

If $o[\mathcal{T}] = \omega+n$, then $o[\mathcal{T}^{(\omega)}] = n$, which yields a maximal element $F\in\mathcal{T}^{(\omega)}$ with $\#F = n$. In particular, by \Cref{computing by pruning}, $o[\mathcal{T}_F] =\omega$, which by the case $n=0$ yields that $S_F$ is infinite.\hfill$\scriptstyle\square$\medskip

Now, for every $n\in\mathbb{N}$, let  $G_n\in\mathcal{T}$ such that  $o[\mathcal{T}_{G_n}] = \omega+k_n$, with $k\leq k_1<k_2<\cdots$. We note that, in particular, $G_n\neq G_m$, for $n\neq m$. By the compactness of the power set $[\mathbb{N}_0]$ (identified, as usual, with $\{0,1\}^{\mathbb{N}_0}$, via $A\mapsto\chi_A$), and by passing to a subsequence, there exists $F\in[\mathbb{N}_0]$ such that
\begin{equation}\label{eq:19}
  G_n\cap\{0,\ldots,n\} = F\cap\{0,1,\ldots,n\} \qquad \text{for all } n\in \mathbb{N}.
\end{equation}
We deduce $F\cap\{0,1,\ldots,n\}\in\mathcal{T}$, and because $\mathcal{T}$ is well-founded, $F$ must be finite. By passing to a further subsequence, we may now assume that $G_n\cap G_m = F$, for all $n\neq m$. For all $n=1,\ldots,m$ pick $H_n\in\mathcal{T}_{G_n}$ such that $\#H_n = k_n\geq k$ and $S_{H_n}$ (with respect to $\mathcal{T}_{G_n}$) is infinite. Let $N = \max\{H_1\cup\cdots\cup H_m\}$ and, for $1\leq n\leq m$, choose $\ell_n\in S_{H_n}\cap\{N+1,\ldots\}$. Define $F_n = (G_n\setminus F)\cup H_n\cup\{\ell_n\}\in\mathcal{T}_F$, and note that the sets $F_1,\dots,F_m$ are pairwise disjoint.
\end{proof}

\begin{proof}[\normalfont\textbf{Proof of \Cref{thm:tree-spaces-uncomplemented}}]\phantom{}
  \begin{enumerate}[label=(\alph*)]
  \item[\ref{thm:tree-spaces-uncomplemented:a}]
  Let $Y$ be one of the spaces listed in the theorem, and let $\mathcal{T}$ be a well-founded disjointly branching tree with $o[\mathcal{T}] \ge \omega\cdot 2$. By \Cref{prop:existence-of-yardstick-families-in-trees}, for every $k\in \mathbb{N}$ and $m = 2^k$, we can find $F\in\mathcal{T}$ and a $k$-yardstick family $F_1,\ldots,F_m$ in $\mathcal{T}_F$. Now let $P_{\mathcal{T}}$ denote the orthogonal projection from $\langle\{  h_I : I\in \mathcal{D}  \}\rangle$ onto $\langle Y_G : G\in \mathcal{T} \rangle$. We claim that
  \begin{equation}\label{eq:21}
    \|P_\mathcal{T}\| \geq \Big\|\mathbb{E}_{\cup_{n=1}^m F_n}P_\mathcal{T}\Big\| = \Big\|\sum_{n=1}^m \mathbb{E}_{F_n}\Big\|.
  \end{equation}
  Indeed, the last equality follows from \Cref{expectations on subsets are a monoid}: Let $H = \cup _{n=1}^m F_n$, then for $g\in \langle \{ h_I : I\in \mathcal{D} \} \rangle$, we have
  \begin{equation*}
    \mathbb{E}_H P_\mathcal{T}g = \sum_{G\in \mathcal{T}\setminus \{ \emptyset \}} \mathbb{E}_H (\mathbb{E}_G - \mathbb{E}_{G'})(g) = \sum_{G\in \mathcal{T}\setminus \{ \emptyset \}} (\mathbb{E}_{H\cap G} - \mathbb{E}_{H\cap G'})(g).
  \end{equation*}
  Note that $H\cap G = H\cap G'$ unless $\max G \in H$, i.e., $\max G \in F_n$ for some $n\in \{ 1,\dots,m \}$. Thus, we obtain
  \begin{equation*}
    \mathbb{E}_H P_{\mathcal{T}}g = \sum_{n=1}^m \sum_{\substack{K\in \mathcal{T}_F\setminus \{ \emptyset \}\\ K\sqsubseteq F_n}} (\mathbb{E}_K - \mathbb{E}_{K'})(g) = \sum_{n=1}^m \mathbb{E}_{F_n}g,
  \end{equation*}
  which proves~\eqref{eq:21}.

  Next, observe that the right-hand side of~\eqref{eq:21} is the norm of the orthogonal projection from $\langle \{ h_I : I\in \mathcal{D} \} \rangle$ onto $Y_{F_1} + \dots + Y_{F_m}$, see~\eqref{eq:20}. By \Cref{cor:norm-P-tends-to-infty}, this norm tends to infinity as $k\to \infty$. Thus, $P_{\mathcal{T}}$ is unbounded and therefore does not extend to a bounded linear operator on $Y$. By \Cref{BRS orthocomplemented or not at all}, this implies that $R_{\mathcal{T}}^{Y}$ is uncomplemented in~$Y$. In particular, the conclusion holds for $\mathcal{T} = \mathcal{T}_\alpha$ if $\omega \cdot 2 \le \alpha < \omega_1$, since $o[\mathcal{T}_\alpha] = \alpha$ (see \Cref{rem:order-of-ordinal-trees}).

  \item[\ref{thm:tree-spaces-uncomplemented:b}]
  It is clear that the tree $\mathcal{T}_{\omega}$ can be constructed in such a way that for every $k,m \in \mathbb{N}$, there exists a $k$-yardstick family $F_1,\dots,F_m$ in $\mathcal{T}_\omega$ (this is evident from the structure of the tree and a suitable labeling of its nodes).
  Analogously, every tree $\mathcal{T}_{\omega + n}$, $n\in \mathbb{N}_0$, can be constructed so that for every $k,m \in \mathbb{N}$, there exist $F\in \mathcal{T}_{\omega + n}$ and a $k$-yardstick family $F_1,\dots,F_m$ in $(\mathcal{T}_{\omega + n})_F$.
Hence, as above, $R_{\omega + n}^Y$ is uncomplemented in $Y$.

\item[\ref{thm:tree-spaces-uncomplemented:c}]
          For $n = 0$, let $\mathcal{T}_\omega$ be constructed as in \Cref{exa:ordinal-trees}. Then, by a result of M\"uller and Schechtman~\cite{mueller:schechtman:1989}, $R_\omega^{H^1}$ is complemented in $H^1$ by a bounded linear projection~$P$. For $n = 1,2,\dots$, note that by iterating~\eqref{showing these are really BRS spaces eq1} in \Cref{showing these are really BRS spaces},
          \begin{equation*}
            R_{\mathcal{T}_{\omega + n}}^{H^1} = \langle \{ h_I : I\in \mathcal{D}_{n-1} \} \rangle + \sum_{I\in \mathcal{D}^n} T_I(R_{\mathcal{T}_\omega}^{H^1}).
          \end{equation*}
          For $y\in \langle \{ h_I : I\in \mathcal{D}_{n-1} \} \rangle$ and $x_I \in H^1$, $I\in \mathcal{D}^n$, put
          \begin{equation*}
            \widetilde{P}\Bigl( y + \sum_{I\in \mathcal{D}^n} T_I x_I \Bigr) = y + \sum_{I\in \mathcal{D}^n} T_I Px_I.
          \end{equation*}
          Then $\widetilde{P}$ is a bounded linear projection from~$H^1$ onto $R_{\mathcal{T}_{\omega + n}}^{H^1}$ with $\|\widetilde{P}\|\le 1 + 2\|P\|$.
\qedhere
\end{enumerate}
\end{proof}

\section{Open problems}

Recall that alternative~\ref{thm:alternatives:b} in our main result \Cref{thm:alternatives-ri} does indeed occur if $Y$ is isomorphic to $L_2$. Within the $L_p$ scale and the class of reflexive Orlicz spaces on $[0,1]$, these are the only examples for which this alternative holds. This naturally leads to the following question.
\begin{problem}
  Are there any non-Hilbertian Haar system spaces for which alternative~\ref{thm:alternatives:b} in \Cref{thm:alternatives-ri} holds?
\end{problem}

In the case of $L_p$, $1 < p < \infty$, $p \ne 2$, Alspach~\cite{alspach:1999} further analyzed the structure of the Bourgain-Rosenthal-Schechtman spaces. In particular, he proved that all the \emph{limit} Bourgain-Rosenthal-Schechtman spaces are pairwise non-isomorphic. It is therefore natural to ask whether this also holds in the general case, thus yielding a stronger conclusion than~\ref{thm:alternatives:a} in \Cref{thm:alternatives-ri}.
\begin{problem}
  If alternative~\ref{thm:alternatives:a} in \Cref{thm:alternatives-ri} or \Cref{thm:alternatives-hardy} holds, are all spaces $R_\alpha^Y$, where $\alpha$ ranges over the countable limit ordinals, pairwise non-isomorphic?
\end{problem}

Our proof of the complementation of the Bourgain-Rosenthal-Schechtman spaces in~$Y$ (see \Cref{prop:complementation}) relies on the assumption that the underlying Haar system space~$X$ has non-trivial Boyd indices, which is equivalent to the Haar system being unconditional in~$X$. Moreover, both the complementation and unconditionality are used in the proof of our main result, \Cref{thm:alternatives-ri}.
In view of the counterexamples in \Cref{thm:tree-spaces-uncomplemented}, where complementation fails, it is thus natural to ask the following questions.
\begin{problem}
  Is the complementation of the Bourgain-Rosenthal-Schechtman spaces in~$X^0$ or~$H_X$ equivalent to $X$ having non-trivial Boyd indices?
\end{problem}
\begin{problem}
  Does the dichotomy in \Cref{thm:alternatives-ri} still hold if one of the Boyd indices of~$X$ is trivial?
\end{problem}
One step towards answering the above problem is to prove a version of the result on unconditional decompositions (\Cref{thm:1.1}) for spaces in which the Haar system is conditional.  As mentioned in \Cref{sec:complemented-embeddings}, the work of Enflo and Starbird~\cite{enflo:starbird:1979} yields precisely this result in~$L_1$, where conditionality forces only alternative~\ref{thm:1.1:a} to occur. 
\begin{problem}
  For which Haar system spaces with at least one trivial Boyd index does the conclusion of \Cref{thm:1.1} hold?
\end{problem}
For such spaces, we see that alternative~\ref{thm:alternatives:c} in \Cref{thm:alternatives-ri} is impossible, by the minimality argument used in the proof of~\ref{thm:alternatives:c}$\Longrightarrow$\ref{thm:alternatives:b}. Assuming the Fatou property, the proof of~\ref{thm:alternatives:a}$\vee $\ref{thm:alternatives:b} then yields uncountably many pairwise non-isomorphic Bourgain--Rosenthal--Schechtman spaces (even though they may be uncomplemented).

Of course, we may ask analogous questions in the Hardy case, in particular:
\begin{problem}
  Does the dichotomy in \Cref{thm:alternatives-hardy} still hold without assuming complementation of the Bourgain-Rosenthal-Schechtman tree spaces in~$H_X$?
\end{problem}

We conclude with the following question about factorization properties in Bourgain-Rosenthal-Schechtman spaces.

\begin{problem}
  Does the factorization property, proved in \cite{konstantos:motakis:brs} for the standard bases of the Bourgain-Rosenthal-Schechtman spaces in $L_p$, $1 < p < \infty$, also hold if the underlying space is an arbitrary Haar system space or associated Hardy space? Do the spaces $R_\alpha^Y$, $\omega \le \alpha < \omega_1$, satisfy the primary factorization property?
\end{problem}

\noindent\textbf{Acknowledgement.}
We would like to thank Richard Lechner and Paul~F.X. M\"uller for useful discussions and especially for drawing our attention to \cite[Remark~4, p.~223]{bourgain:rosenthal:schechtman:1981}, which inspired us to develop the material in \Cref{sec:unboundedness}.

\bibliographystyle{abbrv}
\bibliography{bibliography}
\end{document}